\documentclass[12pt, reqno]{amsart}
\usepackage{amsmath, amsthm, amssymb}
\usepackage{enumitem}

\usepackage{fullpage}
\usepackage[many]{tcolorbox}
\usepackage{xcolor}
\usepackage{wasysym}
\usepackage{mathtools}
\usepackage{textcomp}
\usepackage{stmaryrd}
\usepackage{pdflscape}
\usepackage{microtype}

\usepackage{tikz}
\usetikzlibrary{cd}
\usetikzlibrary{calc}
\usetikzlibrary{arrows,backgrounds,patterns.meta,arrows.meta}
\usetikzlibrary{positioning,shadings}
\usetikzlibrary{shapes}
\usetikzlibrary{backgrounds}
\usetikzlibrary{decorations,decorations.pathreplacing,decorations.markings,decorations.pathmorphing}
\tikzstyle{snake}=[decorate, decoration={snake, segment length=1mm, amplitude=.3mm}]
\tikzstyle{saw}=[decorate, decoration={saw, segment length=.7mm, amplitude=.25mm}]
\newcommand{\tikzmath}[2][]
{\vcenter{\hbox{\begin{tikzpicture}[#1]#2\end{tikzpicture}}}
}
\newcommand{\roundNbox}[6]{
	\draw[rounded corners=5pt, very thick, #1] ($#2+(-#3,-#3)+(-#4,0)$) rectangle ($#2+(#3,#3)+(#5,0)$);
	\coordinate (ZZa) at ($#2+(-#4,0)$);
	\coordinate (ZZb) at ($#2+(#5,0)$);
	\node at ($1/2*(ZZa)+1/2*(ZZb)$) {#6};
}
\tikzset{super thick/.style={line width=3pt}}
\tikzstyle{mid>}=[decoration={markings, mark=at position 0.55 with {\arrow{>}}}, postaction={decorate}]
\tikzstyle{mid<}=[decoration={markings, mark=at position 0.55 with {\arrow{<}}}, postaction={decorate}]

\tikzstyle{frameR}=[preaction={draw=#1,super thick,opacity=.2,decorate,decoration={curveto,amplitude=0,raise=-1.5pt}},thick,#1]
\tikzstyle{frameL}=[preaction={draw=#1,super thick,opacity=.2,decorate,decoration={curveto,amplitude=0,raise=1.5pt}},thick,#1]

\def\semicolon{;}
\def\applytolist#1{
    \expandafter\def\csname multi#1\endcsname##1{
        \def\multiack{##1}\ifx\multiack\semicolon
            \def\next{\relax}
        \else
            \csname #1\endcsname{##1}
            \def\next{\csname multi#1\endcsname}
        \fi
        \next}
    \csname multi#1\endcsname}

\def\calc#1{\expandafter\def\csname c#1\endcsname{{\mathcal #1}}}
\applytolist{calc}QWERTYUIOPLKJHGFDSAZXCVBNM;
\def\bbc#1{\expandafter\def\csname bb#1\endcsname{{\mathbb #1}}}
\applytolist{bbc}QWERTYUIOPLKJHGFDSAZXCVBNM;
\def\bfc#1{\expandafter\def\csname bf#1\endcsname{{\mathbf #1}}}
\applytolist{bfc}QWERTYUIOPLKJHGFDSAZXCVBNM;
\def\sfc#1{\expandafter\def\csname s#1\endcsname{{\sf #1}}}
\applytolist{sfc}QWERTYUIOPLKJHGFDSAZXCVBNM;
\def\fc#1{\expandafter\def\csname f#1\endcsname{{\mathfrak #1}}}
\applytolist{fc}QWERTYUIOPLKJHGFDSAZXCVBNM;
\def\rmc#1{\expandafter\def\csname rm#1\endcsname{{\mathrm #1}}}
\applytolist{rmc}QWERTYUIOPLKJHGFDSAZXCVBNM;

\tikzstyle{shaded}=[fill=red!10!blue!20!gray!30!white]
\tikzstyle{unshaded}=[fill=white]
\tikzstyle{over}=[double, draw=white, super thick, double=]
\tikzstyle{snake}=[decorate, decoration={snake, segment length=1mm, amplitude=.3mm}]
\tikzstyle{saw}=[decorate, decoration={saw, segment length=.7mm, amplitude=.25mm}]

\tikzstyle{coupon}=[draw, very thick, rectangle, rounded corners=5pt]
\tikzset{Rightarrow/.style={double equal sign distance,>={Implies},->},
triplecd/.style={-,preaction={draw,Rightarrow}},
quadruplecd/.style={preaction={draw,Rightarrow,
shorten >=0pt
},
shorten >=1pt,
-,double,double
distance=0.2pt}}
\tikzset{
    tripleline/.style args={[#1] in [#2] in [#3]}{
        #1,preaction={preaction={draw,#3},draw,#2}
    }
}
\tikzstyle{triple}=[tripleline={[line width=.15mm,black] in
      [line width=.7mm,white] in
      [line width=1mm,black]}] 
\tikzset{
    quadrupleline/.style args={[#1] in [#2] in [#3] in [#4]}{
        #1,preaction={preaction={preaction={draw,#4},draw,#3}, draw,#2}
    }
}
\tikzstyle{quadruple}=[quadrupleline={[line width=.3mm,white] in
      [line width=.6mm,black] in
      [line width=1.2mm,white] in
      [line width=1.5mm,black]}]

\tikzdeclarepattern{
  name=primeddots,
  type=uncolored,
  bounding box={(-.6pt,-.6pt) and (.6pt,.6pt)},
  tile size={(5pt,5pt)},
  tile transformation={rotate=60},
  parameters={none}, 
  code={
    \fill(0pt,0pt) circle (.35pt);
  }
}

\tikzdeclarepattern{
  name=primedbox,
  type=uncolored,
  bounding box={(-1pt,-1pt) and (1pt,1pt)},
  tile size={(5pt,5pt)},
  tile transformation={rotate=60},
  code={
    \draw[very thin] (-0.9pt,-0.9pt) -- (-0.9pt,0.9pt) -- (0.9pt,0.9pt) -- (0.9pt,-0.9pt) -- (-0.9pt,-0.9pt);
  }
}

\tikzdeclarepattern{
  name=primedplus,
  type=uncolored,
  bounding box={(-1pt,-1pt) and (1pt,1pt)},
  tile size={(5pt,5pt)},
  tile transformation={rotate=30},
  parameters={none}, 
  code={
    \draw[very thin] (-.9pt,-.9pt) -- (.9pt,.9pt);
    \draw[very thin] (.9pt,-.9pt) -- (-.9pt,.9pt);
  }
}

\tikzdeclarepattern{
  name=primedstar,
  type=uncolored,
  bounding box={(-1.2pt,-1.2pt) and (1.2pt,1.2pt)},
  tile size={(5pt,5pt)},
  tile transformation={rotate=30},
  parameters={none}, 
  code={
    \draw[very thin] (-1.2pt,0pt) -- (1.2pt,0pt);
    \draw[very thin] (-.6pt,-1.04pt) -- (.6pt,1.04pt);
    \draw[very thin] (-.6pt,1.04pt) -- (.6pt,-1.04pt);
  }
}

\tikzdeclarepattern{
  name=primeddots2,
  type=uncolored,
  bounding box={(-.3pt,-.3pt) and (.3pt,.3pt)},
  tile size={(2.5pt,2.5pt)},
  tile transformation={rotate=60},
  parameters={none}, 
  code={
    \fill(0pt,0pt) circle (.35pt);
  }
}

\tikzstyle{primedregion}[none]=[
	preaction={fill=#1},
	pattern=primeddots,
  draw=#1,
]

\tikzstyle{boxregion}[none]=[
	preaction={fill=#1},
	pattern=primedbox,
  draw=#1,
]

\tikzstyle{plusregion}[none]=[
	preaction={fill=#1},
	pattern=primedplus,
  draw=#1,
]

\tikzstyle{starregion}[none]=[
	preaction={fill=#1},
	pattern=primedstar,
  draw=#1,
]

\tikzstyle{primedregion2}[none]=[
	preaction={fill=#1},
	pattern=primeddots2,
  draw=#1,
]

\newcommand{\Bim}{\mathsf{Bim}}

\newcommand{\ZColor}{violet}

\newcommand{\AsColor}{red}

\newcommand{\AColor}{lightgray!50}
\newcommand{\BColor}{lightgray}

\newcommand{\Asplitcolor}{red!10}
\newcommand{\XsColor}{red!80!black}
\newcommand{\YsColor}{blue}
\newcommand{\Bsplitcolor}{blue!15}

\newcommand{\acol}{\AColor}
\newcommand{\bcol}{\Asplitcolor}
\newcommand{\ccol}{blue!10}
\newcommand{\Acol}{\AsColor}
\newcommand{\Xcol}{\XsColor}
\newcommand{\Ycol}{blue}
\newcommand{\Zcol}{\ZColor}
\newcommand{\Wcol}{orange}

\newcommand{\splitcat}{\mathsf{split}}

\makeatletter
\newcommand{\xMapsto}[2][]{\ext@arrow 0599{\Mapstofill@}{#1}{#2}}
\def\Mapstofill@{\arrowfill@{\Mapstochar\Relbar}\Relbar\Rightarrow}
\makeatother

\definecolor{violet}{RGB}{148,0,211}
\definecolor{DarkGreen}{RGB}{0,150,0}

\usepackage[pdftex,plainpages=false,hypertexnames=false,pdfpagelabels]{hyperref}
\definecolor{medium-blue}{rgb}{0,0,.8}
\hypersetup{colorlinks, linkcolor={purple}, citecolor={medium-blue}, urlcolor={medium-blue}}
\newcommand{\arxiv}[1]{\href{http://arxiv.org/abs/#1}{\tt arXiv:\nolinkurl{#1}}}
\newcommand{\arXiv}[1]{\href{http://arxiv.org/abs/#1}{\tt arXiv:\nolinkurl{#1}}}

\newcommand{\doi}[1]{\href{http://dx.doi.org/#1}{{\tt DOI:#1}}}

\DeclareMathOperator{\coev}{coev}
\DeclareMathOperator{\End}{End}
\DeclareMathOperator{\ev}{ev}

\DeclareMathOperator{\FPdim}{FPdim}
\DeclareMathOperator{\Hom}{Hom}
\DeclareMathOperator{\id}{id}

\DeclareMathOperator{\Irr}{\pi_0}
\DeclareMathOperator{\mate}{mate}

\newcommand{\op}{\mathrm{op}}

\DeclareMathOperator{\Tr}{Tr}
\DeclareMathOperator{\tr}{tr}

\DeclareFontFamily{U}{dmjhira}{}\DeclareFontShape{U}{dmjhira}{m}{n}{ <-> dmjhira }{}\DeclareRobustCommand{\yo}{\kern-0.35ex\text{\usefont{U}{dmjhira}{m}{n}\symbol{"48}}} 

\newcommand{\set}[2]{\left\{#1 \middle| #2\right\}}
\newcommand{\bra}[1]{\langle#1|}
\newcommand{\ket}[1]{|#1\rangle}
\renewcommand{\cent}{\textup{\textcent}}

\newcommand{\isometry}{\mathrel{\,\joinrel\lhook\joinrel\rightharpoonup}}
\newcommand{\fullyfaithful}{\mathrel{\,\joinrel\lhook\joinrel\rightharpoondown}}

\newcommand{\Mod}{\mathsf{Mod}}
\newcommand{\Fun}{\mathsf{Fun}}
\newcommand{\Vect}{\mathsf{Vec}}
\newcommand{\Hilb}{\mathsf{Hilb}}

\newcommand{\HstarAlg}{\mathsf{H^*Alg}}

\def\semicolon{;}
\def\applytolist#1{
    \expandafter\def\csname multi#1\endcsname##1{
        \def\multiack{##1}\ifx\multiack\semicolon
            \def\next{\relax}
        \else
            \csname #1\endcsname{##1}
            \def\next{\csname multi#1\endcsname}
        \fi
        \next}
    \csname multi#1\endcsname}

\def\calc#1{\expandafter\def\csname c#1\endcsname{{\mathcal #1}}}
\applytolist{calc}QWERTYUIOPLKJHGFDSAZXCVBNM;
\def\bbc#1{\expandafter\def\csname bb#1\endcsname{{\mathbb #1}}}
\applytolist{bbc}QWERTYUIOPLKJHGFDSAZXCVBNM;
\def\bfc#1{\expandafter\def\csname bf#1\endcsname{{\mathbf #1}}}
\applytolist{bfc}QWERTYUIOPLKJHGFDSAZXCVBNM;
\def\sfc#1{\expandafter\def\csname s#1\endcsname{{\sf #1}}}
\applytolist{sfc}QWERTYUIOPLKJHGFDSAZXCVBNM;
\def\fc#1{\expandafter\def\csname f#1\endcsname{{\mathfrak #1}}}
\applytolist{fc}QWERTYUIOPLKJHGFDSAZXCVBNM;
\def\rmc#1{\expandafter\def\csname rm#1\endcsname{{\mathrm #1}}}
\applytolist{rmc}QWERTYUIOPLKJHGFDSAZXCVBNM;

\theoremstyle{plain}
\newtheorem{thm}{Theorem}[section]
\newtheorem*{thm*}{Theorem}
\newtheorem{thmalpha}{Theorem}

\newtheorem{cor}[thm]{Corollary}

\newtheorem*{cor*}{Corollary}

\newtheorem*{conj*}{Conjecture}
\newtheorem{lem}[thm]{Lemma}
\newtheorem*{lem*}{Lemma}

\newtheorem{prop}[thm]{Proposition}

\newtheorem*{quest*}{Question}
\newtheorem*{claim*}{Claim}

\theoremstyle{definition}
\newtheorem{defn}[thm]{Definition}
\newtheorem{defnalpha}[thmalpha]{Definition}
\newtheorem{fact}[thm]{Fact}

\newtheorem{construction}[thm]{Construction}

\newtheorem{nota}[thm]{Notation}

\newtheorem{ex}[thm]{Example}
\newtheorem{sub-ex}[thm]{Sub-Example}
\newtheorem{counter-ex}[thm]{Counter-Example}
\newtheorem{rem}[thm]{Remark}
\newtheorem*{rem*}{Remark}

\AtBeginEnvironment{tikzcd}{
  \tikzset{
    every node/.append style={
        font=\normalfont
      }
  }
}
\usepackage{quiver}

\usetikzlibrary{arrows.meta}
\usetikzlibrary{fit}

\newcommand{\coloneq}{:=}
\def\coloneq{\mathrel{\vcenter{\baselineskip0.55ex\lineskiplimit0pt\hbox{\upshape.}\hbox{\upshape.}}}=} 
\def\eqcolon{\mathrel{\scalebox{-1}[1]{$\coloneq$}}}
\let\coloneqq\coloneq

\DeclareMathOperator{\isomeq}{\mathbin{\,\cong^\dag}}

\newcommand{\dtimesaux}[4]{\mathbin{\vcenter{\hbox{\raisebox{#4}{\scalebox{#3}{\ooalign{\hfil$\vcenter{\hbox{$#1\times$}}$\hfil\cr\hfil$\vcenter{\hbox{\mbox{#2\upshape\wasylozenge}}}$\hfil\cr}}}}}}}

\newcommand{\dtimes}{%
  \mathchoice
    {\dtimesaux{\displaystyle}{\normalsize}{1.15}{0.35ex}}
    {\dtimesaux{\textstyle}{\normalsize}{1.15}{0.35ex}}
    {\dtimesaux{\scriptstyle}{\scriptsize}{1.0}{0ex}}
    {\dtimesaux{\scriptscriptstyle}{\tiny}{1.0}{0ex}}%
}

\newcommand{\threehilbtimes}{\dtimes}

\makeatletter
\newcommand{\bigboxplus}{%
  \DOTSB\mathop{\mathpalette\mattos@bigplus\relax}\slimits@
}
\newcommand\mattos@bigplus[2]{%
  \vcenter{\hbox{%
    \sbox\z@{$#1\sum$}%
    \resizebox{!}{0.9\dimexpr\ht\z@+\dp\z@}{\raisebox{\depth}{$\m@th#1\boxplus$}}%
  }}%
  \vphantom{\sum}%
}
\makeatother

\title{Orthonormal bases for higher Hilbert spaces}
\author{Giovanni Ferrer, Brett Hungar, David Penneys, and Greyson Wesley}
\date{\today}

\begin{document}

\begin{abstract}
In our previous article [\arxiv{2410.05120}], we introduced the notion of a finite dimensional 3-Hilbert space, categorifying Baez's 2-Hilbert spaces.
In this article, by further categorifying Baez's higher linear algebra, we provide useful tools for working with 3-Hilbert spaces, including, generalized scalar multiplication, orthonormal bases, and unitary adjoints for operators.
We use these tools to endow the $\mathrm{C}^*$-3-category of 3-Hilbert spaces with a self-enrichment.
We prove a Unitary Yoneda Lemma/Riesz Representation Theorem for 3-Hilbert spaces: the Yoneda embedding is an isometric equivalence.
Finally, we define a unitary version of the Deligne product 
on 3-Hilbert spaces and prove that it satisfies an isometric version of the folding trick.
\end{abstract}

\maketitle

\tableofcontents

\newpage

\begin{landscape}
\thispagestyle{empty}    
\begin{table}
\hspace*{-1.5cm}
\begin{tabular}{|c||c|c|c|}
\hline
& $\Hilb$ & $2\Hilb$ & $3\Hilb$
\\\hline\hline
length
&
$\|\xi\|$
&
$d_c\coloneq\Tr^\cC(\id_c)$
&
$d_x\coloneq \Psi^\fX_x(\id_{1_x})$
\\\hline
inner product
&
$\langle \eta|\xi\rangle\in \bbC$
&
$\langle a|b\rangle_\Hilb^\cC\coloneq \cC(a\to b)\in \Hilb$
&
$\langle x|y\rangle^\fX_{2\Hilb}\coloneq \fX(x\to y)\in 2\Hilb$
\\\hline
simple object
&
$\xi\in H$ with $\|\xi\|=1$
&
$c\in \cC$ 
with $\Omega_c\in \Hilb$ 1-dimensional
& 
$x\in \fX$ 
with $\Omega_x\in 2\Hilb$ an $\rmH^*$-fusion cat.
\\\hline
ONB
&
$\{e_i\}$
spanning, $\langle e_i|e_j\rangle=\delta_{i=j}$
&
$\pi_0\cC$
set of representatives of simples
&
$\pi_0\fX$
one simple per component
\\\hline
addition
&
sum $\eta+\xi$
&
orthogonal direct sum $a\oplus b$
&
Hilbert direct sum $x\boxplus y$
\\\hline
scalar action
&
$\lambda\cdot \xi$ for $\lambda\in\bbC$
&
$H\odot c$
for $H\in \Hilb$
&
$\cM\boxdot x$
for $\cM\in 2\Hilb$
\\\hline
relative scalar action
&
$\lambda\cdot \|\xi\|^{-2} \cdot\xi$ for $\lambda\in \bbC$
&
$H\odot_{\Omega_c} c$
for $H\in \Mod^\dag(\Omega_c,\Tr^\cC_c)$
&
$\cM\boxdot_{\Omega_x} x$
for $\cM\in \Mod^\dag(\Omega_x)$
\\\hline
linear combination
&
$\sum_i \lambda_i e_i$
&
$\bigoplus_{b\in \pi_0\cC} H_b\odot_{\Omega_b} b$
&
$\bigboxplus_{b\in \pi_0\fX} \cM_b \boxdot_{\Omega_b} b$
\\\hline
Fourier expansion
&
$\xi = \sum_i \langle e_i|\xi\rangle e_i$
&
$c=\bigoplus_{b\in \pi_0\cC} \langle b|c\rangle \odot_{\Omega_b} b$
&
$x=\bigboxplus_{b\in \pi_0\fX} \langle b|x\rangle \boxdot_{\Omega_b} b$
\\\hline
linear map
&
$T\colon H\to K$
&
$\dag$-functor $F\colon  \cA\to \cB$
&
$\dag,\vee$-preserving 2-functor $F\colon \fX\to \fY$
\\\hline
coordinates
&
$H \cong \bigoplus_{e_i:\text{ONB}} \bbC \cdot e_i$
&
$\cC \cong \bigoplus_{c \in \pi_0\cC} \Mod^\dag(\Omega_c)$
&
$\fX \cong \bigboxplus_{x \in \pi_0\fX} \Mod^\dag(\Omega_x)$
\\\hline
adjoint
&
$\langle T\eta |\xi\rangle_K=\langle \eta| T^*\xi\rangle_H$
&
$\cB(F(a)\to b)\isomeq \cA(a\to F^*(b))$
&
$\fY(F(x)\to y)\isomeq \fX(x\to F^*(y))$
\\\hline
length preserving $\isometry$
&
$\|\cdot\|$-preserving 
$T\colon  H\to K$
&
$\Tr$-preserving $F\colon \cA\to \cB$
&
$\Psi$-preserving 
$F\colon \fX\to \fY$ 
\\\hline
isometry $\hookrightarrow$
&
$\langle T \eta | T \xi \rangle_K = \langle \eta | \xi \rangle_H$ 
&
$\cB(F(a) \to F(a')) \isomeq \cA(a \to a')$
&
$\fY(F(x) \to F(x')) \isomeq \fX(x \to x')$
\\\hline
unitary
&
$T^* = T^{-1}$
(invertible isometry)
&
isometric equivalence
&
isometric equivalence
\\\hline
linear functional
&
$T \colon H \to \bbC$
&
$\dag$-functor $F \colon \cA \to \Hilb$
&
$\dag$,$\vee$-preserving 2-functor $F \colon \fX \to 2\Hilb$
\\\hline
dual space
&
$H^\vee \coloneq \Hom(H,\bbC)$
&
$\Fun^\dag(\cA,\Hilb)$
&
$\Fun^{\dag,\vee}(\fX,2\Hilb)$
\\\hline
Riesz Representation
&
$\overline{H} \cong H^\vee$
&
$\cA^\op \isomeq \Fun^\dag(\cA,\Hilb)$
&
$\fX^\op \isomeq \Fun^{\dag,\vee}(\fX,2\Hilb)$
\\\hline
ket operators
&
$|\xi \rangle \colon \bbC \to H$ where $1_\bbC \mapsto \xi$
&
$|c\rangle \colon \Hilb \to \cC$ where $\bbC \mapsto c$
&
$|x\rangle \colon 2\Hilb \to \fX$ where $\Hilb \mapsto x$
\\\hline
relative ket operators
&
\textquotesingle\textquotesingle
&
$|c\rangle_\Omega \colon \Mod^\dag(\Omega_c) \to \cC$ where $\Omega_c \mapsto c$
&
$|x\rangle_\Omega \colon \Mod^\dag(\Omega_x) \to \fX$ where $\Omega_x \mapsto x$
\\\hline
bra operators
&
$\langle \eta | \coloneqq |\eta\rangle^* \colon H \to \bbC$
&
$\langle c | \coloneqq |c\rangle^* \colon \cC \to \Hilb$
&
$\langle x | \coloneqq |x\rangle^* \colon \fX \to 2\Hilb$
\\\hline
relative bra operators
&
\textquotesingle\textquotesingle
&
${}_\Omega\langle c | \coloneqq |c\rangle_\Omega^* \colon \cC \to \Mod^\dag(\Omega_c)$
&
${}_\Omega\langle x | \coloneqq |x\rangle_\Omega^* \colon \fX \to \Mod^\dag(\Omega_x)$
\\\hline
resolution of $\id$
&
$\id_H = \sum_{e_i : \text{ONB}} |e_i \rangle \langle e_i |$
&
$\id_{\cC} = \bigoplus_{b \in \pi_0\cC} |b \rangle_\Omega \langle b |$
&
$\id_{\fX} = \bigboxplus_{b \in \pi_0\fX} |b \rangle_\Omega \langle b |$
\\\hline
monoidal product
&
tensor product $\otimes$
&
unitary Deligne product $\boxtimes$
&
unitary Deligne product $\threehilbtimes$
\\\hline
monoidal unit
&
$\bbC$
&
$\Hilb$
&
$2\Hilb$
\\\hline
$E_1$-algebras
&
$\rmH^*$-algebras
&
$\rmH^*$-multifusion categories
&
???
\\\hline
\end{tabular}
\vspace*{1cm}
\caption{\label{fig:Categorify}Categorification of linear algebra: finite dimensional Hilbert spaces and operators}
\end{table}
\end{landscape}

\newpage
\section{Introduction}

Higher categories of higher vector spaces were introduced as receptacles for fully extended topological quantum field theories via the cobordism hypothesis \cite{MR1355899}.
Fusion 2-categories and semisimple 2-categories were introduced in \cite{1812.11933}, and a formal construction of $n\Vect$ via higher idempotent completion was implicitly outlined in \cite{1905.09566}.

Unitary theories are desired for applications to topologically ordered phases of matter in theoretical condensed matter physics \cite{MR1951039,cond-mat/0506438,PhysRevB.71.045110,MR1910833,NaaijkensThesis,MR4268163,MR4362722,MR4444089,MR4945955}.
The 2-category of finite dimensional 2-Hilbert spaces was introduced in \cite{MR1448713} (see \cite{2411.01678} for work on complete $\rmW^*$-categories, which behave like infinite dimensional 2-Hilbert spaces), and the 3-category of finite dimensional 3-Hilbert spaces was studied in detail in our previous article \cite{MR5078555}. 

In this article, we provide many important tools for working with finite dimensional higher Hilbert spaces.
(Some of our results have similar or analogous formulations in the non-unitary setting; see Remark \ref{rem:NonUnitary} below.)
Table \ref{fig:Categorify}, which appeared before the introduction, gives a detailed account of how various constructions in higher Hilbert spaces, like scalar multiplication, linear combinations, and orthonormal bases, change as the category level increases. 

Recall form \cite{MR1448713} that a 2-Hilbert space is a finitely semisimple $\rmC^*$-category $\cC$ equipped with a family of positive faithful traces $\{\Tr^\cC_c: \cC(c\to c)\to \bbC\}_{c\in \cC}$ satisfying
$$
\Tr^\cC_a(gf) = \Tr^\cC_b(fg)
\qquad\qquad\qquad
\forall\, 
f:a\to b
\text{ and } 
g:b\to a.
$$
These traces are equivalent data to enriching $\cC$ over $\Hilb$, the $\rmC^*$-category of finite dimensional Hilbert spaces, which consists of a positive definite inner product $\langle-|-\rangle_{a\to b}$ on each hom space $\cC(a\to b)$, satisfying
\begin{equation}
\label{eq:2HilbsAreHilbEnriched}
\langle gf|h\rangle_{a\to c}
=
\langle g|hf^\dag\rangle_{b\to c}
=
\langle f|g^\dag h\rangle_{a\to b}
\qquad\qquad
\forall\, 
f:a\to b,
g:b\to c,
\text{ and } 
h:a\to c.
\end{equation}
These equalities give a Unitary Yoneda Lemma for 2-Hilbert spaces \cite[\S2.1]{MR4750417} (based on \cite[Rem.~3.61 and footnote]{MR4598730} and \cite{MR3687214,MR4079745}):
each $\cC(-\to c): \cC^{\op}\to \Hilb$ is a dagger functor, and the Yoneda embedding $\cC\hookrightarrow \Fun^\dag(\cC^{\op}\to \Hilb)$ is a dagger equivalence.
We include a background section on 2-Hilbert spaces in \S\ref{sec:backgroundon2Hilbs} below.

Recall from \cite{MR5078555} that a 3-Hilbert space is a finite semisimple $\rmC^*$-2-category $\fX$ equipped with
\begin{itemize}
\item 
a unitary adjoint functor (UAF) $\vee$, which provides unitary duals for 1-morphisms (see Definition \ref{defn:UAF} below), and
\item 
a spherical weight $\Psi$, which is a family of positive linear functionals $\Psi_x\colon \End_\fX(1_x)\to \bbC$ for all $x\in \fX$  (see Definition \ref{3hilbertspace} below)
satisfying the following compatibility condition with the UAF:
$$
\Psi^\fX_a
\left(
\tikzmath{
\fill[fill=lightgray!50,rounded corners=5] (-.65,-.7) rectangle (1.2,.7);
\draw[thick,fill=lightgray] (-.15,0) arc (180:360+180:.5);
\node at (-.15,-.5) {$\scriptstyle X$};
\node at (.9,-.5) {$\scriptstyle X^\vee$};
\roundNbox{fill=white}{(0,0)}{.3}{0}{0}{$f$};
}
\right)
=
\Psi^\fX_b
\left(
\tikzmath{
\fill[fill=lightgray,rounded corners=5] (-.45,-.7) rectangle (1.5,.7);
\draw[thick,fill=lightgray!50] (-.05,0) arc (180:360+180:.5);
\node at (.95,-.5) {$\scriptstyle X$};
\node at (-.15,-.5) {$\scriptstyle X^\vee$};
\roundNbox{fill=white}{(0.8,0)}{.3}{0}{0}{$f$};
}
\right)
\qquad\qquad\qquad
\forall f\in \End_\fX({}_aX_b).
$$
\end{itemize}
The above formula defines a canonical positive faithful trace $\Tr^{\fX(a\to b)}_X$ on $\End_\fX({}_a X_b)$, equipping $\fX$ with a $2\Hilb$ enrichment.
We also insist that a 3-Hilbert space is \emph{complete}, i.e., admits Hilbert direct sums (see Definition \ref{defn:HilbertDirectSum} below) and all $\rmH^*$-monads split (see Definition \ref{defn:H*monad} below).

We include a background section on 3-Hilbert spaces in \S\ref{sec:backgrounon3Hilbs} below, where we actually prove some new foundational results.
Of particular importance is Lemma \ref{lem:uaf-reps-unitary-adjs}, which proves that the compatibility of the UAF with above $2\Hilb$-enrichment categorifies \eqref{eq:2HilbsAreHilbEnriched}:
$$
\fX({}_bY^\vee \otimes_a X_c \Rightarrow {}_bZ_c) \isomeq
\fX({}_aX_c \Rightarrow {}_aY\otimes_b Z_c) \isomeq
\fX({}_a X \otimes_c Z^\vee_b \Rightarrow {}_aY_b).
$$
In Corollary \ref{cor:isometric-yoneda}, we prove that the functor
\begin{equation}
\label{eq:BraFuUnctors}
\langle a| :=\fX(-\to a) : \fX^{1\op}\to 2\Hilb
\end{equation}
is uniquely representable up to unique isometric equivalence in $\fX$, which is a stronger notion of adjoint equivalence in $\fX$ with respect to the fixed UAF.

In \S\ref{subsec:coefficients}, we introduce a new notion of scalar multiplication for 3-Hilbert spaces, generalizing \cite[(11)]{2411.01678}.
Just as each Hilbert space comes with a $\bbC$-action, every 2-Hilbert space $\cC$ is tensored over $\Hilb$, in that $\cC$ is canonically a $\Hilb$-module category.
Each $c\in\cC$ gives a creation operator $\Hilb\to \cC$ given by $H\mapsto H\odot c$. 
However, more is true: we may canonically scale each $c\in \cC$ by a \emph{generalized scalar} $H\in \Mod^\dag(\Omega_c)$
the category of Hilbert space modules for the $\rmH^*$-algebra $\Omega_c\coloneq\End_\cC(c)$, which is a $\rmW^*$-algebra equipped with a positive faithful trace.
This construction was originally formulated in \cite[(11)]{2411.01678} in the language of $\rmW^*$-categories, and we provide an isometric version in Construction \ref{cstr:rel-tsr-prod-over-A} for 2-Hilbert spaces.

This story immediately categorifies to 3-Hilbert spaces.
Given $c\in \fX$, $\Omega_c:=\End_\fX(c)$ is an \emph{$\rmH^*$-multifusion category}, which is a unitary multifusion category equipped with a fixed unitary dual functor in the sense of \cite{MR4133163} and a spherical weight on $\End(1_c)$.
Just as modules for $\rmH^*$-algebras come with module traces, modules for $\rmH^*$-multifusion categories come with module category traces in the sense of \cite{MR3019263,MR4598730}.

\begin{defnalpha}[\ref{defn:coefficientsforobjects}] \label{defnalpha:coefficientsforobjects}
Let $\fX$ be a 3-Hilbert space, $c\in\fX$, and $\cM_{\Omega_c}\in \Mod^\dag(\Omega_c)$.
We define
$
\cM \boxdot_{\Omega_c} c
\in\fX
$
as the object corresponding to $\cM$ under the isometric embedding
$$
\begin{tikzcd}[row sep =0]
\cM\arrow[r, mapsto]
&
\cM\boxdot_{\Omega_c} c
\\
\Mod^\dag(\Omega_c)  
\arrow[r, hook,"|c\rangle_\Omega"]
& \fX
\\
\mbox{}
\\
\mbox{}
\\
\rmB \Omega_c \arrow[uuur, hook]
\arrow[uuu, hook]
\end{tikzcd}
$$
where this \emph{generalized ket} operator $|c\rangle_\Omega \colon \Mod^\dag(\Omega_c)\to \fX$ is the unique extension of the map $\rmB\Omega_c\hookrightarrow \fX$ 
given by $*\mapsto c$,
afforded by the universal property of completion
(see \eqref{eq:UniversalPropertyOfCompletion} below).
\end{defnalpha}
We show that this construction is compatible with \emph{bra operators} \eqref{eq:BraFuUnctors} in Proposition \ref{prop:2HilbertScalarsPullOut}.
One can also define $\cM\boxdot_\cA c$ for any module $\cM_\cA$ for an $\rmH^*$-multifusion category $\cA$ equipped with an isometric map $\cA\isometry \Omega_c$; see Definition \ref{D:reltsrprodH*monad} for more details.
Moreover, there is a canonical pentagonator which witnesses the higher associativity of generalized scalar multiplication.

\begin{rem}
In fact, this generalized scalar multiplication has already appeared in another language in the `infinite dimensional $\rmH^*$-algebra' setting, namely the 2-category of tracial von Neumann algebras.
Fix a tracial von Neumann algebra $(N,\tr_N)$,
an $\rmH^*$-multifusion category $\cA$, 
and an isometric dagger tensor functor $F:\cA \isometry \Omega_N:=\Bim(N)$.
Then $\cM\in \Mod^\dag(\cA)$ corresponds to an $\rmH^*$-algebra $A\in \HstarAlg(\cA)$, and the \emph{realization} $|A|_F$ in the sense of \cite{MR3948170,MR4079745} is precisely $\cM\boxdot_\cA N$.
Indeed, the connection can be made by observing that two simplifications were made in these articles:
\begin{itemize}
\item 
$N$ was always chosen to be a $\rm II_1$ factor obscuring the use of infinite dimensional $\rmH^*$-algebras/tracial von Neumann algebras, and
\item 
the compact $\rmW^*$-algebra objects used in \cite{MR4079745} were assumed to be \emph{connected}, and were thus equipped with a canonical state which is automatically tracial by \cite[Prop.~2.6]{MR3948170}, making them honest $\rmH^*$-algebra objects.
\end{itemize}
\end{rem}

\begin{rem}
The recent results in \cite{2507.05185} indicate that generalized scalar multiplication should be intimately connected the the SymTFT construction in quantum field theory \cite{1702.00673,MR4814695}.
In more detail, the article \cite{2507.05185} gave an operator algebraic interpretation of the SymTFT framework in terms of a physical boundary subalgebra $B$ for a quasi-local algebra $A$ over $\bbZ$, which should be viewed as the `fixed points' by a fusion category symmetry $\cS$ on $A$.
Under mild conditions, $A$ is the Q-system realization of a Lagrangian algebra object in the braided $\rmC^*$ tensor category of \emph{DHR bimodules} of $B$ \cite{MR4814692}.
In the 2-category of $\rmC^*$-algebras, $A=\cM\boxdot_\cC B$ where $\cC=\mathsf{DHR}(B)$, $\cM$ is the topological boundary corresponding to the Lagrangian algebra $L\in \mathsf{DHR}(B)$, and the symmetry category $\cS=\End(\cM_\cC)$.
\end{rem}

Equipped with the notion of generalized scalars, we discuss \emph{orthonormal bases} (ONB) for 3-Hilbert spaces in \S\ref{subsec:ONBS}. 

\begin{defnalpha}[\ref{defn:ONB}]
An \emph{orthonormal basis} (ONB)\footnote{\label{Footnote:ONB}Here, \emph{basis} refers to the property that there is only one $b_i$ per component, and \emph{normal} refers to the property that each $b_i$ is simple.
The \emph{ortho} part of \emph{orthonormal} is automatic since $\fX(b_i\to b_j)=0$ when $i\neq j$.} for a 3-Hilbert space $\fX$ consists of a choice of simple $b_i$ for each connected component of $\fX$. 
We denote a choice of ONB by $\pi_0\fX$.\footnote{At this point, this is a slight abuse of nomenclature, as an ONB is really a choice of simple in each component, rather than the components themselves.
Some might prefer the notation $\operatorname{Irr}(\fX)$ for such a choice.
We prefer $\pi_0\fX$ as readers familiar with fusion categories might confuse $\operatorname{Irr}(\fX)$ with a set of representatives for the 1-equivalence classes of simple objects, not just one simple per connected component of $\fX$ cf.~\cite[Def.~1.2.22]{1812.11933}.
} 
\end{defnalpha}

The data of an ONB is equivalent to the data of a finite list $\{\cC_i\}_i = \{\Omega_{b_i}\}_i$ of H*-fusion categories and an isometric equivalence of 3-Hilbert spaces $\bigboxplus_i \Mod^\dag(\cC_i) \cong \fX$.
We prove that the space of all ONBs is \emph{contractible} in the higher categorical sense.
This is a higher categorical analogue of the fact that given any two choices of basis for a vector space, there is a unique change of basis map.

\begin{thmalpha}[\ref{spaceofONBscontractible}]
The 3-groupoid of ONBs for a 3-Hilbert space $\fX$ is contractible.
\end{thmalpha}

Using ONBs, we show in Corollary \ref{cor:AllPresheafsRepresentable} that every presheaf $\fX^{1\op}\to 2\Hilb$ is isometrically representable via 
$$
F(x)=
\fX\left(x\to
\bigboxplus_{b\in\pi_0\fX} F(b)\boxdot_{\Omega_b} b\right).
$$
This formula allows us to define the notion of \emph{unitary 2-adjunction} for functors/operators between 3-Hilbert spaces.

Using the tools of generalized scalar multiplication, ONBs, and unitary 2-adjunction, we endow the space $\Hom(\fX \to \fY)\coloneq\Fun^{\dag,\vee}(\fX\to \fY)$ of UAF and $\dag$-preserving functors between 3-Hilbert spaces with the structure of a 3-Hilbert space itself.
That is, $3\Hilb$ is \emph{self-enriched}, as expected in \cite[\S5.3]{MR5078555}.
In \S\ref{subsec:selfenrichmentdesiderata} below, we prove this self-enrichment satisfies the desiderata from that article.
Finally, we use this self-enrichment to prove the following two important structural results for 3-Hilbert spaces.

\begin{thmalpha}[Unitary Yoneda]\label{thm:unitaryyoneda}
The Yoneda embedding $x\mapsto \fX(-\to x)$ is an isometric equivalence of 3-Hilbert spaces $\fX\cong \Hom(\fX^{1\op}\to 2\Hilb)$.
\end{thmalpha}

\begin{thmalpha}[Folding trick]\label{thm:foldingtrick}
There is a canonical isometric equivalence
$\Hom(\fX\to \fY)\cong \fY\threehilbtimes \fX^{1\op}$.
\end{thmalpha}

\begin{rem}
\label{rem:NonUnitary}
The generalized scalar multiplication from Definition \ref{defnalpha:coefficientsforobjects} is similar in spirit to the relative Deligne product in the non-unitary setting; see \cite[\S3]{MR4600461}.
The analogue of Theorem \ref{thm:unitaryyoneda}, that the (absolute) Yoneda embedding is an equivalence for non-unitary finite semisimple 2-categories, is \cite[Prop.~1.4.11]{1812.11933}.
Theorem \ref{thm:foldingtrick} for non-unitary finite semisimple 2-categories follows by combining \cite[Thm.~5.1.2]{MR4600461} and \cite[Thm.~2.2.4]{2311.16827}.
\end{rem}

\medskip
\noindent
{\textbf{Applications.}}
There are several forthcoming applications of the tools in this article for 3-Hilbert spaces.

In the upcoming article \cite{UnitaryDiskLike}, Wesley proves that (finite) 2- (respectively 3)-Hilbert spaces are precisely the local data of a fully extended 2D (respectively 3D) unitary TQFT. 
The 3D case directly uses many of the tools developed in this article.

Second, Penneys and Wesley together with Kevin Walker
are developing string net lattice models in higher dimensions using graphical calculus and skein theoretic techniques.
These models are essentially a Hamiltonian presentation of the fully extended TQFTs discussed in the previous application.
The tools of this article are essential for getting the unitarity of these constructions correct for $3+1$D models.

We also expect that the tools of this article will be helpful in characterizing unitary dualizability for unitary multifusion 2-categories in the spirit of \cite{MR4133163}.
Such a characterization will be essential for both discussing the many versions of 3-Hilbert spaces extending the treatment for 2D in \cite{2606.11334}, along with defining the notion of 4-Hilbert space (see below).

\medskip
\noindent
{\textbf{Towards looking deeper.}}
We conjecture that the $\rmC^*$-3-category $3\Hilb$ equipped with 
\begin{itemize}
\item unitary 2-adjunction $*$ as described in \S\ref{subsec:Unitary2Adjunction}, and
\item spherical weights $\Phi_\fX$ on each $\End(\fX)$ arising from the self-enrichment in \S \ref{sec:3HilbSelfEnriched}
\end{itemize}
forms a \emph{4-Hilbert space}, a notion which has not yet been established. 
Making this precise and rigorous would require
\begin{itemize}
\item a study of unitary dual (resp. adjoint) 2-functors on multifusion 2-categories (resp. semisimple $\rmC^*$-3-categories), and
\item 
a study of (isometric/Hilbert) completeness for such $\rmC^*$-3-categories in the presence of spherical weights. 
\end{itemize}
These tasks both lie outside of the scope of this paper. 
To begin these questions, one should use that $2\Hilb$ forms the prototypical example of an $\rmH^*$-(multi)fusion 2-category, and $3\Hilb$ should have a canonical $\rmO(3)$-dagger structure in the sense of \cite{2403.01651,2606.11334}.
\subsection*{Acknowledgments}
We thank Chumeng Di for many discussions along the way, in particular in the proof of Proposition \ref{prop:EnrichmentWellDefined}.
We also thank
Thibault D\'ecoppet, 
Corey Jones,
Lukas M\"uller, and Luuk Steuhouwer
for helpful discussions. 
The authors were supported by NSF DMS grants 2154389, 2244045, and 2554723.

\section{Background on 2-Hilbert spaces}
\label{sec:backgroundon2Hilbs}
In this section, we recall the necessary background on  2-Hilbert spaces from \cite{MR1448713}, streamlining our approach 
using techniques from \cite{2411.01678}.

\begin{nota}
In this article, we typically use the characters $H,K,L$ for Hilbert spaces, $\cA,\cB,\cC$ for 1-categories, and $\fX,\fY,\fZ$ for 2-categories.
That is, we attempt to use different fonts and different sections of the alphabet for objects with different category numbers.

For an object $c\in\cC$, we write $\Omega_c\coloneq\End_\cC(c)$.
For an object $x\in\fX$, we write $\Omega_x\coloneq\End_\fX(x)$ and $\Omega^2_x\coloneq\End_\fX(1_x)$.
\end{nota}

\subsection{2-Hilbert spaces}\label{subsec:2Hilbs}

A \emph{2-Hilbert space} is a finite semisimple $\rmC^*$-category $\cC$ equipped with a positive definite trace $\Tr^\cC$, which is a family of positive definite linear functionals $\Tr^\cC_c\colon  \Omega_c\to \bbC$ for each $c\in \cC$ satisfying the tracial condition
$$
\Tr^\cC_a(f^\dag g) = \Tr^\cC_b(g^\dag f)
\qquad\qquad\qquad
\forall f,g\in\cC(a\to b).
$$
In \cite{MR1448713}, a set $\pi_0\cC$ of representatives for the simple objects of $\cC$ is viewed as an \emph{orthonormal basis} (ONB) for $\cC$, where equivalence of simple objects is given by isometric (trace-preserving) equivalence in $\cC$.

We see that our 2-Hilbert space is \emph{enriched} over $\Hilb$, in that each $\cC(a\to b)$ comes equipped with an inner product $\langle f|g\rangle_{\cC(a\to b)} \coloneq\Tr^\cC_a(f^\dag  g)$, which satisfies
\begin{equation}
\langle h^\dag \circ k | g \rangle_{a\to b}
=
\langle k | h \circ g \rangle_{a\to c}
=
\langle k\circ g^\dag | h  \rangle_{b\to c}
\tag{\ref{eq:2HilbsAreHilbEnriched}}
\end{equation}
for all 
$g\colon a\to b$, $h\colon b\to c$, and $k\colon  a\to c$.
As noted in \cite[Footnote~12]{MR4598730},
the second equality in \eqref{eq:2HilbsAreHilbEnriched} holds if and only if each representable functor $\cC(-\to c)\colon \cC^{\op}\to \Hilb$ is a $\dag$-functor, in which case, the first equality in \eqref{eq:2HilbsAreHilbEnriched} implies that the Yoneda embedding $\cC\hookrightarrow \Fun^\dag(\cC^{\mathrm{op}}\to \Hilb)$ is a $\dag$-functor.
Moreover, when $\Fun^\dag(\cC^{\op}\to \Hilb)$ is equipped with the positive definite trace
$$
\Tr^{\Fun}_F(\eta\colon F\Rightarrow F)
\coloneq
\sum_{c\in\Irr\cC} d_c \Tr^\Hilb_{F(c)}(\eta_{c}),
$$
the Yoneda embedding is also isometric, i.e., trace-preserving \cite[Rem.~2.16]{MR5078555}.

\begin{rem}
The trace $\Tr^\cC$ allows us to identify $\cC(a\to b)$ with the Hilbert-space valued inner product \cite{MR2325696,2411.01678}:
$$
\langle a|b\rangle^\cC_\Hilb \coloneq p_b L^2(\End_\cC(a\oplus b),\Tr^\cC_{a\oplus b}) p_a.
$$
\end{rem}

\begin{nota}
In this article, all ($n$-)functors are assumed to be linear $\dag$-functors unless otherwise stated.
We use the following suggestive notation, under the desideratum that the most well-behaved notion gets the best notation:
\begin{itemize}
\item 
We write $F\colon \cA\isometry \cB$ for a trace-preserving functor, i.e., $\Tr^\cB_{F(a)}(F(f))=\Tr^\cA_a(f)$ for all $f\in \Omega_a$.
The hook on the back represents faithfulness, and the upwards harpoon represents trace-preserving.
The fact that trace-preserving functors are automatically faithful is captured by the hook and the harpoon being on the same side.
\item 
We write $F\colon  \cA\fullyfaithful \cB$ for a fully faithful functor.
Since fullness and faithfulness are independent notions, the hook and harpoon are on different sides.
\item 
We write $F\colon \cA\hookrightarrow \cB$ for a fully faithful trace-preserving functor, i.e., an \emph{isometry}.
Observe that $\hookrightarrow \,=\, \isometry + \fullyfaithful$.
\item 
We write $F\colon \cA\cong^\dag \cB$ to denote an \emph{isometric equivalence} of categories.
(Sometimes we omit the $\dag$ if we include `isometric equivalence' in words.)
\end{itemize}
\end{nota}

\begin{ex}
Recall from \cite{MR13235,MR3971584,MR5078555} that an $\rmH^*$-\emph{algebra} is a finite dimensional tracial von Neumann algebra, a.k.a., a finite dimensional $\rmC^*$/von Neumann algebra $A$ equipped with a faithful (non-normalized) positive trace $\Tr_A\colon  A\to \bbC$.
We define $\Mod^\dag(A)$ as the space of finite dimensional right Hilbert space $A$-modules $H_A$, which are organically equipped with right $A$-valued inner products
$$
\langle \eta|\xi\rangle_A \coloneq L_\eta^\dag L_\xi
$$
where $L_\eta\colon  L^2(A,\Tr_A)\to H$ is given by $a\Omega\mapsto \eta a$.
Moreover, since $H_A$ is finite dimensional,
$$
\End(H_A) = \operatorname{span}\set{L_\xi L_\eta^\dag}{\eta,\xi\in H},
$$ 
which is again an $\rmH^*$-algebra with trace determined by
$$
\Tr_{\End(H_A)}(L_\xi L_\eta^\dag)
=
\Tr_A(L_\eta^\dag L_\xi) 
= 
\Tr_A(\langle \eta|\xi\rangle_A)
=
\langle \eta|\xi\rangle_H.
$$
One verifies that these traces endow $\Mod^\dag(A,\Tr_A)$ with the structure of a 2-Hilbert space, and all 2-Hilbert spaces arise in this way.
\end{ex}

\begin{rem}
\label{rem:ExtendTraceToCompletion}
One can also define the 2-Hilbert space structure on $\Mod^\dag(A,\Tr_A)$ as follows.
Since $\End(L^2(A,\Tr_A)_A) \cong A$ we may equip $\End(L^2(A,\Tr_A)_A)$ with the trace $\Tr_A$.
This induces a unique trace $\Tr_{L^2(A,\Tr_A)^{\oplus n}} \coloneq \bigoplus^n \Tr_A$ on each free module $L^2(A,\Tr_A)^{\oplus n}_A$.
Then as every object in $\Mod^\dag(A,\Tr_A)$ is projective, we then extend the trace to every projective module $pL^2(A,\Tr_A)^{\oplus n}$ 
for $p\in M_n(A) = \End(L^2(A,\Tr_A)^{\oplus n}_A)$
determined by $\Tr_{pL^2(A,\Tr_A)^{\oplus n}} = \Tr_{L^2(A,\Tr_A)^{\oplus n}}(p \, \cdot\,).$ 
\end{rem}

\begin{defn}
A \emph{unitary category} is a $\rmC^*$-category $\cC$ with finite dimensional hom spaces such that each \emph{linking algebra}
$$
L(a_1,\dots, a_n)\coloneq
\begin{bmatrix}
\Omega_{a_1} & \cC(a_2\to a_1) & \cdots & \cC(a_n\to a_1) \\
\cC(a_1\to a_2) & \Omega_{a_2} & \cdots & \cC(a_n\to a_2) \\
\vdots & \vdots & \ddots & \vdots \\
\cC(a_1\to a_n) & \cC(a_2\to a_n) & \cdots & \Omega_{a_n}
\end{bmatrix}
$$
is a finite dimensional $\rmC^*$-algebra under matrix multiplication and conjugate transpose.

A \emph{pre-2-Hilbert space} is a unitary category $\cC$ equipped with a faithful positive trace $\Tr^\cC$ that is \emph{finite} in that there is a global bound on the dimension of the center of all linking algebras.
\end{defn}

\begin{ex}
The delooping $\rmB(A,\Tr_A)$  of an $\rmH^*$-algebra $(A,\Tr_A)$ is a pre-2-Hilbert space.
\end{ex}

\begin{defn}
The \emph{completion} $\cC^\cent$ of a pre-2-Hilbert space $\cC$ is the orthogonal projection completion of its orthogonal additive envelope \cite{2411.01678}.
The faithful positive trace $\Tr^\cC$ is extended to the completion analogous to Remark \ref{rem:ExtendTraceToCompletion}, under which the completion becomes a 2-Hilbert space.
Conversely, every 2-Hilbert space is complete (admits orthogonal direct sums and is orthogonal projection complete).
\end{defn} 

\begin{ex}
For an $\rmH^*$-algebra $(A,\Tr_A)$, 
the completion of the \emph{delooping}
$\rmB(A,\Tr_A)^\cent\cong^\dag{\Mod^\dag(A,\Tr_A)}$.
\end{ex}

The completion of a pre-2-Hilbert space comes equipped with an isometric fully faithful functor $\cC\hookrightarrow \cC^\cent$
which satisfies the following universal property.
\begin{itemize} 
\item 
For every 2-Hilbert space $(\cD,\Tr^\cD)$ (which is automatically complete), precomposition with $\iota$ is an 
equivalence of unitary categories
$$
\Hom(\cC^\cent\to \cD)
\xrightarrow{-\circ \iota}
\Hom(\cC\to \cD).
$$
\end{itemize}

\subsection{Generalized scalars in 2-Hilbert spaces}

A 2-Hilbert space is also \emph{tensored} over $\Hilb$, in that $\cC$ is canonically a $\Hilb$-module category.
In other words, there is a bilinear functor $\odot\colon \Hilb\times \cC\to \cC$ satisfying an associativity coherence with the monoidal structure of $\Hilb$, where for $c\in \cC$ and $H\in \Hilb$, the object $H\odot c$ is the orthogonal direct sum of $\dim(H)$ copies of $c$. 
The intimate connection between enrichment and tensoring for ordinary categories \cite{MR2177301} extends to the unitary setting \cite{MR4814691}
in that we have unitary hom-tensor adjunction
$$
\cC(H\odot a \to b) \isomeq \Hilb(H\to \cC(a\to b))
$$
where $\Hilb$ has the usual trace determined by $\Tr^\Hilb_\bbC(1)=1$.

Now given an object $c\in (\cC,\Tr^\cC)$, an isometric (trace-preserving) map of $\rmH^*$-algebras $(A,\Tr_A)\isometry \Omega_c\coloneq\End_\cC(c)$,\footnote{One might call such a $c\in \cC$ an \emph{$(A,\Tr_A)$-module internal to $(\cC,\Tr^\cC)$}.} 
and $H_A\in \Mod^\dag(A,\Tr_A)$,
there is another closely related construction 
$H\odot_A c$ modified from \cite[(11)]{2411.01678}.

\begin{construction}[{Adapted from \cite[(11)]{2411.01678}}]
\label{cstr:rel-tsr-prod-over-A}
For a 2-Hilbert space $(\cC,\Tr^\cC)$ and $c\in \cC$,  an isometry $(A,\Tr_A)\isometry (\Omega_c,\Tr^\cC_c)$
has a unique extension to an isometry $-\odot_A c\colon \Mod^\dag(A,\Tr_A)\isometry \cC$.
$$
\begin{tikzcd}[row sep=0]
H_A
\arrow[r,mapsto]
&
H\odot_A c
\\
\Mod^\dag(A,\Tr_A) 
\arrow[r,harpoon, hook]
&
\cC
\\\mbox{}\\\mbox{}\\
\rmB (A,\Tr_A)
\arrow[uuu,hook]
\arrow[r,harpoon, hook]
&
\rmB\Omega_c
\arrow[uuu,hook]
\end{tikzcd}
$$
This functor is well-behaved with respect to both Connes fusion in the 2-category $\HstarAlg$ of $\rmH^*$-algebras and Hilbert space bimodules, as well as isometric functors between 2-Hilbert spaces.
That is, the universal property gives a coherent family of unitary natural isomorphisms
$$
(K\boxtimes_B H)\odot_A c
\cong
K\odot_B (H\odot_A c)
$$
whenever $H$ is a $B-A$ bimodule \cite[Lem.~3.36]{2411.01678},
and for every trace-preserving functor $F\colon (\cC,\Tr^\cC)\to (\cD,\Tr^\cD)$, we get a canonical natural unitary
\begin{equation} \label{eq:FunctorsAreLinear}
F(H\odot_A c) \cong H\odot_A F(c).
\end{equation}
\end{construction}

\begin{sub-ex}
\label{sub-ex:ConnesFusion1}
As a specific example of the above feature, 
since $(\cC,\Tr^\cC)$ is a 2-Hilbert space, each Hilbert space $\cC(c\to b)$ is a right Hilbert $(\Omega_c,\Tr^\cC_c)$-module with $\Omega_c$-valued inner product
$$
\langle f_2|f_1\rangle_{\Omega_c}\coloneq f_2^\dag f_1 \in \Omega_c,
$$
and a left Hilbert $\Omega_b$-module with $\Omega_b$-valued inner product 
$$
{}_{\Omega_b}\langle g_1,g_2\rangle\coloneq g_1g_2^\dag\in \Omega_b.
$$
Since a linear functor
$\Mod^\dag(\Omega_c)\to \Hilb$ is determined by where it sends the generator $L^2(\Omega_c,\Tr^\cC_c)$,
for every $a,b\in \cC$,
we have a canonical natural unitary isomorphism
$$
\cC(c\to b)\odot_{(\Omega_c,\Tr^\cC_c)} \cC(a\to c)
\cong
\cC(c\to b)\boxtimes_{(\Omega_c,\Tr^\cC_c)} \cC(a\to c),
$$
where the latter denotes the Connes fusion relative tensor product in the 2-category $\HstarAlg$. 
We refer the interested reader to \cite[Part I, \S 3]{UQSLbook} for an accessible exposition on the subject.
Hence the inner product is given by
$$
\langle
f_1\boxtimes g_1
|
f_2\boxtimes g_2
\rangle
=
\langle
\langle f_2|f_1\rangle_{\Omega_c}\cdot g_1
|
g_2
\rangle
=
\langle f_2^\dag f_1 g_1|g_2\rangle
=
\Tr^\cC(g_2^\dag f_2^\dag f_1g_1),
$$
which is also equal to the formulas
$$
\langle
f_1\cdot {}_{\Omega_c}\langle g_1,g_2\rangle|
f_2
\rangle
\qquad\qquad\text{and}\qquad\qquad
\Tr^\cC(\langle f_2|f_1\rangle_{\Omega_c}\cdot{}_{\Omega_c}\!\langle g_1,g_2\rangle).
$$
Thus the composition map
$$
-\circ-\colon \cC(c\to b)\odot_{\Omega_c}\cC(a\to c)\to \cC(a\to b),
$$
which is well-defined as it is $\Omega_c$-middle linear, is an isometry.
\end{sub-ex}

\begin{rem}
When $c\in\pi_0\cC$,
even though $\Omega_c\cong \bbC$ as algebras, the inner products on the Hilbert spaces
$$
\cC(c\to b) \boxtimes_{(\Omega_c,\Tr^\cC_c)} \cC(a\to c)
\qquad\quad\text{ and }\qquad\quad
\cC(c\to b) \otimes \cC(a\to c)
$$
differ by the scalar $d_c \coloneq \Tr^\cC_c(\id_c)$.
\end{rem}

We end this section by remarking that our operation $-\odot_A c\colon  \Mod^\dag(A,\Tr_A)\to \cC$ for a given $A\isometry \Omega_c$ satisfies a unitary hom-tensor adjunction.
Indeed, applying \eqref{eq:FunctorsAreLinear} to the functor
$\cC(-\to b)\colon  \cC^{\op}\to \Hilb$ 
gives a canonical unitary isomorphism
\[
\cC(H \odot_A c \to b) \underset{\eqref{eq:FunctorsAreLinear}}{\cong} 
\cC(c \to b) 
\boxtimes_A 
\overline{H} 
\cong 
\Mod^\dag(A)(H_A \to \cC(c \to b)_A).
\]

\section{Background on 3-Hilbert spaces}
\label{sec:backgrounon3Hilbs}

In this section, we recall the necessary background on 3-Hilbert spaces from \cite{MR5078555}.
We include some new results, including:
\begin{itemize}
\item
In \S\ref{subsec:H*monads}, we prove that the space of splittings of an $\rmH^*$-monad is contractible, which was overlooked in \cite{MR5078555}.
\item
In \S\ref{sec:UnitaryYoneda}, we prove an isometric unitary Yoneda lemma internal to a 3-Hilbert space.
\end{itemize}

\subsection{Pre 3-Hilbert spaces}

\begin{defn}
Let $\fX$ be a 2-category.
The \emph{linking ($E_1$)-algebra} on the objects $a_1,\dots,a_n \in \fX$, denoted $\cL(a_1,\dots,a_n)$, is the matrix $E_1$-algebra
\[ 
\begin{bmatrix}
\Omega_{a_1} & \fX(a_2\to a_1) & \cdots & \fX(a_n\to a_1) \\
\fX(a_1\to a_2) & \Omega_{a_2} & \cdots & \fX(a_n\to a_2) \\
\vdots & \vdots & \ddots & \vdots \\
\fX(a_1\to a_n) & \fX(a_2\to a_n) & \cdots & \Omega_{a_n}
\end{bmatrix}
\]
where multiplication is given by matrix multiplication and composition in $\fX$.
Notice that if $\fX$ has finite direct sums, we have an isomorphism of algebras $ \cL(a_1,\dots,a_n) \cong \End_\fX(\bigoplus_i a_i)$.
\end{defn}

\begin{defn}
A \emph{unitary 2-category} is a $\rmC^*$-2-category (see \cite{MR4419534}) such that all the $n$-fold linking algebras are unitary multitensor categories.
That is, for any set of objects $a_1,\dots,a_n$, the monoidal category $\cL(a_1,\dots,a_n)$ is semisimple and rigid, and its linking algebras are all finite dimensional $\rm C^*$-algebras.
Observe that this means that $\fX$ is \emph{locally Cauchy complete}, i.e., all hom 1-categories $\fX(a\to b)$ admit orthogonal direct sums and are orthogonal projection complete.

We will restrict our attention to \emph{finite} unitary 2-categories, i.e., ones where each linking algebra is a unitary multifusion category, and there is a global bound on the dimension 
$\dim(\End_{\cZ(\cL)}(1_\cL))$ 
over all linking algebras $\cL$ of $\fX$.
\end{defn}

\begin{defn}\label{defn:UAF}
Let $\fX$ be a unitary 2-category. A \emph{unitary adjoint functor} (UAF) $\vee \colon \fX^{1\op,2\op} \to \fX$ consists of a choice of adjoint $({}_bX^\vee_a,\ev_X,\coev_X)$ for every 1-morphism ${}_aX_b$ such that
\begin{itemize}
    \item the canonical unit maps $1_a^\vee\to 1_a$ are unitary,\footnote{Here, the canonical map is given by the unitor $\ev_a: 1_a^\vee\otimes_a 1_a\to 1_a$, where $1_a$ is suppressed in the domain.}
    \item the canonical tensorator maps $\nu_{X,Y} \colon X^\vee \otimes Y^\vee \Rightarrow (Y \otimes X)^\vee$ are unitary; and
    \item for a 2-morphism $f$ in $\fX$, $(f^\vee)^\dag = (f^\dag)^\vee$.
\end{itemize}
Equivalently, these are the conditions in order for $\vee$ to assemble into a $\dag$-2-functor.
\end{defn}

\begin{nota}
We make heavy use of the graphical calculus for 2-categories.
For background, we refer the reader to \cite{MR3971584} and \cite{MR4419534}.
We denote objects $a,b\in\fX$ by shaded regions, a 1-morphism ${}_aX_b\in\fX$ and its adjoint ${}_bX^\vee_a$ by strands
$$
X=
\tikzmath{
\begin{scope}
\clip[rounded corners=5pt] (0,0) rectangle (.6,.6);
\fill[lightgray!50] (0,0) rectangle (.3,.6);    
\fill[lightgray] (.3,0) rectangle (.6,.6);    
\end{scope}
\draw[thick] (.3,0) -- (.3,.6);
}
\qquad\qquad
X^\vee=
\tikzmath{
\begin{scope}
\clip[rounded corners=5pt] (0,0) rectangle (.6,.6);
\fill[lightgray] (0,0) rectangle (.3,.6);    
\fill[lightgray!50] (.3,0) rectangle (.6,.6);    
\end{scope}
\draw[thick] (.3,0) -- (.3,.6);
}
\qquad\qquad\text{where}\qquad\qquad
a=
\tikzmath{
\fill[lightgray!50, rounded corners=5pt] (0,0) rectangle (.6,.6);    
}
\qquad\qquad
b=
\tikzmath{
\fill[lightgray, rounded corners=5pt] (0,0) rectangle (.6,.6);    
}
\,,
$$
and the evaluation and coevaluation for are given by shaded cups and caps:
$$
\ev_X =
\tikzmath{
\fill[lightgray, rounded corners=5pt] (0,0) rectangle (1.2,.6);
\draw[thick,fill=lightgray!50] (.3,0) arc(180:0:.3cm);
}
\qquad
\coev_X^\dag =
\tikzmath{
\fill[lightgray!50, rounded corners=5pt] (0,0) rectangle (1.2,.6);
\draw[thick,fill=lightgray] (.3,0) arc(180:0:.3cm);
}
\qquad
\ev_X^\dag =
\tikzmath{
\fill[lightgray, rounded corners=5pt] (0,0) rectangle (1.2,.6);
\draw[thick,fill=lightgray!50] (.3,.6) arc(-180:0:.3cm);
}
\qquad
\coev_X =
\tikzmath{
\fill[lightgray!50, rounded corners=5pt] (0,0) rectangle (1.2,.6);
\draw[thick,fill=lightgray] (.3,.6) arc(-180:0:.3cm);
}
\,.
$$
\end{nota}

\begin{rem}
A UAF $\vee$ on a unitary 2-category endows each $\End_\fX({}_a X_b)$ with positive definite $\End_\fX(1_a)$- and $\End_\fX(1_b)$-valued traces given by
$$
\tr^\vee_R\left(
\tikzmath{
\begin{scope}
\clip[rounded corners=5] (-.5,-.7) rectangle (.5,.7);
\fill[fill=lightgray!50] (-.5,-.7) rectangle (0,.7);
\fill[fill=lightgray] (.5,-.7) rectangle (0,.7);
\end{scope}
\draw[thick] (0,-.7) --node[left]{$\scriptstyle X$} (0,-.3);
\draw[thick] (0,.7) --node[left]{$\scriptstyle X$} (0,.3);
\roundNbox{fill=white}{(0,0)}{.3}{0}{0}{$f$}
}
\right)
\coloneq
\tikzmath{
\fill[fill=lightgray!50,rounded corners=5] (-.55,-.7) rectangle (1.1,.7);
\draw[thick,fill=lightgray] (-.15,0) arc (180:360+180:.5);
\node at (-.15,-.5) {$\scriptstyle X$};
\roundNbox{fill=white}{(0,0)}{.3}{0}{0}{$f$};
}
\qquad\qquad
\tr^\vee_L\left(
\tikzmath{
\begin{scope}
\clip[rounded corners=5] (-.5,-.7) rectangle (.5,.7);
\fill[fill=lightgray!50] (-.5,-.7) rectangle (0,.7);
\fill[fill=lightgray] (.5,-.7) rectangle (0,.7);
\end{scope}
\draw[thick] (0,-.7) --node[left]{$\scriptstyle X$} (0,-.3);
\draw[thick] (0,.7) --node[left]{$\scriptstyle X$} (0,.3);
\roundNbox{fill=white}{(0,0)}{.3}{0}{0}{$f$}
}
\right)
\coloneq
\tikzmath{
\fill[fill=lightgray,rounded corners=5] (-.3,-.7) rectangle (1.35,.7);
\draw[thick,fill=lightgray!50] (-.05,0) arc (180:360+180:.5);
\node at (.95,-.5) {$\scriptstyle X$};
\roundNbox{fill=white}{(0.8,0)}{.3}{0}{0}{$f$};
}
\qquad\qquad
\forall f\in \End_\fX({}_aX_b).
$$
\end{rem}

\begin{defn}[{c.f.~\cite[Def.~4.45]{MR5078555}}]
\label{3hilbertspace}
Given a unitary 2-category $\fX$ equipped with a UAF $\vee$, a \emph{spherical weight} is a collection of positive definite linear functionals $\Psi^\fX_a\colon  \End_\fX(1_a)\to \bbC$ for each $a\in \fX$ such that
$$
\Psi^\fX_a\circ \tr^\vee_R
=
\Psi^\fX_b\circ \tr^\vee_L
\qquad\qquad\qquad\qquad
\forall\, a,b\in\fX.
$$
A \emph{pre-3-Hilbert space} $(\fX, \vee, \Psi)$ is a finite unitary 2-category $\fX$ equipped with a UAF $\vee$ and a spherical weight $\Psi$.
\end{defn}

\begin{defn}
A \emph{1-morphism} of pre-3-Hilbert spaces is a $\dag$-2-functor $F\colon  \fX\to \fY$ that is \emph{UAF-preserving}, i.e., the canonical isomorphism
\begin{equation}
\label{eq:IsometricIsUAFPreserving}
\tikzmath{
\begin{scope}
\clip[rounded corners=5pt] (-1.6,-1) rectangle (2.2,1.7);
\fill[lightgray!50] (-.4,-1) -- (-.4,-.3) -- (.4,-.3) arc(-180:0:.3cm) -- (1,1.7) -- (-1.6,1.7) -- (-1.6,-1);
\fill[lightgray] (-.4,-1) -- (-.4,-.3) -- (.4,-.3) arc(-180:0:.3cm) -- (1,1.7) -- (2.2,1.7) -- (2.2,-1);
\end{scope}
\draw[thick] (-.4,-1) --node[left]{$\scriptstyle F(X^\vee)$} (-.4,-.3);
\draw[thick] (.4,-.3) arc(-180:0:.3cm) --node[right]{$\scriptstyle F(X)^\vee$} (1,1.7);
\draw[thick,double] (0,.3) -- (0,.7);
\roundNbox{fill=white}{(0,1)}{.3}{.4}{.4}{$F(\ev_X)$}
\roundNbox{fill=white}{(0,0)}{.3}{.4}{.4}{$F^2_{X^\vee,X}$}
}
\end{equation}
is unitary.\footnote{At this point, some readers might be confused why we do not require 1-morphisms to preserve spherical weights.
One should view the spherical weight as determining `lengths,' and one would never require linear maps between ordinary 1-Hilbert spaces to preserve lengths of vectors.
This is the same reason why we do not require 1-morphisms between 2-Hilbert spaces to preserve the positive traces.}

A \emph{2-morphism} $\alpha \colon F \Rightarrow G$ between 1-morphisms $F,G \colon \fX \to \fY$ is a $\dag$-2-natural transformation, i.e., a 2-natural transformation $\alpha = (\{\alpha_a\}_{a \in \fX}, \{\alpha_X\}_{X \in \fX(a \to b)})$ such that the naturator 
$$\alpha_X \in \fY(FX \otimes_{Fb} \alpha_b \to \alpha_a \otimes_{Ga} GX)$$
is unitary.

A \emph{3-morphism} $m$ is simply a modification $m = (m_a)_{a \in \fX}$ between the underlying 2-natural transformations. Notice that modifications between $\dag$-2-transformations admit an adjoint $m^\dag = (m^\dag_a)_{a \in \fX}$ due to each $\alpha_X$ being unitary. In Notation \ref{nota:overlaygraphicalcalculus}, we will provide an overlay graphical calculus for these morphisms between pre-3-Hilbert spaces.

We denote the unitary 2-category of 1-morphisms by $\Hom(\fX\to \fY)$.
\end{defn}

\begin{rem}
The weight $\Psi$ on $\fX$ determines the UAF $\vee$ on $\fX$.
That is, by \cite[Prop.~4.9]{MR5078555}, given two UAFs $\vee,\curlyvee$ on $\fX$ that are spherical under $\Psi$, the canonical isomorphism
$$
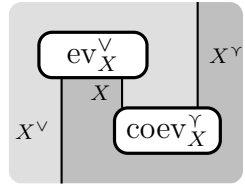
\tikzmath{
\begin{scope}
\clip[rounded corners=5pt] (-1.1,-.7) rectangle (2,1.7);
\fill[lightgray!50] (-.4,-.7) -- (-.4,.7) -- (.4,.7) -- (.4,.3) -- (1.4,.3) -- (1.4,1.7) -- (-1.1,1.7) -- (-1.1,-.7);
\fill[lightgray] (-.4,-.7) -- (-.4,.7) -- (.4,.7) -- (.4,.3) -- (1.4,.3) -- (1.4,1.7) -- (2,1.7) -- (2,-.7);
\end{scope}
\draw[thick] (.4,.3) --node[left]{$\scriptstyle X$} (.4,.7);
\draw[thick] (-.4,-.7) --node[left]{$\scriptstyle X^\vee$} (-.4,.7);
\draw[thick] (1.4,.3) --node[right]{$\scriptstyle X^\curlyvee$} (1.4,1.7);
\roundNbox{fill=white}{(0,1)}{.3}{.4}{.43}{$\ev_X^\vee$}
\roundNbox{fill=white}{(1,0)}{.3}{.4}{.4}{$\coev_X^{\curlyvee}$}
}
$$
is unitary, and thus $\vee\cong \curlyvee$.
This fact categorifies the following statement:
if $*,\dag$ are two $*$-structures on a $\bbC$-algebra $A$ and $\Tr\colon  A\to \bbC$ is positive definite for both $*$ and $\dag$ such that $\Tr(a^*a)=\Tr(a^\dag a)$ for all $a\in A$, then $*=\dag$.
(Indeed, by polarization, $\Tr(a^*b)=\Tr(a^\dag b)$ for all $b\in A$, and thus $a^*=a^\dag$.)
\end{rem}

\begin{defn}
The correct notion of equivalence for pre-3-Hilbert spaces is \emph{isometric equivalence}, i.e., a 1-morphism $F\colon  \fX\to \fY$
whose underlying dagger 2-functor is an equivalence of categories
which preserves the spherical weights. We denote isometric equivalence by $\fX\isomeq \fY$.

Again by \cite[Prop.~4.9]{MR5078555}, an isometric 1-morphism between pre-3-Hilbert spaces automatically preserves UAFs, meaning that the canonical map
\eqref{eq:IsometricIsUAFPreserving}
is automatically unitary.
(Indeed, $\Psi^\fY$ is a spherical weight for both $F(\vee_\fX(-))$ and $\vee_\fY$ on the image of $\fX\subset\fY$, whence \cite[Prop.~4.9]{MR5078555} gives the result.)
We use the notation $F\colon  \fX\isometry \fY$ to denote an isometry.
\end{defn}

\begin{ex}
\label{ex:UnitaryAdjunction}
The unitary 2-category $2\Hilb$ has a canonical pre-3-Hilbert space structure where the UAF is given by \emph{unitary adjunction} \cite[\S2]{MR4750417}, and the spherical weight is given by the formula
$$
\Psi_{(\cA, \Tr^\cA)}(\eta\colon  F\Rightarrow F)
\coloneq
\sum_{a\in\Irr\cA} d_a \Tr^\cA_{F(a)}(\eta_a)
\qquad\qquad\qquad
\forall\, F\colon  \cA\to\cA.
$$
The unitary adjoint of a dagger functor $F\colon  (\cA,\Tr^\cA)\to (\cB,\Tr^\cB)$ is a functor $G\colon  (\cB,\Tr^\cB)\to (\cA,\Tr^\cA)$ together with unitary natural isomorphisms
$$
\cB(F(a)\to b)\isomeq \cA(a\to G(b))
\qquad\qquad\qquad\qquad
\forall\,a\in \cA,\, b\in\cB.
$$
Unitary adjoints exist and are unique up to unique unitary isomorphism.
In fact, an explicit formula for the unitary adjoint of $F\colon \cA\to \cB$ is given by
$$
F^*(b)\coloneq 
\bigoplus_{a\in\pi_0\cA} 
\cB(F(a)\to b)\odot_{\Omega_a} a
=
\bigoplus_{a\in\pi_0\cA} 
\langle F(a)|b\rangle_{\Hilb}^\cB\odot_{\Omega_a} a.
$$
\end{ex}

\begin{ex}
The prototypical example of a pre-3-Hilbert space is the delooping of an $\rmH^*$-\emph{multifusion category}, which is a unitary multifusion category $\cC$ equipped with a \emph{unitary dual functor} (UDF) $\vee$ and a \emph{spherical trace} $\psi\colon  \End_\cC(1_\cC)\to \bbC$ in the sense of \cite{MR4133163}.
The UDF is exactly a UAF on the delooping $\rmB \cC$ (unitary 2-category with one object),
and the spherical trace $\psi$ exactly gives a spherical weight $\Psi$.

Conversely, given any pre-3-Hilbert space $(\fX,\vee,\Psi)$, 
every object $a\in \fX$
canonically gives an $\rmH^*$-multifusion category $\Omega_a\coloneq\End_\fX(a)$ by restricting $\vee$ and $\Psi$.
\end{ex}

\begin{rem} \label{rem:2Hilbvaluedinnerprod}
By \cite[Rem.~4.6]{MR5078555}, a pre-3-Hilbert space $\fX$ is organically enriched over $2\Hilb$.
One gets a faithful positive trace on each $\fX(a\to b)$ by
$$
\Tr^\fX_{X}(f)\coloneq\Psi^\fX_a
\left(
\tikzmath{
\fill[fill=lightgray!50,rounded corners=5] (-.55,-.7) rectangle (1.1,.7);
\draw[thick,fill=lightgray] (-.15,0) arc (180:360+180:.5);
\node at (-.15,-.5) {$\scriptstyle X$};
\roundNbox{fill=white}{(0,0)}{.3}{0}{0}{$f$};
}
\right)
=
\Psi^\fX_b
\left(
\tikzmath{
\fill[fill=lightgray,rounded corners=5] (-.3,-.7) rectangle (1.35,.7);
\draw[thick,fill=lightgray!50] (-.05,0) arc (180:360+180:.5);
\node at (.95,-.5) {$\scriptstyle X$};
\roundNbox{fill=white}{(0.8,0)}{.3}{0}{0}{$f$};
}
\right)
\qquad\qquad\qquad
\forall f\in \End_\fX({}_aX_b).
$$
\end{rem}

\begin{ex}
\label{ex:HomIsTensor}
For the 3-Hilbert space $2\Hilb$, a special case of \cite[Prop.~3.14]{MR5078555} gives an isometric equivalence
$$
\Hom(\cA\to \cB) \cong^\dag \cB\boxtimes \cA^{\op}
$$
when $\Hom(\cA\to \cB)$ is equipped with the spherical weight $\Psi$ from Example \ref{ex:UnitaryAdjunction}.
\end{ex}

A pre-3-Hilbert space structure on $\fX$ allows us to define the notion of a (co)isometry internal to $\fX$.

\subsection{Isometries and Hilbert direct sums}

\begin{defn}
A 1-morphism ${}_aX_b$ in a pre-3-Hilbert space $\fX$ is called an \emph{isometry} if 
$$
\tikzmath{
\fill[lightgray!50,rounded corners=5pt]
(-.6,-.6) rectangle (.6,.6);
\draw[fill=lightgray, thick] (0,0) circle (.3);
}
=
\tikzmath{
\fill[lightgray!50,rounded corners=5pt]
(-.6,-.6) rectangle (.6,.6);
}
\qquad\text{and}\qquad
\tikzmath{
\fill[lightgray!50,rounded corners=5pt]
(-.6,-.6) rectangle (.6,.6);
\draw[fill=lightgray, thick] (-.3,-.6) arc(180:0:.3cm);
\draw[fill=lightgray, thick] (-.3,.6) arc(-180:0:.3cm);
}
=
\tikzmath{
\fill[lightgray!50,rounded corners=5pt]
(-.6,-.6) rectangle (.6,.6);
\fill[lightgray] (-.3,-.6) rectangle (.3,.6);
\draw[thick] (-.3,-.6) -- (-.3,.6);
\draw[thick] (.3,-.6) -- (.3,.6);
}
\quad\underset{\text{\cite[Rem.~4.12]{MR5078555}}}{\Longleftrightarrow}\quad
\tikzmath{
\fill[lightgray,rounded corners=5pt]
(-.6,-.6) rectangle (.6,.6);
\draw[fill=lightgray!50, thick] (-.3,-.6) arc(180:0:.3cm);
\draw[fill=lightgray!50, thick] (-.3,.6) arc(-180:0:.3cm);
}
=
\tikzmath{
\fill[lightgray,rounded corners=5pt]
(-.6,-.6) rectangle (.6,.6);
\fill[lightgray!50] (-.3,-.6) rectangle (.3,.6);
\draw[thick] (-.3,-.6) -- (-.3,.6);
\draw[thick] (.3,-.6) -- (.3,.6);
}\,,
$$
i.e., $\coev_X\colon 1_a\to {}_aX\otimes_b X^\vee_a$ is unitary and $\ev_X\colon {}_bX^\vee\otimes_a X_b\to 1_b$ is an isometry.
We call ${}_aX_b$ a \emph{coisometry} if ${}_bX^\vee_a$ is an isometry.
It is straightforward to check that ${}_aX_b$ is both an isometry and coisometry if and only if it is an adjoint equivalence under the UAF $\vee$, which we call an \emph{isometric equivalence}. 
\end{defn}

\begin{ex}[{\cite[Ex.~4.13]{MR5078555}}]
The isometries in the pre-3-Hilbert space $2\Hilb$ of 2-Hilbert spaces are the fully faithful isometric functors $F\colon (\cA,\Tr^\cA)\hookrightarrow (\cB,\Tr^\cB)$, i.e. 
$$
\Tr^\cB_{F(a)}(F(f))=\Tr^\cA_a(f) \qquad\qquad\qquad \forall f\in \End_\cA(a).
$$
\end{ex}

The notion of isometry was designed to give an elegant definition of the Hilbert direct sum of objects in a pre-3-Hilbert space.

\begin{defn}
\label{defn:HilbertDirectSum}
A \emph{Hilbert direct sum} of $a,b\in \fX$ consists of an object $a\boxplus b\in \fX$ and isometries $I\colon a\hookrightarrow a\boxplus b$ and $J\colon  b\hookrightarrow a\boxplus b$ such that 
$$
\ev_I\ev_I^\dag + \ev_J\ev_J^\dag 
=
\tikzmath{
\fill[lightgray!30,rounded corners=5pt]
(-.6,-.6) rectangle (.6,.6);
\draw[fill=lightgray!65, thick] (0,0) circle (.3);
\node at (0,0) {$\scriptstyle a$};
}
+
\tikzmath{
\fill[lightgray!30,rounded corners=5pt]
(-.6,-.6) rectangle (.6,.6);
\draw[fill=lightgray, thick] (0,0) circle (.3);
\node at (0,0) {$\scriptstyle b$};
}
=
\tikzmath{
\fill[lightgray!30,rounded corners=5pt]
(-.6,-.6) rectangle (.6,.6);
}
=
\id_{1_{a\boxplus b}}.
$$

The Hilbert direct sum is unique up to unique isometric equivalence when it exists.
Moreover, all 1-morphisms of pre-3-Hilbert spaces preserve Hilbert direct sums by \cite[Lem.~4.23]{MR5078555}.

There is a formal \emph{Hilbert direct sum completion} $\Hilb_{\boxplus}(\fX)$ which comes with an isometric fully faithful 1-morphism $\iota\colon \fX \hookrightarrow \Hilb_{\boxplus}(\fX)$
satisfying the following universal property: \begin{itemize} 
\item for every $\fY$ which admits Hilbert direct sums, precomposition with $\iota$ is a
pointwise isometric equivalence of unitary 2-categories
\[
\Hom(\Hilb_{\boxplus}(\fX)\to \fY)
\xrightarrow{-\circ \iota}
\Hom(\fX\to \fY).
\]
\end{itemize}
In more detail, if 1-morphisms $F_1,F_2\colon  \Hilb_\boxplus(\fX)\to \fY$ both extend $F\colon \fX\to \fY$, then the unique equivalence $\eta\colon  F_1\Rightarrow F_2$ that maps to $\id_F$ is
pointwise isometric,\footnote{That this $\eta$ is pointwise isometric was missed in \cite[Prop.~4.28]{MR5078555}.} 
i.e., its components $\eta_x\in \fY(F_1(x)\to F_2(x))$ are isometries in $\fY$:
$$
F_1(x)\isomeq \bigboxplus F(x_i) \isomeq F_2(x).
$$
Moreover, this equivalence preserves the full subcategories of isometries \cite[Prop.~4.28]{MR5078555}.
\end{defn}

\begin{ex}
The pre-3-Hilbert space $2\Hilb$ admits Hilbert direct sums, which is given by $(\cA,\Tr^\cA)\boxplus (\cB,\Tr^\cB)\coloneq(\cA\oplus\cB, \Tr^\cA\oplus \Tr^\cB)$.
\end{ex}

\subsection{\texorpdfstring{$\rmH^*$}{H*}-monads}
\label{subsec:H*monads}
\begin{defn}\label{defn:H*monad}
An $\rmH^*$-\emph{monad} in a pre-3-Hilbert space $\fX$ is an object $a\in\fX$ and a 1-morphism ${}_aA_a\in \End_\fX(a)$ equipped with an associative multiplication $\mu\colon  {}_aA\otimes_aA_a\to {}_aA_a$
and a unit $\iota\colon  1_a\to {}_aA_a$ satisfying the following axioms:
\begin{enumerate}[label=($\rmH^*$\arabic*)]
\item
\label{H:Frobenius}
($\rmC^*$-Frobenius)
$\mu^\dag$ is an $A$–$A$-bimodule map, i.e.,
\[
\tikzmath{
    \fill[\AColor, rounded corners=5pt] (-.3,-.6) rectangle (1.5,.6);
    \draw[\AsColor,thick] (0,-.6) -- (0,0) arc (180:0:.3cm) arc (-180:0:.3cm) -- (1.2,.6);
    \draw[\AsColor,thick] (.3,.3) -- (.3,.6);
    \draw[\AsColor,thick] (.9,-.3) -- (.9,-.6);
}
=
\tikzmath{
    \fill[\AColor, rounded corners=5pt] (-.3,0) rectangle (.9,1.2);
    \draw[\AsColor,thick] (0,0) arc (180:0:.3cm);
    \draw[\AsColor,thick] (0,1.2) arc (-180:0:.3cm);
    \draw[\AsColor,thick] (.3,.3) -- (.3,.9);
}
=
\tikzmath{
    \fill[\AColor, rounded corners=5pt] (-.3,.6) rectangle (1.5,-.6);
    \draw[\AsColor,thick] (0,.6) -- (0,0) arc (-180:0:.3cm) arc (180:0:.3cm) -- (1.2,-.6);
    \draw[\AsColor,thick] (.3,-.3) -- (.3,-.6);
    \draw[\AsColor,thick] (.9,.3) -- (.9,.6);
}
\quad\text{where}\;
a= \tikzmath{
    \tikzset{scale=1.25}
    \fill[\AColor,rounded corners](0,0)rectangle(.5,.5);
},
\;
{}_aA_a= \tikzmath{
    \tikzset{scale=1.25}
    \fill[\AColor,rounded corners](0,0)rectangle(.5,.5);
    \draw[\AsColor,thick](0.25,0)--(0.25,.5);
},
\;
\mu=\tikzmath{
    \tikzset{scale=1.25,yscale=-1}
    \fill[\AColor,rounded corners](-.05,0)rectangle(.65,.5);
    \draw[\AsColor,thick](0.3,0)--(0.3,.25);
    \draw[\AsColor,thick](0.05,.5) arc(-180:0:.25);
},
\;
\mu^\dag=
 \tikzmath{
    \tikzset{scale=1.25}
    \fill[\AColor,rounded corners](-.05,0)rectangle(.65,.5);
    \draw[\AsColor,thick](0.3,0)--(0.3,.25);
    \draw[\AsColor,thick](0.05,.5) arc(-180:0:.25);
}.
\]
\item\label{H:Separable} (Separable) The endomorphism $\mu\mu^\dag$ of $_aA_a$ is invertible, i.e.,
\[
\tikzmath{
    \tikzset{scale=0.3}
    \coordinate (O) at (0,0);
    \fill[\AColor,rounded corners] (O) rectangle ++(3,4);
    \draw[\AsColor,thick] (O) ++(1.5,0) -- ++(0,1);
    \draw[\AsColor,thick] (O) ++(1.5,3) -- ++(0,1);
    \draw[thick,fill=\AColor, draw=\AsColor](O)++(1.5,2) circle(1);
}
\in\End_{\fX}(A)^\times.
\]
\item\label{H:Standard} (Standard) For all endomorphisms $f \in \End_{\fX}({}_aA_a)$, 
$$
\Psi_a\!\left(
\tikzmath{
\tikzset{scale=-1}
\begin{scope}[scale=0.8]
\fill[\AColor, rounded corners=5pt] (-0.5,1.2) rectangle (1.3,-1.2);
\draw[\AsColor,thick] (.4,.7) -- (.4,1);
\filldraw[\AsColor] (.4,1) circle (.05cm);
\draw[\AsColor,thick] (.8,.3) arc (0:180:.4cm) -- +(0,-.6) arc (-180:-90:.4) coordinate (x1) arc (-90:0:.4) -- +(0,.6);
\draw[\AsColor,thick] (x1) -- +(0,-.3) coordinate (x2);
\filldraw[\AsColor] (x2) circle (.05cm);
\end{scope}
\roundNbox{fill=white}{(0,0)}{.275}{0}{0}{$f$}
}
\right)
=
\Psi_a\!\left(
\tikzmath{
\begin{scope}[scale=0.8]
\fill[\AColor, rounded corners=5pt] (-0.5,1.2) rectangle (1.3,-1.2);
\draw[\AsColor,thick] (.4,.7) -- (.4,1);
\filldraw[\AsColor] (.4,1) circle (.05cm);
\draw[\AsColor,thick] (.8,.3) arc (0:180:.4cm) -- +(0,-.6) arc (-180:-90:.4) coordinate (x1) arc (-90:0:.4) -- +(0,.6);
\draw[\AsColor,thick] (x1) -- +(0,-.3) coordinate (x2);
\filldraw[\AsColor] (x2) circle (.05cm);
\end{scope}
\roundNbox{fill=white}{(0,0)}{.275}{0}{0}{$f$}
}
\right)
\qquad\quad\text{where}\qquad 
\tikzmath{
    \tikzset{scale=1.25}
    \fill[\AColor,rounded corners](0,0)rectangle(.5,.5);
    \draw[\AsColor,thick](0.25,.5)--(.25,.25);
    \filldraw[\AsColor](.25,.25)circle(0.05cm);
}
=\iota,
\qquad
\tikzmath{
    \tikzset{scale=1.25}
    \tikzset{yscale=-1}
    \fill[\AColor,rounded corners](0,0)rectangle(.5,.5);
    \draw[\AsColor,thick](0.25,.5)--(.25,.25);
    \filldraw[\AsColor](.25,.25)circle(0.05cm);
}
=\iota^\dag.
$$
Equivalently (cf.~\cite[Facts~3.19, (A5)]{MR5078555}), the morphisms
$$
\tikzmath{
\fill[\AColor, rounded corners=5pt] (1.8,-.3) rectangle (-.5,1.2);
\draw[\AsColor,thick] (0,-.3) --node[left]{$\scriptstyle A$} (0,.3) arc (180:0:.3cm) arc (-180:0:.3cm) --node[right]{$\scriptstyle A^\vee$} (1.2,1.2);
\draw[\AsColor,thick] (.3,.9) -- (.3,.6);
\filldraw[\AsColor] (.3,.9) circle (.05cm);
}
\qquad\qquad\text{and}\qquad\qquad
\tikzmath{
\fill[\AColor, rounded corners=5pt] (-1.8,-.3) rectangle (.5,1.2);
\draw[\AsColor,thick] (0,-.3) --node[right]{$\scriptstyle A$} (0,.3) arc (0:180:.3cm) arc (0:-180:.3cm) --node[left]{$\scriptstyle A^\vee$} (-1.2,1.2);
\draw[\AsColor,thick] (-.3,.9) -- (-.3,.6);
\filldraw[\AsColor] (-.3,.9) circle (.05cm);
}
$$
are unitary (and thus equal).
\end{enumerate}
A \emph{splitting} of an $\rmH^*$-monad $({}_aA_a,\mu,\iota)$ is a pair $({}_aX_b,\gamma)$ where ${}_aX_b$ is a 1-morphism such that $\ev_X\ev_X^\dag \in \Omega^2_b$ is invertible and $\gamma\colon {}_aA_a\Rightarrow {}_aX\otimes_b X^\vee_a$ is a unitary monad isomorphism.\footnote{The symbol $\gamma$ in a splitting is meant to resemble a \emph{zipper}, where the ${}_aA_a$ string `unzips' into ${}_aX\otimes_b X^\vee_a$.
}
\end{defn}

\begin{rem}
Diagrammatically, if we write
\[
a=\tikzmath{
\tikzset{scale=0.7}
    \coordinate(O) at (0,0); 
    \fill[\AColor,rounded corners=5pt](O)rectangle ++(1,1);
},
\qquad\qquad
b=\tikzmath{
\tikzset{scale=0.7}
    \coordinate(O) at (0,0); 
    \fill[\Asplitcolor,rounded corners=5pt](O)rectangle ++(1,1);
},
\qquad\qquad
X=\tikzmath{
    \tikzset{scale=0.7}
    \coordinate(O) at (0,0); 
    \clip[rounded corners](O)rectangle ++(1,1);
    \fill[\AColor](O)rectangle ++(1,1);                    
    \fill[\Asplitcolor](O)++(0.5,0)rectangle ++(0.5,1);         
    \draw[\XsColor,thick](O)++(0.5,0)--++(0,1);                 
},
\qquad\qquad
\gamma=\tikzmath{
    \tikzset{scale=.7,yscale=-1}
    \coordinate(O) at (0,0); 
    \fill[\AColor,rounded corners](O)++(-0.25,0)rectangle ++(1.5,1);         
    \draw[\AsColor,thick](O)++(0.5,0.5)--++(0,0.5);
    \fill[\Asplitcolor](O)++(1,0) arc(0:180:0.5) --cycle;
    \draw[\XsColor,thick] (O)++(1,0) arc(0:180:0.5);
},
\]
then to say $(X,\gamma)$ splits $(A,\mu,\iota)$ means that
\[
\tikzmath{
\tikzset{scale=0.3}
    \coordinate (O) at (0,0);
    \fill[\Asplitcolor, rounded corners=5pt] (O) rectangle ++(3,3);
    \draw[thick,fill=\AColor, draw=\XsColor] (O)++(1.5,1.5) circle (0.8); 
}
\in
(\Omega^2_b)^\times,
\]
that $\gamma$ is unitary, i.e.,
\[
\tikzmath{
    \tikzset{scale=0.3}
    \coordinate (O) at (0,0);
    \fill[\AColor,rounded corners] (O) rectangle ++(3,4);
    \draw[\AsColor,thick] (O) ++(1.5,0) -- ++(0,1);
    \draw[\AsColor,thick] (O) ++(1.5,3) -- ++(0,1);
    \draw[thick,fill=\Asplitcolor, draw=\XsColor](O)++(1.5,2) circle(1);
}
=
\tikzmath{
\tikzset{scale=0.3}
    \coordinate (O) at (0,0);
    \fill[\AColor,rounded corners] (O) rectangle ++(2,4);
    \draw[\AsColor,thick] (O) ++(1,0) -- ++(0,4);
},
\qquad
\tikzmath{
\tikzset{scale=0.3}
    \coordinate (O) at (0,0);
    \fill[\AColor,rounded corners] (O) rectangle ++(4,4); 
    \fill[\Asplitcolor] (O) ++(1,4) arc(-180:0:1) -- cycle; 
    \fill[\Asplitcolor] (O) ++(1,0) arc(180:0:1) -- cycle; 
    \draw[\AsColor,thick] (O) ++(2,1)--++(0,2);
    \draw[\XsColor,thick] (O) ++(1,4) arc(-180:0:1);
    \draw[\XsColor,thick] (O) ++(1,0) arc(180:0:1);
}
=
\tikzmath{
\tikzset{scale=0.3}
    \coordinate (O) at (0,0);
    \fill[\AColor,rounded corners] (O) rectangle ++(4,4); 
    \fill[\Asplitcolor] (O) ++(1.2,0) rectangle ++(1.6,4); 
    \draw[\XsColor,thick] (O) ++(1.2,0) -- ++(0,4); 
    \draw[\XsColor,thick] (O) ++(2.8,0) -- ++(0,4);
},
\]
and that $\gamma$ is compatible with the canonical monad structure on ${}_aX\otimes_b X^\vee_a$, i.e.,
\[
\tikzmath{
    \tikzset{scale=0.3}
    \coordinate (O) at (0,0);
    \fill[\AColor,rounded corners] (O) rectangle ++(4,4); 
    \fill[\Asplitcolor] (O) ++(1,4) arc(-180:0:1) --cycle; 
    \draw[\AsColor,thick] (O) ++(1,0) arc(180:0:1); 
    \draw[\AsColor,thick] (O) ++(2,1)--++(0,2); 
    \draw[\XsColor,thick] (O) ++(1,4) arc(-180:0:1); 
}
=
\tikzmath{
    \tikzset{scale=0.25}
    \coordinate (O) at (0,0);
    \fill[\AColor,rounded corners] (O) rectangle ++(6,5);
    \fill[\Asplitcolor,rounded corners] (O) ++(1.8,5) to[out=-90,in=90] ++(-1, -3.2) arc(-180:0:0.7) -- ++(0, 0.2) arc(180:0:0.8) -- ++(0, -0.2) arc(-180:0:0.7) to[out=90,in=-90] ++(-1, 3.2) -- cycle;
    \draw[\AsColor,thick] (O) ++(1.5,0) -- ++(0,1.1); 
    \draw[\AsColor,thick] (O) ++(4.5,0) -- ++(0,1.1);
    \draw[\XsColor,thick] (O) ++(1.8,5) to[out=-90,in=90] ++(-1, -3.2) arc(-180:0:0.7) -- ++(0,0.2);
    \draw[\XsColor,thick] (O) ++(4.2,5) to[out=-90,in=90] ++(1, -3.2) arc(0:-180:0.7) -- ++(0,0.2);
    \draw[\XsColor,thick] (O) ++(2.2,2.0) arc(180:0:0.8);            
}
\qquad\text{and}\qquad
\tikzmath{
    \tikzset{scale=0.3}
    \coordinate (O) at (0,0);
    \fill[\AColor,rounded corners] (O) rectangle ++(4,4);
    \fill[\Asplitcolor] (O) ++(1,3) arc(-180:0:1) --++(0,1)--++(-2,0)--cycle; 
    \filldraw[\AsColor](2,1)circle(0.125cm);
    \draw[\AsColor,thick] (O) ++(2,2)--++(0,-1);
    \draw[\XsColor,thick] (O) (1,4)--++(0,-1) arc(-180:0:1) --++(0,1);
}
=
\tikzmath{
    \tikzset{scale=0.3}
    \coordinate (O) at (0,0);
    \fill[\AColor,rounded corners] (O) rectangle ++(4,4);
    \fill[\Asplitcolor] (O) ++(1,1.75) arc(-180:0:1) --++(0,2.25)--++(-2,0)--cycle; 
    \draw[\XsColor,thick] (O) (1,4)--++(0,-2.25) arc(-180:0:1) --++(0,2.25);
}
.
\]
\end{rem}

\begin{rem}
\label{rem:ModuleFromSplitting}
    As noted in \cite[Rmk. 4.36]{MR5078555}, if $({}_aX_b,\gamma)$ splits the $\rmH^*$-monad $({}_aA_a,\mu,\iota)$, then $X$ is a left $A$-module and $X^\vee$ is a right $A$-module with actions
    \[
    \tikzmath{
    \tikzset{scale=0.6}
        \fill[\AColor,rounded corners](0,0)rectangle(2,2);
        \begin{scope}
            \clip(1.25,0) rectangle (2,2);
            \fill[\Asplitcolor,rounded corners](.75,0)rectangle(2,2);
        \end{scope}
        \draw[\AsColor,thick](0.5,0)to[out=90,in=220](1.25,1);
        \draw[\XsColor,thick](1.25,0)--++(0,2);
    }
    \coloneq   
    \tikzmath{
    \tikzset{scale=0.6,xscale=1.25}
        \def\snakepath{(1.5,0)--++(0,1.15)arc(0:180:0.25)arc(0:-180:0.25)[rounded corners]--++(0,.85)}
        \fill[\AColor,rounded corners](0,0)rectangle(2,2);
        \begin{scope}
            \clip(1.5,0) rectangle (2,2);
            \fill[\Asplitcolor,rounded corners](.75,0)rectangle(2,2);
            \fill[\Asplitcolor,rounded corners](.75,0)rectangle(2,2);
        \end{scope}
        \fill[\Asplitcolor]\snakepath--(2,2)--(2,0)--cycle;
        \draw[\AsColor,thick](0.75,0)--(0.75,0.9);
        \draw[\XsColor,thick]\snakepath;
    }
    \qquad\text{and}\qquad
    \tikzmath{
        \tikzset{scale=0.6,xscale=-1}
        \fill[\AColor,rounded corners](0,0)rectangle(2,2);
        \begin{scope}
            \clip(1.25,0) rectangle (2,2);
            \fill[\Asplitcolor,rounded corners](.75,0)rectangle(2,2);
        \end{scope}
        \draw[\AsColor,thick](0.5,0)to[out=90,in=220](1.25,1);
        \draw[\XsColor,thick](1.25,0)--++(0,2);
    }
    \coloneq
    \tikzmath{
    \tikzset{scale=0.6,xscale=-1.25}
        \def\snakepath{(1.5,0)--++(0,1.15)arc(0:180:0.25)arc(0:-180:0.25)[rounded corners]--++(0,.85)}
        \fill[\AColor,rounded corners](0,0)rectangle(2,2);
        \begin{scope}
            \clip(1.5,0) rectangle (2,2);
            \fill[\Asplitcolor,rounded corners](.75,0)rectangle(2,2);
            \fill[\Asplitcolor,rounded corners](.75,0)rectangle(2,2);
        \end{scope}
        \begin{scope}[xshift=0mm]
        \fill[\Asplitcolor]\snakepath--(2,2)--(2,0)--cycle;
        \draw[\AsColor,thick](0.75,0)--(0.75,0.9);
        \draw[\XsColor,thick]\snakepath;
        \end{scope}
    }.
    \]
Moreover, it is straightforward to show the following relation regarding the ${}_aA_a$-bubble:
$$
\tikzmath{
    \tikzset{scale=0.3}
    \coordinate (O) at (0,0);
    \fill[\AColor,rounded corners] (O) rectangle ++(3,4);
    \draw[\AsColor,thick] (O) ++(1.5,0) -- ++(0,1);
    \draw[\AsColor,thick] (O) ++(1.5,3) -- ++(0,1);
    \draw[thick,fill=\AColor, draw=\AsColor](O)++(1.5,2) circle(1);
}
=
\tikzmath{
    \tikzset{scale=0.3}
    \coordinate (O) at (0,0);
    \fill[\AColor,rounded corners] (O) rectangle ++(3,4);
    \draw[\AsColor,thick] (O) ++(1.5,0) -- ++(0,1);
    \draw[\AsColor,thick] (O) ++(1.5,3) -- ++(0,1);
    \draw[thick,fill=\Asplitcolor, draw=\XsColor](O)++(1.5,2) circle(1);
    \draw[thick,fill=\AColor, draw=\XsColor](O)++(1.5,2) circle(.5);
}
\,.
$$
\end{rem}

\begin{defn}
Let ${}_aA_a$, ${}_bB_b$, and ${}_cC_c$ be $\mathrm{H}^*$-monads in a pre-3-Hilbert space $\fX$. For a $B$--$A$-bimodule ${}_bX_a$ and an $A$--$C$-bimodule ${}_aY_c$, the \emph{relative tensor product} ${}_aX\otimes_AY_c$ is the $A$--$C$-bimodule given by the 
splitting of the \emph{separability idempotent}
\[
   \tikzmath[scale=0.75]{
    \def\xsteplen{1}
    \fill[gray!10,rounded corners] (-0.3*\xsteplen,0) rectangle (0.5*\xsteplen,2);
    \fill[gray!50,rounded corners] (0.5*\xsteplen,0) rectangle (1.3*\xsteplen,2);
    \fill[\AColor] (0,0) rectangle (\xsteplen,2);
   \draw[red,thick](0,1)--(\xsteplen,1);
   \draw[thick](0,0)--(0,2);
   \draw[thick](\xsteplen,0)--(\xsteplen,2);
   }
   \coloneq
   \tikzmath[scale=0.75]{
    \def\xsteplen{1.5}
   \def\sprad{0.225}
        \fill[gray!10,rounded corners] (-0.3*\xsteplen,0) rectangle (1.3*\xsteplen,2);
        \fill[gray!50,rounded corners] (0.5*\xsteplen,0) rectangle (1.3*\xsteplen,2);
        \fill[\AColor] (0,0) rectangle (\xsteplen,2);
       \draw[thick,red] (.5*\xsteplen,1) circle (\sprad);
       \draw[thick,red] (0,1.75) to[out=-45,in=90] (.5*\xsteplen,{1+\sprad});
       \draw[thick,red] (0.5*\xsteplen,{1-\sprad}) to[out=-90,in=135] (\xsteplen,0.25);
       \node[red,above right] at (.5*\xsteplen,1){\text{\!\scriptsize$-1$}};
       \draw[thick](0,0)--(0,2);
       \draw[thick](\xsteplen,0)--(\xsteplen,2);
    }
    =
   \tikzmath[scale=0.75,yscale=-1]{
    \def\xsteplen{1.5}
    \def\sprad{0.225}
    \fill[gray!10,rounded corners] (-0.3*\xsteplen,0) rectangle (1.3*\xsteplen,2);
    \fill[gray!50,rounded corners] (0.5*\xsteplen,0) rectangle (1.3*\xsteplen,2);
    \fill[\AColor] (0,0) rectangle (\xsteplen,2);
       \draw[thick,red] (.5*\xsteplen,1) circle (\sprad);
       \draw[thick,red] (0,1.75) to[out=-45,in=90] (.5*\xsteplen,{1+\sprad});
       \draw[thick,red] (0.5*\xsteplen,{1-\sprad}) to[out=-90,in=135] (\xsteplen,0.25);
       \node[red,above right] at (.5*\xsteplen,1){\text{\!\scriptsize$-1$}};
       \draw[thick](0,0)--(0,2);
       \draw[thick](\xsteplen,0)--(\xsteplen,2);
    }
    \colon 
    {}_bX\otimes_a Y_c\to {}_bX\otimes_a Y_c
    \quad\text{where}\quad
    \tikzmath{
        \tikzset{scale=0.2}
        \coordinate (O) at (0,0);
        \fill[\AColor,rounded corners] (-1,0) rectangle ++(5,4);
        \draw[\AsColor,thick] (O) ++(1.5,0) -- ++(0,1.425);
        \draw[\AsColor,thick] (O) ++(1.5,2.575) -- ++(0,1.425);
        \draw[thick,fill=\AColor, draw=\AsColor](O)++(1.5,2) circle(0.85);
        \node[red](r) at (3,3){$\scriptstyle r$};
    }
    \coloneq 
    \left(\tikzmath{
        \tikzset{scale=0.235}
        \coordinate (O) at (0,0);
        \fill[\AColor,rounded corners] (-.5,0) rectangle ++(4,4);
        \draw[\AsColor,thick] (O) ++(1.5,0) -- ++(0,1);
        \draw[\AsColor,thick] (O) ++(1.5,3) -- ++(0,1);
        \draw[thick,fill=\AColor, draw=\AsColor](O)++(1.5,2) circle(1.15);
    }\right)^{\!r}\!.
\]
\end{defn}

Just as the Hilbert direct sums of a pair of objects forms a contractible space, so do the splittings of an $\rmH^*$-monad.
Since this fact was unfortunately overlooked in \cite{MR5078555}, we include a short discussion here.

\begin{defn}
Let $({}_aA_a,\mu,\iota)$ be an $\rmH^*$-monad in a pre-3-Hilbert space $(\fX,\vee,\Psi)$. 
Splittings of $A$ form a 2-groupoid $\splitcat(A)$ as follows.
\begin{itemize}
\item 
1-morphisms $({}_aX_b,\gamma)\to ({}_aY_c,\delta)$ are pairs
$({}_bZ_c,\zeta)$ where ${}_bZ_c$ is an isometric equivalence and $\zeta\colon {}_aY_c\Rightarrow {}_aX \otimes_b Z_c$ is a unitary intertwining the left $A$-module actions, or equivalently,
$(1_X\otimes\coev_Z \otimes 1_{X^\vee})\circ\gamma =
(\zeta\otimes \overline{\zeta})\circ \delta$, which in diagrams is
\begin{equation}
\label{eq:Splitting1Morphism}
\tikzmath{
\tikzset{scale=1.25,xscale=1.25}
    \clip[rounded corners] (-.9,-0.6) rectangle (.9,0.9);
    \fill[\acol] (-1.2,-0.6) rectangle (1.2,0.9);
    \draw[\Acol, thick] (0,-0.6) -- (0,-0.3);
    \filldraw[fill=\bcol, draw=\Xcol, thick] (-0.6,0.9) -- node[left, xshift=.5mm, \Xcol]{$\scriptstyle X$} (-0.6,0.3) arc(-180:0:0.6) -- node[right, xshift=-1mm,yshift=.25ex, \Xcol]{$\scriptstyle X^\vee$} (0.6,0.9);
    \filldraw[fill=\ccol, draw=\Zcol, snake, thick] (-0.25,0.9) -- (-0.25,0.4) arc(-180:0:0.25) node[midway, below, yshift=-1mm, \Zcol] {} -- (0.25,0.9);
    \node[\Zcol] at (-0.4, 0.6) {$\scriptstyle Z$};
    \node[\Zcol,xshift=.35ex,yshift=.25ex] at (0.4, 0.6) {$\scriptstyle Z^\vee$};
}
=
\tikzmath{
\tikzset{scale=1.25}
    \clip[rounded corners](-1.1,-.6)rectangle (1.1,.9);
    \fill[\AColor](-1.2,-.6)rectangle (1.2,.9);
    \draw[thick,\AsColor](0,-.6)--(0,-.3);
    \filldraw[draw=\YsColor,thick,fill=\ccol](-.6,.3)arc(-180:0:.6);
    \fill[\ccol](-.6,.3)arc(-90:0:.3)--++(0,.3)--++(.3,0)--++(0,-.6)--cycle;
    \fill[\bcol](-.6,.3)arc(-90:-180:.3)--++(0,.3)--++(.6,0)--++(0,-.3)arc(0:-90:.3);
    \draw[thick,\XsColor] (-.6,.3)arc(-90:-180:.3)--++(0,.3);
    \draw[thick,\Zcol,snake] (-.6,.3)arc(-90:0:.3)--(-.3,.9);
    \begin{scope}[xscale=-1]
        \fill[\ccol](-.6,.3)arc(-90:0:.3)--++(0,.3)--++(.3,0)--++(0,-.6)--cycle;
        \fill[\bcol](-.6,.3)arc(-90:-180:.3)--++(0,.3)--++(.6,0)--++(0,-.3)arc(0:-90:.3);
        \draw[thick,\XsColor] (-.6,.3)arc(-90:-180:.3)--++(0,.3);
        \draw[thick,\Zcol,snake] (-.6,.3)arc(-90:0:.3)--(-.3,.9); 
    \end{scope}
}
\,
\qquad\text{where}\;\quad
\begin{array}{l}
\left(\tikzmath{
    \tikzset{scale=0.7}
    \coordinate(O) at (0,0); 
    \clip[rounded corners](O)rectangle ++(1,1);
    \fill[\acol](O)rectangle ++(1,1);                    
    \fill[\bcol](O)++(0.5,0)rectangle ++(0.5,1);
    \draw[\Xcol,thick](O)++(0.5,0)--++(0,1);
},
\tikzmath{
    \tikzset{scale=.7,yscale=-1}
    \coordinate(O) at (0,0); 
    \fill[\AColor,rounded corners](O)++(-0.25,0)rectangle ++(1.5,1);      
    \draw[\Acol,thick](O)++(0.5,0.5)--++(0,0.5);
    \fill[\bcol](O)++(1,0) arc(0:180:0.5) --cycle;
    \draw[\Xcol,thick] (O)++(1,0) arc(0:180:0.5);
}\right)
=
({}_aX_b,\gamma),
\\[2ex]
\left(\tikzmath{
    \tikzset{scale=0.7}
    \coordinate(O) at (0,0); 
    \clip[rounded corners](O)rectangle ++(1,1);
    \fill[\acol](O)rectangle ++(1,1);             
    \fill[\ccol](O)++(0.5,0)rectangle ++(0.5,1);  
    \draw[\Ycol,thick](O)++(0.5,0)--++(0,1); 
},
\tikzmath{
    \tikzset{scale=.7,yscale=-1}
    \coordinate(O) at (0,0); 
    \fill[\acol,rounded corners](O)++(-0.25,0)rectangle ++(1.5,1);
    \draw[\Acol,thick](O)++(0.5,0.5)--++(0,0.5);
    \fill[\ccol](O)++(1,0) arc(0:180:0.5) --cycle;
    \draw[\Ycol,thick] (O)++(1,0) arc(0:180:0.5);
}\right)
=
({}_aY_c,\delta),
\\[2ex] 
\left(\tikzmath{
    \tikzset{scale=0.7}
    \coordinate(O) at (0,0); 
    \clip[rounded corners](O)rectangle ++(1,1);
    \fill[\bcol](O)rectangle ++(1,1);
    \fill[\ccol](O)++(0.5,0)rectangle ++(0.5,1);             
    \draw[\Zcol,snake,thick](O)++(0.5,0)--++(0,1);
},
\tikzmath{
    \tikzset{scale=.725,yscale=-1}
    \clip[rounded corners](-0.25,0)rectangle (-0.25,0)rectangle ++(1.5,1);         
    \fill[\acol,rounded corners](0,0)++(-0.25,0)rectangle ++(1.5,1);
    \fill[\bcol](1,0) arc(0:180:0.5) --cycle;
    \fill[\ccol](.5,.5)arc(90:0:0.5)--(1.5,0)--(1.5,1)--(.5,1)--cycle;
    \draw[\Ycol,thick](0.5,0.5)--++(0,0.5);
    \draw[\Xcol,thick] (0,0) arc(180:90:0.5);
    \draw[\Zcol,snake,thick] (.5,.5) arc(90:0:0.5);
}\right)
=
({}_bZ_c,\zeta).
\end{array}
\end{equation}
\item 
2-morphisms $({}_bW_c,\omega)\Rightarrow ({}_bZ_c,\zeta)$ are unitaries $u\colon {}_bW_c\Rightarrow {}_bZ_c$ satisfying 
\begin{equation}
\label{eq:Splitting2Morphism}
\zeta 
=
\tikzmath{
    \clip[rounded corners](-.1,0)rectangle(1.5,1.5);
    \fill[\acol,rounded corners](-.25,0)rectangle(1.75,1.5);
    \coordinate(omega) at (.7,.4);
    \coordinate(u) at (1.2,.9);
    \fill[\bcol](u)arc(0:-180:.5)coordinate(q);
    \fill[\bcol](q) rectangle (1.35,1.5-|u);
    \fill[\ccol](0,0-|omega)--(omega)arc(-90:0:.5)--(1.35,1.5-|u)--(1.5,1.5)--(1.5,0)--cycle;
    \draw[\Xcol,thick](omega)arc(270:180:.5)--++(0,.6);
    \draw[thick,\Ycol](omega)--(0,0-|omega);
    \draw[\Zcol,snake,thick](omega)arc(-90:0:.5)--(1.35,1.5-|u);
}
=
\tikzmath{
    \clip[rounded corners](-.1,0)rectangle(1.5,1.5);
    \fill[\acol,rounded corners](-.25,0)rectangle(1.75,1.5);
    \coordinate(omega) at (.7,.4);
    \coordinate(u) at (1.2,.9);
    \fill[\bcol](u)arc(0:-180:.5)coordinate(q);
    \fill[\bcol](q) rectangle (1.35,1.5-|u);
    \fill[\ccol](0,0-|omega)--(omega)arc(-90:0:.5)--(1.35,1.5-|u)--(1.5,1.5)--(1.5,0)--cycle;
    \draw[\Xcol,thick](omega)arc(270:180:.5)--++(0,.6);
    \draw[\Zcol,snake,thick](u)--(1.35,1.5-|u);
    \draw[thick,\Ycol](omega)--(0,0-|omega);
    \draw[\Wcol,snake,thick](omega)arc(-90:0:.5);
    \node[draw=black,circle,very thick,rounded corners,fill=white,inner sep=2pt] at (u){$\scriptstyle u$};
}
= (1_X\otimes u)\circ \omega
\;\;\;\text{where}\;\;
\begin{array}{c}
\!\!\left(\tikzmath{
    \tikzset{scale=0.7}
    \coordinate(O) at (0,0); 
    \clip[rounded corners](O)rectangle ++(1,1);
    \fill[\bcol](O)rectangle ++(1,1);
    \fill[\ccol](O)++(0.5,0)rectangle ++(0.5,1);              
    \draw[\Zcol,snake,thick](O)++(0.5,0)--++(0,1);                       
},
\tikzmath{
    \tikzset{scale=.725,yscale=-1}
    \clip[rounded corners](-0.25,0)rectangle (-0.25,0)rectangle ++(1.5,1);         
    \fill[\acol,rounded corners](0,0)++(-0.25,0)rectangle ++(1.5,1);
    \fill[\bcol](1,0) arc(0:180:0.5) --cycle;
    \fill[\ccol](.5,.5)arc(90:0:0.5)--(1.5,0)--(1.5,1)--(.5,1)--cycle;
    \draw[\Ycol,thick](0.5,0.5)--++(0,0.5);
    \draw[\Xcol,thick] (0,0) arc(180:90:0.5);
    \draw[\Zcol,snake,thick] (.5,.5) arc(90:0:0.5);
}\right)
=
({}_bZ_c,\zeta),
\\[2ex]
\left(\tikzmath{
    \tikzset{scale=0.7}
    \coordinate(O) at (0,0); 
    \clip[rounded corners](O)rectangle ++(1,1);
    \fill[\bcol](O)rectangle ++(1,1);
    \fill[\ccol](O)++(0.5,0)rectangle ++(0.5,1);              
    \draw[\Wcol,snake,thick](O)++(0.5,0)--++(0,1);                       
},
\tikzmath{
    \tikzset{scale=.725,yscale=-1}
    \clip[rounded corners](-0.25,0)rectangle (-0.25,0)rectangle ++(1.5,1);         
    \fill[\acol,rounded corners](0,0)++(-0.25,0)rectangle ++(1.5,1);
    \fill[\bcol](1,0) arc(0:180:0.5) --cycle;
    \fill[\ccol](.5,.5)arc(90:0:0.5)--(1.5,0)--(1.5,1)--(.5,1)--cycle;
    \draw[\Ycol,thick](0.5,0.5)--++(0,0.5);
    \draw[\Xcol,thick] (0,0) arc(180:90:0.5);
    \draw[\Wcol,snake,thick] (.5,.5) arc(90:0:0.5);
}\right)
=
({}_bW_c,\omega).
\end{array}
\end{equation}
\end{itemize}
\end{defn}

\begin{lem}\label{lem:h-star-monad-splittings-unique}
The space of splittings of an $\rmH^*$-monad is contractible.
\end{lem}
\begin{proof}
Observe that \eqref{eq:Splitting2Morphism} immediately implies that at most one 2-morphism exists between parallel 1-morphisms; indeed, $\zeta,\omega$ are unitary and $1_X\otimes -$ is injective as $\ev^\dag_X\ev_X$ is invertible.
Hence it suffices to prove that every two splittings admit a 1-morphism and that parallel 1-morphisms admit a 2-morphism.
\\ 

\item[\underline{$\exists$ 1-morphism:}]
Suppose $({}_aX_b,\gamma)$ and $({}_aY_c,\delta)$ are two splittings of the $\rmH^*$-monad ${}_aA_a$. It is an immediate consequence of \cite[Remarks 4.12, 4.40]{MR5078555} that ${}_bW_c\coloneq{}_bX^\vee\otimes_AY_c$ is an isometric equivalence $b \isomeq c$. One then verifies that 
\[
    \omega\coloneq
    \tikzmath{
        \def\tkzxstep{0}
        \def\tkzlbord{-1.4}
        \tikzset{scale=1.25}
        \clip[rounded corners] (\tkzlbord,1) rectangle (.25,-0.25);
        \fill[\AColor] (\tkzlbord,1) rectangle (\tkzxstep,-0.25);
        \fill[\Bsplitcolor] (\tkzxstep,1) rectangle (0.25,-0.25);
        \draw[thick,\YsColor] (\tkzxstep,-.25) -- (\tkzxstep,1);
        \filldraw[thick, \XsColor, fill=\Asplitcolor] (-.3,1) --(-.3,.7) arc(0:-180:.3cm) --(-.9,1);
        \draw[\AsColor,thick](-.9,.8) to[out=210,in=90] (-1.2,.5) -- ++(0,-.25) arc(90:450:2pt) node[above right, red, xshift=-1mm, yshift=-2mm]{$\scriptscriptstyle 1/2$};
        \draw[\AsColor,thick](-.3,.8)--(0,.8);
    }
    \qquad\text{where}\qquad
    \tikzmath{
        \tikzset{scale=0.2}
        \coordinate (O) at (0,0);
        \fill[\AColor,rounded corners] (-1,0) rectangle ++(5,4);
        \draw[\AsColor,thick] (O) ++(1.5,2.575) -- ++(0,1.425);
        \draw[thick,fill=\AColor, draw=\AsColor](O)++(1.5,2) circle(0.575);
        \node[red](r) at (3,3){$\scriptstyle r$};
    }\coloneq    
    \tikzmath{
        \tikzset{scale=0.2}
        \coordinate (O) at (0,0);
        \fill[\AColor,rounded corners] (-1,0) rectangle ++(5,4);
        \draw[\AsColor,thick] (O) ++(1.5,.75) -- ++(0,0.925);
        \draw[\AsColor,thick] (O) ++(1.5,2.825) -- ++(0,1.175);
        \draw[thick,fill=\AColor, draw=\AsColor](O)++(1.5,2.25) circle(0.575);
        \node[red](r) at (3,3){$\scriptstyle r$};
        \fill[red](1.5,.75)circle(0.3cm);
    }\,,\quad
    r\in\mathbb{R}
\]
is a unitary $A$-module intertwiner.
\\

\item[\underline{$\exists$ 2-morphism:}] Given a 1-morphism $({}_bZ_c,\zeta)$ parallel to our established 1-morphism $ ({}_bX^\vee\otimes_A Y_c,\omega)\colon ({}_aX_b,\gamma) \to ({}_aY_c,\delta)$,
one verifies using \eqref{eq:Splitting1Morphism} that
\[
u\coloneq
\tikzmath{
    \clip[rounded corners] (-1, -0.5) rectangle (1.2, 1.5);
    \fill[\bcol] (-1, -0.5) rectangle (1.2, 1.5);
    \coordinate (zetadag) at (0.5, 0.5);
    \fill[\ccol] (0.5, 1.5) -- (zetadag) arc(90:0:0.5) -- (1.0, -0.5) -- (1.5, -0.5) -- (1.5, 1.5) -- cycle;
    \fill[\acol] (-0.5, 1.5) -- (-0.5, 0) arc(180:360:0.25) arc(180:90:0.5) -- (0.5, 1.5) -- cycle;
    \draw[\Ycol, thick] (zetadag) -- (0.5, 1.5);
    \draw[\Zcol, snake, thick] (1.0, -0.5) -- (1.0, 0) arc(0:90:0.5);
    \draw[\Xcol, thick] (-0.5, 1.5) -- (-0.5, 0) arc(180:360:0.25) arc(180:90:0.5);
    \draw[\AsColor,thick] (-0.5, 0.8) arc(90:0:.3);
    \draw[\Acol,thick] (-0.5,1.1)--(0.5,1.1); arc(90:0:.3);
    \filldraw[\AsColor,fill=\AColor,thick] (-0.2, 0.4) circle(0.1)node[above right,xshift=-1.5mm]{$\scriptscriptstyle-1/2$};
}\colon {}_bZ_c\Rightarrow {}_bX^\vee\otimes_A Y_c
\]
is a unitary satisfying \eqref{eq:Splitting2Morphism}. In more detail, unitarity follows from
\begin{align*}
    u^\dagger u
    &=
\tikzmath{
    \clip[rounded corners=2.5mm] (-1.2, -1.2) rectangle (1.5, 2.4);
    \fill[\Asplitcolor] (-1.2, -1.2) rectangle (1.5, 2.4);
    \begin{scope}
        \coordinate(o) at (0.6, 0.15);
        \fill[\ccol] (1.5, 0.6) -- (0.6, 0.6) -- (o) to[out=0,in=90]++(0.3, -0.3)coordinate(qZ) -- (qZ |- 0, -1.2) -- (1.5, -1.2) -- cycle;
        \fill[\AColor] (0.6, 0.6) coordinate(l) -- (o) arc(90:180:0.45) -- ++(0, -0.15) arc(360:180:0.45) coordinate[pos=0.25](fd) coordinate(qY) -- (qY |- 0, 0.6) -- cycle;
        \draw[\Zcol,thick,snake](o)to[out=0,in=90](qZ) -- (qZ |- 0, -1.2);
        \draw[thick,\XsColor](o)arc(90:180:0.45)--++(0, -0.15)arc(360:180:0.45)--(qY |- 0, 0.6);
        \draw[\AsColor,thick](0.6, 0.3) -- (-0.75, 0.3);
        \draw[\AsColor,thick](fd)--++(-0.15, 0.15)coordinate(end);
        \draw[thick,\YsColor](o)--(l);
        \fill[\AColor,thick,draw=\AsColor](end)circle(2pt)node[above, xshift=-2mm,inner sep=1.5pt,\AsColor]{$\scriptscriptstyle-1/2$};
    \end{scope}
    \begin{scope}
        \coordinate(o) at (0.6, 1.05);
        \fill[\ccol] (1.5, 0.6) -- (0.6, 0.6) -- (o) to[out=0,in=-90]++(0.3, 0.3)coordinate(qZ) -- (qZ |- 0, 2.4) -- (1.5, 2.4) -- cycle;
        \fill[\AColor] (0.6, 0.6) coordinate(l) -- (o) arc(270:180:0.45) -- ++(0, 0.15) arc(0:180:0.45) coordinate[pos=0.25](fd) coordinate(qY) -- (qY |- 0, 0.6) -- cycle;
        \draw[\Zcol,thick,snake](o)to[out=0,in=-90](qZ) -- (qZ |- 0, 2.4);
        \draw[thick,\XsColor](o)arc(270:180:0.45)--++(0, 0.15)arc(0:180:0.45)--(qY |- 0, 0.6);
        \draw[\AsColor,thick](fd)--++(-0.15, -0.15)coordinate(end);
        \draw[\AsColor,thick](0.6, 0.9) -- (-0.75, 0.9);
        \draw[thick,\YsColor](o)--(l);
        \fill[\AColor,thick,draw=\AsColor](end)circle(2pt)node[below=0.5mm,xshift=-1.5mm,inner sep=1.5pt,\AsColor]{$\scriptscriptstyle-1/2$};
    \end{scope}
}
=
\tikzmath{
    \clip[rounded corners=2.5mm] (-1.2, -1.2) rectangle (1.5, 2.4);
    \fill[\Asplitcolor] (-1.2, -1.2) rectangle (1.5, 2.4);
    \begin{scope}
        \coordinate(o) at (0.6, 0.15);
        \fill[\ccol] (1.5, 0.6) -- (0.6, 0.6) -- (o) to[out=0,in=90]++(0.3, -0.3)coordinate(qZ) -- (qZ |- 0, -1.2) -- (1.5, -1.2) -- cycle;
        \fill[\AColor] (0.6, 0.6) coordinate(l) -- (o) arc(90:180:0.45) -- ++(0, -0.15) arc(360:180:0.45) coordinate[pos=0.25](fd) coordinate(qY) -- (qY |- 0, 0.6) -- cycle;
        \draw[\Zcol,thick,snake](o)to[out=0,in=90](qZ) -- (qZ |- 0, -1.2);
        \draw[thick,\XsColor](o)arc(90:180:0.45)--++(0, -0.15)arc(360:180:0.45)--(qY |- 0, 0.6);
        \draw[\AsColor,thick](0.15, -0.15) -- (-0.75, -0.15);
        \draw[\AsColor,thick](fd)--++(-0.15, 0.15)coordinate(end);
        \draw[thick,\YsColor](o)--(l);
        \fill[\AColor,thick,draw=\AsColor](end)circle(2pt)node[above, xshift=-2mm,inner sep=1.5pt,\AsColor]{$\scriptscriptstyle-1/2$};
    \end{scope}
    \begin{scope}
        \coordinate(o) at (0.6, 1.05);
        \fill[\ccol] (1.5, 0.6) -- (0.6, 0.6) -- (o) to[out=0,in=-90]++(0.3, 0.3)coordinate(qZ) -- (qZ |- 0, 2.4) -- (1.5, 2.4) -- cycle;
        \fill[\AColor] (0.6, 0.6) coordinate(l) -- (o) arc(270:180:0.45) -- ++(0, 0.15) arc(0:180:0.45) coordinate[pos=0.25](fd) coordinate(qY) -- (qY |- 0, 0.6) -- cycle;
        \draw[\Zcol,thick,snake](o)to[out=0,in=-90](qZ) -- (qZ |- 0, 2.4);
        \draw[thick,\XsColor](o)arc(270:180:0.45)--++(0, 0.15)arc(0:180:0.45)--(qY |- 0, 0.6);
        \draw[\AsColor,thick](fd)--++(-0.15, -0.15)coordinate(end);
        \draw[\AsColor,thick](0.15, 1.35) -- (-0.75, 1.35);
        \draw[thick,\YsColor](o)--(l);
        \fill[\AColor,thick,draw=\AsColor](end)circle(2pt)node[below=0.5mm,xshift=-1.5mm,inner sep=1.5pt,\AsColor]{$\scriptscriptstyle-1/2$};
    \end{scope}
}
    =
    \tikzmath{
    \clip[rounded corners=2.5mm] (-1.2, -1.05) rectangle (1.2, 2.1);
    \fill[\Asplitcolor,rounded corners] (-1.2, -1.05) rectangle (1.2, 2.4);
    \fill[\ccol] (1.2, -1.05) rectangle (0.6, 2.1);
    \draw[\Zcol,thick,snake] (0.6, -1.05) -- (0.6, 2.4);
    \fill[\AColor] (0.15, -0.3) arc(360:180:0.45) -- (-0.75, 1.35) arc(180:0:0.45) -- cycle;
    \draw[thick,\XsColor] (0.15, -0.3) arc(360:180:0.45) coordinate[pos=0.25](fd) -- (-0.75, 1.35) arc(180:0:0.45) coordinate[pos=0.75](fdTop) -- cycle;
    \draw[\AsColor,thick] (0.15, 0) -- (-0.75, 0);
    \draw[\AsColor,thick] (0.15, 1.05) -- (-0.75, 1.05);
    \draw[\AsColor,thick] (fd) -- ++(-0.15, 0.15) coordinate(end);
    \fill[\AColor,thick,draw=\AsColor] (end) circle(2pt) node[above,xshift=-1.5mm, inner sep=1.5pt, \AsColor] {$\scriptscriptstyle-1/2$};
    \draw[\AsColor,thick] (fdTop) -- ++(-0.15, -0.15) coordinate(endTop);
    \fill[\AColor,thick,draw=\AsColor] (endTop) circle(2pt) node[below=0.5mm,xshift=-1.5mm, inner sep=1.5pt, \AsColor] {$\scriptscriptstyle-1/2$};
}
\underset{\text{\cite[Rem.~4.40]{MR5078555}}}{=}
\tikzmath{
    \clip[rounded corners=2.5mm] (-0.6, -1.05) rectangle (0.6, 2.4);
    \fill[\ccol] (-0.6, -1.05) rectangle (0.6, 2.4);
    \fill[\Asplitcolor] (-0.6, -1.05) rectangle (0, 2.4);
    \draw[\Zcol,thick,snake] (0, -1.05) -- (0, 2.4);
}
    =
    \id_{{}_bZ_c}
\\
uu^\dag
&=
\tikzmath{
    \clip[rounded corners=2.5mm] (-1.2, -1.9) rectangle (1.2,1.6);
    \fill[\Asplitcolor] (-1.2, -2.1) rectangle (1.2, 2.1);
    \begin{scope}[yshift=0.3cm]
        \coordinate (zetadag) at (0.45, 0.45);
        \fill[\ccol] (0.45, 1.8) -- (zetadag) arc(90:0:0.45) -- (0.9, -0.6) -- (1.2, -0.6) -- (1.2, 1.8) -- cycle;
        \fill[\acol] (-0.45, 1.8) -- (-0.45, 0.15) arc(180:360:0.3) arc(180:90:0.3) -- (0.45, 1.8) -- cycle;
        \draw[\Ycol, thick] (zetadag) -- (0.45, 1.8);
        \draw[\Zcol, snake, thick] (0.9, -0.6) -- (0.9, 0) arc(0:90:0.45);
        \draw[\Xcol, thick] (-0.45, 1.8) -- (-0.45, 0.15) arc(180:360:0.3) arc(180:90:0.3);
        \draw[\AsColor,thick] (-0.45, 0.75) arc(90:0:0.3);
        \draw[\AsColor,thick] (-0.45, 1.05) -- (0.45, 1.05);
        \filldraw[\AsColor,fill=\AColor,thick] (-0.15, 0.45) circle(0.075) node[above right,xshift=-1.5mm]{$\scriptscriptstyle-1/2$};
    \end{scope}
    \begin{scope}[yshift=-0.5cm,yscale=-1]
        \coordinate (zetadag) at (0.45, 0.45);
        \fill[\ccol] (0.45, 1.8) -- (zetadag) arc(90:0:0.45) -- (0.9, -0.6) -- (1.2, -0.6) -- (1.2, 1.8) -- cycle;
        \fill[\acol] (-0.45, 1.8) -- (-0.45, 0.15) arc(180:360:0.3) arc(180:90:0.3) -- (0.45, 1.8) -- cycle;
        \draw[\Ycol, thick] (zetadag) -- (0.45, 1.8);
        \draw[\Zcol, snake, thick] (0.9, -0.6) -- (0.9, 0) arc(0:90:0.45);
        \draw[\Xcol, thick] (-0.45, 1.8) -- (-0.45, 0.15) arc(180:360:0.3) arc(180:90:0.3);
        \draw[\AsColor,thick] (-0.45, 0.75) arc(90:0:0.3);
        \draw[\AsColor,thick] (-0.45, 1.05) -- (0.45, 1.05);
        \filldraw[\AsColor,fill=\AColor,thick] (-0.15, 0.45) circle(0.075) node[below right,xshift=-1.5mm]{$\scriptscriptstyle-1/2$};
    \end{scope}
}
\underset{\text{\cite[Rem.~4.40]{MR5078555}}}{=}
\tikzmath{
    \clip[rounded corners=2.5mm] (-1.2, -1.9) rectangle (1.2, 1.6);
    \fill[\Asplitcolor] (-1.2, -2.1) rectangle (1.2, 2.1);
    \fill[\ccol] (0.45, 2.1) -- (0.45, -2.1) -- (1.2, -2.1) -- (1.2, 2.1) -- cycle;
    \fill[\acol] (-0.45, 2.1) -- (-0.45, -2.1) -- (0.45, -2.1) -- (0.45, 2.1) -- cycle;
    \draw[\Ycol, thick] (0.45, 2.1) -- (0.45, -2.1);
    \draw[\Xcol, thick] (-0.45, 2.1) -- (-0.45, -2.1);
    \draw[\AsColor,thick] (-0.45, 0.9) -- (0.45, 0.9);
    \draw[\AsColor,thick] (-0.45, -1.1) -- (0.45, -1.1);
}
=
\tikzmath{
    \clip[rounded corners=2.5mm] (-1.2, -1.9) rectangle (1.2, 1.6);
    \fill[\Asplitcolor] (-1.2, -2.1) rectangle (1.2, 2.1);
    \fill[\ccol] (0.45, 2.1) -- (0.45, -2.1) -- (1.2, -2.1) -- (1.2, 2.1) -- cycle;
    \fill[\acol] (-0.45, 2.1) -- (-0.45, -2.1) -- (0.45, -2.1) -- (0.45, 2.1) -- cycle;
    \draw[\Ycol, thick] (0.45, 2.1) -- (0.45, -2.1);
    \draw[\Xcol, thick] (-0.45, 2.1) -- (-0.45, -2.1);
    \draw[\AsColor,thick] (-0.45, -0.1) -- (0.45, -0.1);
}
=
\id_{{}_bX^\vee\otimes_A Y_c},
\end{align*}
while \eqref{eq:Splitting2Morphism} follows from
\[
(1_X\otimes u)\circ \omega
=
\tikzmath{
    \clip[rounded corners] (-1.5, -1.2) rectangle (1.2, 1.5);
    \fill[\AColor] (-2.1,-1.2) rectangle (1.2,1.5); 
    \fill[\ccol] (0.45, 1.5) -- (0.45, 0.45) arc(90:0:0.45) -- (0.9, -0.3) arc(0:-90:0.6) -- (0.3, -1.2) -- (1.2, -1.2) -- (1.2, 1.5) -- cycle;
    \fill[\Asplitcolor] (-1.05,1.5)--(-1.05,0)to[out=-90,in=180] (0.3,-0.9)
    arc(-90:0:0.6) -- (0.9, -0.3) -- (0.9, 0) arc(0:90:0.45) arc(90:180:0.3) arc(360:180:0.3) -- (-0.45, 1.5) -- cycle;
    \coordinate (zetadag) at (0.45, 0.45);
    \coordinate (t) at (0.3, -0.9);
    \draw[\Ycol, thick] (zetadag) -- (0.45, 1.5);
    \draw[\Zcol, snake, thick] (0.9, -0.3) -- (0.9, 0) arc(0:90:0.45);
    \draw[\Zcol, snake, thick] (0.9, -0.3) arc(0:-90:0.6);
    \draw[\Ycol,thick] (t)--(0.3, -1.2);
    \draw[\Xcol,thick] (t) to[out=180,in=-90](-1.05,0) -- (-1.05,1.5);
    \draw[\Xcol, thick] (-0.45, 1.5) -- (-0.45, 0.15) arc(180:360:0.3) arc(180:90:0.3);
    \draw[\AsColor,thick] (-0.45, 0.75) arc(90:0:0.3);
    \draw[\AsColor,thick] (-0.45, 1.05) -- (0.45, 1.05);
    \filldraw[\AsColor,fill=\AColor,thick] (-0.15, 0.45) circle(0.075) node[above right,xshift=-2mm]{$\scriptscriptstyle-1/2$};
}
=
\tikzmath{
    \clip[rounded corners] (-1.5, -1.2) rectangle (1.2, 1.5);
    \fill[\AColor] (-2.1,-1.2) rectangle (1.2,1.5); 
    \fill[\ccol] (0.45, 1.5) -- (0.45, 0.45) arc(90:0:0.45) -- (0.9, -0.3) arc(0:-90:0.6) -- (0.3, -1.2) -- (1.2, -1.2) -- (1.2, 1.5) -- cycle;
    \fill[\Asplitcolor] (-1.05, 1.5)--++(0,-.6) arc(-180:0:0.3)--++(0,.6)--cycle;
    \fill[\Asplitcolor] (0.45, 1.5) -- (0.45, 0.45) arc(90:0:0.45) -- (0.9, -0.3) arc(0:-90:0.6) to[out=180,in=-90](-1.05,0.0) arc(180:0:0.3) -- (-0.45, -0.15) arc(180:360:0.3) -- (0.15, 0.15) arc(180:90:0.3) -- cycle;
    \coordinate (zetadag) at (0.45, 0.45);
    \coordinate (t) at (0.3,-0.9);
    \draw[\Ycol, thick] (zetadag) -- (0.45, 1.5);
    \draw[\Zcol, snake, thick] (0.9, -0.3) -- (0.9, 0) arc(0:90:0.45);
    \draw[\Zcol, snake, thick] (0.9, -0.3) arc(0:-90:0.6);
    \draw[\Ycol,thick] (t)--(0.3, -1.2);
    \draw[\Xcol,thick] (t) to[out=180,in=-90](-1.05,0.0) arc(180:0:0.3);
    \draw[\Xcol, thick] (-0.45, 0.0) -- (-0.45, -0.15) arc(180:360:0.3) -- (0.15, 0.15) arc(180:90:0.3);
    \draw[\Xcol,thick] (-1.05, 1.5)--++(0,-.6) arc(-180:0:0.3)--++(0,.6);
    \draw[\AsColor,thick] (-0.75,0.3) -- (-0.75,0.6);
    
    \draw[\AsColor,thick] (-0.45, 0.9) arc(90:0:0.3);
    \draw[\AsColor,thick] (-0.45, 1.2) -- (0.45, 1.2);
    \filldraw[\AsColor,fill=\AColor,thick] (-0.15, 0.6) circle(0.075) node[above right,xshift=-2mm]{$\scriptscriptstyle-1/2$};
}
=
\tikzmath{
    \clip[rounded corners] (-1.9, -1.2) rectangle (1.2, 1.5);
    \fill[\AColor] (-2.1,-1.2) rectangle (1.5,1.5); 
    \fill[\ccol] (0.3, -1.2) rectangle (1.5,1.5);
    \fill[\Asplitcolor] (-1.05, 1.5)--++(0,-.6) arc(-180:0:0.3)--++(0,.6)--cycle;
    \draw[\Xcol,thick] (-1.05, 1.5)--++(0,-.6) arc(-180:0:0.3)--++(0,.6);
    \draw[\AsColor,thick] (-1.05, 0.9) arc(90:180:0.3)coordinate(q);
    \draw[\AsColor,thick] (-0.45, 1.2) -- (0.3, 1.2);    \filldraw[\AsColor,fill=\AColor,thick] (q) circle(0.075) node[above left,xshift=1mm]{$\scriptscriptstyle1/2$};
    \coordinate (ttop) at (0.3, 0.15);
    \coordinate (t) at (0.3,-0.9);
    \fill[\Asplitcolor]
      (t) arc(-90:-180:0.3)
      -- ++(0,0.45)
      arc(180:90:0.3)
      arc(90:0:0.3)
      -- ++(0,-0.45)
      arc(0:-90:0.3)
      -- cycle;
    \draw[\Ycol, thick] (ttop) -- (0.3, 1.5);
    \draw[\Ycol,thick] (t)--(0.3, -1.2);
    \draw[\Zcol,snake,thick] (t) arc(-90:0:0.3)--++(0,.45)arc(0:90:0.3);
    \draw[\Xcol, thick] (t) arc(-90:-180:0.3)--++(0,.45)arc(180:90:0.3);
}
=
\omega.\qedhere
\]
\end{proof}

\begin{defn} 
There is a formal \emph{$\rmH^*$-monad completion} 
$\HstarAlg(\fX)$
which comes with an isometric 1-morphism $\iota\colon \fX\to \HstarAlg(\fX)$ 
which 
satisfies the following universal property: \begin{itemize} 
\item for every $\fY$ in which all $\rmH^*$-monads split, precomposition with $\iota$ is a
pointwise isometric equivalence of unitary 2-categories
$$
\Hom( \HstarAlg(\fX)\to \fY)
\xrightarrow{-\circ \iota}
\Hom(\fX\to \fY).
$$
\end{itemize}
Again, the pointwise isometric property was missed in \cite[Prop.~4.43]{MR5078555}, but it follows readily from Lemma \ref{lem:h-star-monad-splittings-unique}.
Moreover, this equivalence preserves the full subcategories of isometries \cite[Prop.~4.43]{MR5078555}.
\end{defn} 

\begin{fact}
\label{fact:IsometricPreservesStuff}
Suppose $\fX,\fY$ are pre-3-Hilbert spaces and $F\colon \fX\isometry \fY$ is an isometric $\dag$-functor. 
Then, since $F$ preserves UAFs by \cite[Prop.~4.9]{MR5078555} as in \eqref{eq:IsometricIsUAFPreserving},
$F$ automatically preserves (co)isometries by \cite[Lem.~4.23]{MR5078555}, Hilbert direct sums, $\rmH^*$-monads, and their splittings.
\end{fact}

\subsection{3-Hilbert spaces}

\begin{defn}
A \emph{3-Hilbert space} is a pre-3-Hilbert space that is \emph{complete}, i.e., admits all Hilbert direct sums and all $\rmH^*$-monads split.\footnote{In fact, one can provide an even weaker notion than pre-3-Hilbert space whose iterated completion (all $\dag$,$\vee$-preserving functors into $2\Hilb$) will again be a 3-Hilbert space.
This will appear in \cite{UnitaryDiskLike}.
}
\end{defn}

\begin{ex}\label{ex:ModforH*mFC}
Given an $\rmH^*$-multifusion category $\cC$, 
the unitary 2-category
$\Mod^\dag(\cC)$ of unitary $\cC$-module categories equipped with unitary $\cC$-module traces in the sense of \cite{MR3019263,MR4598730} is a 3-Hilbert space where the UAF is given by unitary adjunction as in Example \ref{ex:UnitaryAdjunction}, and the spherical weight $\Psi$ is given 
by the following formula when $\cC$ is indecomposable:
\begin{equation}\label{modeeqn}
\Psi_{\cM}^{\Mod^\dag(\cC)}(\eta\colon \id_\cM\Rightarrow \id_\cM)
\coloneq
\frac{\dim(\End_\cC(1_\cC))^2}{\FPdim(\cC) \psi_\cC(\id_{1_\cC})}
\sum_{m\in\Irr\cM}
d_m \Tr^\cM_m(\eta_m).
\end{equation}
The formula for arbitrary $\cC$ is then given by decomposing $\cC$ into its indecomposable summands.

In the sequel, we may use the abbreviations
\begin{itemize}
\item 
$d_c\coloneq\psi(\tr_\cC(\id_c))$,
\item 
$D_\cC\coloneq \sum_{c\in\pi_0\cC} d_c^2$,
\item 
$n_\cC\coloneq\dim(\End_\cC(1_\cC))$, and
\item 
$
\delta_\cC\coloneq \psi_\cC(\id_{1_\cC})
$.\footnote{\label{footnote-notation}These renormalization factors also appear in \cite[\S2.4]{2104.02101}.
In that article's notation, these two terms are precisely $N(1_\cC)^2$ and $\mathrm{tr}_s(\id_{1_{\mathcal{C}}})$ respectively, where the latter is defined by the evaluation $\mathrm{ev}(\varnothing_{1_\mathcal{C}})$ of the ``empty'' string diagram $\varnothing_{1_\mathcal{C}}$ on the 2-sphere labeled by $\id_{1_{\mathcal{C}}}$.} 
\end{itemize}
The normalization factor above is introduced to fulfill the following desiderata:
\begin{itemize}
\item 
the $\cC$-module $\dag$-functor $\cM_\cC\mapsto \Fun^\dag_\cC(\cC\to \cM)$ given by $m\mapsto m\lhd -$ is isometric \cite[Prop.~3.9]{MR5078555},
\item 
the natural map $\cM\boxtimes_\cC \cN \to \Fun^\dag_\cC(\cM^{\op}\to \cN)$
is isometric \cite[Prop.~3.14]{MR5078555},
and
\item 
we have an isometric equivalence $\Mod^\dag(\cC)\cong \HstarAlg(\cC)$
\cite[Thm.~4.51]{MR5078555}.
\end{itemize}

\end{ex}

By \cite[Prop.~4.28 and 4.43]{MR5078555}, 
given a pre-3-Hilbert space,
the formal construction 
$\iota\colon \fX\hookrightarrow \fX^\cent\coloneq\HstarAlg(\Hilb_\boxplus(\fX))$ 
is called the \emph{completion} of $\fX$.
The completion
satisfies the universal property that 
for every 3-Hilbert space $\fY$, precomposition with $\iota$ is a pointwise isometric equivalence of unitary 2-categories
\begin{equation}
\label{eq:UniversalPropertyOfCompletion}
\Hom(\fX^\cent\to \fY)
\xrightarrow{-\circ \iota}
\Hom(\fX\to \fY).
\end{equation}
Moreover, this equivalence preserves full subcategories of isometries.

\begin{ex}
\label{ex:CompleteH*mFC}
Given an $\rmH^*$-multifusion category $\cC$, $(\rmB \cC)^\cent\cong \HstarAlg(\cC) \cong \Mod^\dag(\cC)$.
\end{ex}

\begin{ex}
\label{ex:CompleteBLoops}
For any connected 3-Hilbert space $\fX$ and $x\in\fX$, $\fX\cong (\rmB\Omega_x)^\cent$.
\end{ex}

\begin{rem}
Recall from Section \S \ref{subsec:2Hilbs}, for an H*-algebra $A$ we also have $(\rmB A)^\cent \cong \Mod^\dag(A)$. As $A_A$ generates $\Mod^\dag(A)$, we saw that the trace $\Tr^{\Mod^\dag(A)}$ is uniquely determined on an $A$-module $M_A$ by the concrete relation $\Tr^{\End_A(M)}(|\xi \rangle \langle \eta | ) = \Tr^A(\langle \eta | \xi \rangle_A)$ for $\eta,\xi \in M$. 

To categorify this fact, we briefly introduce a categorified trace in the spirit of \cite{MR3578212,MR4750417}. Indeed, just as $\cC( - \to - )$ acts as a categorified inner product for an H*-multifusion category, we may define a $\Hilb$-valued trace on $\cC$ by
$$
\operatorname{TR}^\cC(c) \coloneqq \cC(1_\cC \to c)
$$
equipped with the inner product $\langle f | g \rangle_{\operatorname{TR}^\cC(c)} = \psi^\cC(g^\dag \circ f)$ for $f,g \in \operatorname{TR}^\cC(c)$. Notice this $\Hilb$-valued trace admits natural \emph{traciator unitaries} witnessing its traciality
\[
\operatorname{TR}^\cC(b \otimes c) 
\cong^\dag 
\cC(b^\vee \to c)
\cong^\dag
\operatorname{TR}^\cC(c \otimes b). 
\]
This construction also satisfies a form of positive definiteness, in the sense that
\[
\operatorname{TR}^\cC(c^\vee \otimes c) \cong^\dag \cC(c \to c) 
\]
is a $\rmW^*$-algebra, which is zero if and only if $c=0$.
Applying this analysis to the following linking H*-multifusion category from \cite[Sub-Example 3.21]{MR5078555} 
\[
\cL(\cM_\cC) = \End(\cM_\cC \oplus \cC_\cC) = 
\begin{bmatrix}
\cC & \cM^\op
\\
\cM & \End_\cC(\cM)
\end{bmatrix}
, 
\]
we obtain 
\[
\operatorname{TR}^{\End_\cC(\cM)}(|n \rangle_\cC \langle m |)
\cong^\dag 
\operatorname{TR}^{\cC}(\langle m | n \rangle_\cC) \qquad\text{for } m,n \in \cM.
\]
Here $\langle m|n\rangle_\cC$ is the right $\cC$-valued \emph{internal hom}---denoted by $[m,n]_\cC$ in \cite{MR5078555}---which is determined by the unitary adjunction
\[
    \cM(n\triangleleft c\to m)\cong\cC(c\to \langle n|m\rangle_\cC),
\]
i.e., $\langle n|\colon\cM\to\cC$ is the unitary adjoint of $n\triangleleft -\colon \cC\to \cM$.
\end{rem}

\subsection{The 3-category of 3-Hilbert spaces}

\begin{nota}\label{nota:overlaygraphicalcalculus}
For $F,G\in \Hom(\fX\to \fY)$, we denote a natural transformation by $\alpha \colon  F\Rightarrow G$, and for natural transformations $\alpha,\beta\colon F\Rightarrow G$, we denote a modification by $m\colon  \alpha \Rrightarrow \beta$.
That is, for the objects, 1-morphisms, and 2-morphisms in $\Hom(\fX\to \fY)$, 
we use the 3-categorical notation as they are 1-cells, 2-cells, and 3-cells respectively in $3\Hilb$.
By \cite[Prop.~2.13]{MR4369356}, $\Hom(\fX\to \fY)$ is again a $\rmC^*/\rmW^*$ 2-category, where the $\dag$ on modifications is given componentwise.

We use the overlay graphical calculus for 2-categories of functors between 2-categories from \cite{MR4369356}.
We continue to write 1-composition in $\Hom(\fX\to \fY)$ from \emph{left to right}.\footnote{We denote the 1-composition $G\circ F\colon  \fX\to \fZ$ of 2-functors $F\colon  \fX\to \fY$ and $G\colon  \fY\to \fZ$ \emph{right to left} as usual, but in the overlay graphical calculus, this composition happens \emph{out of the page}.
}
We denote functors $F,G\colon \fX\to \fY$ applied to objects and 1-morphisms in $\fX$ by overlaying \emph{patterns} to shaded regions and strands in $\fX$.
For example, we overlay $F$ and $G$ on $a\in \fX$ by
$$
F=
\tikzmath{
\filldraw[primedregion=white, rounded corners=5pt, draw=black, dotted] (0,0) rectangle (.6,.6);    
}
\qquad\qquad
G=
\tikzmath{
\fill[boxregion=white, rounded corners=5pt, draw=black, dotted] (0,0) rectangle (.6,.6);    
}
\qquad\qquad
a=
\tikzmath{
\fill[lightgray!50, rounded corners=5pt] (0,0) rectangle (.6,.6);    
}
\qquad\rightsquigarrow\qquad
F(a)=
\tikzmath{
\fill[primedregion=lightgray!50, rounded corners=5pt] (0,0) rectangle (.6,.6);    
}
\qquad\qquad
G(a)=
\tikzmath{
\fill[boxregion=lightgray!50, rounded corners=5pt] (0,0) rectangle (.6,.6);    
}
$$
For a 1-morphism ${}_aX_b$, we then denote $F(X)$ and $G(X)$ as follows.
$$
X=
\tikzmath{
\begin{scope}
\clip[rounded corners=5pt] (0,0) rectangle (.6,.6);
\fill[lightgray!50] (0,0) rectangle (.3,.6);    
\fill[lightgray] (.3,0) rectangle (.6,.6);    
\end{scope}
\draw[thick,\XsColor] (.3,0) -- (.3,.6);
}
\qquad
\qquad
b=
\tikzmath{
\fill[lightgray, rounded corners=5pt] (0,0) rectangle (.6,.6);    
}
\qquad
\rightsquigarrow
\qquad
F(X)=
\tikzmath{
\begin{scope}
\clip[rounded corners=5pt] (0,0) rectangle (.6,.6);
\fill[primedregion=lightgray!50] (0,0) rectangle (.3,.6);    
\fill[primedregion=lightgray] (.3,0) rectangle (.6,.6);    
\end{scope}
\draw[thick,\XsColor] (.3,0) -- (.3,.6);
}
\qquad\qquad
G(X)=
\tikzmath{
\begin{scope}
\clip[rounded corners=5pt] (0,0) rectangle (.6,.6);
\fill[boxregion=lightgray!50] (0,0) rectangle (.3,.6);    
\fill[boxregion=lightgray] (.3,0) rectangle (.6,.6);    
\end{scope}
\draw[thick,\XsColor] (.3,0) -- (.3,.6);
}
$$
For a 2-morphism $f\colon {}_aX_b\Rightarrow {}_aY_b$, we make a slight departure from \cite{MR4369356} by denoting $F(f)$ and $G(f)$ by
$$
Y=
\tikzmath{
\begin{scope}
\clip[rounded corners=5pt] (0,0) rectangle (.6,.6);
\fill[lightgray!50] (0,0) rectangle (.3,.6);    
\fill[lightgray] (.3,0) rectangle (.6,.6);    
\end{scope}
\draw[thick,\YsColor] (.3,0) -- (.3,.6);
}
\qquad
\rightsquigarrow
\qquad
F(f)
=
\tikzmath{
\begin{scope}
\clip[rounded corners=5pt] (-.6,-.6) rectangle (.6,.6);
\fill[primedregion=lightgray!50] (-.6,-.6) rectangle (0,.6);    
\fill[primedregion=lightgray] (0,-.6) rectangle (.6,.6);    
\end{scope}
\draw[dotted, rounded corners=5pt] (-.6,-.6) rectangle (.6,.6);
\draw[thick, \XsColor] (0,-.6) -- (0,-.3);
\draw[thick, \YsColor] (0,.6) -- (0,.3);
\roundNbox{primedregion=white, draw=black}{(0,0)}{.3}{0}{0}{$f$}
}
\qquad\qquad
G(f)
=
\tikzmath{
\begin{scope}
\clip[rounded corners=5pt] (-.6,-.6) rectangle (.6,.6);
\fill[boxregion=lightgray!50] (-.6,-.6) rectangle (0,.6);    
\fill[boxregion=lightgray] (0,-.6) rectangle (.6,.6);    
\end{scope}
\draw[dotted, rounded corners=5pt] (-.6,-.6) rectangle (.6,.6);
\draw[thick, \XsColor] (0,-.6) -- (0,-.3);
\draw[thick, \YsColor] (0,.6) -- (0,.3);
\roundNbox{boxregion=white, draw=black}{(0,0)}{.3}{0}{0}{$f$}
}\,.
$$
Transformations $\alpha,\beta\colon F\Rightarrow G$ 
are represented by strings with \emph{textures}, 
$$
\alpha
=
\tikzmath{
\begin{scope}
\clip[rounded corners=5pt] (0,0) rectangle (.6,.6);
\fill[primedregion=white] (0,0) rectangle (.3,.6);    
\fill[boxregion=white] (.3,0) rectangle (.6,.6);    
\end{scope}
\draw[dotted, rounded corners=5pt] (0,0) rectangle (.6,.6);
\draw[thick] (.3,0) -- (.3,.6);
}
=
\left(
\left\{
\alpha_a
=
\tikzmath{
\begin{scope}
\clip[rounded corners=5pt] (0,0) rectangle (.6,.6);
\fill[primedregion=lightgray!50] (0,0) rectangle (.3,.6);    
\fill[boxregion=lightgray!50] (.3,0) rectangle (.6,.6);    
\end{scope}
\draw[thick] (.3,0) -- (.3,.6);
}
\right\}_{a\in \fX}
,
\left\{
\alpha_X
=
\tikzmath{
\begin{scope}
\clip[rounded corners=5] (-.3,-.5) rectangle (1.3,1.5);
\fill[lightgray] (0,-.5)  -- (0,0) .. controls ++(0,.5) and ++(0,-.5) .. (1,1) -- (1,1.5) -- (1.3,1.5) -- (1.3,-.5) -- cycle;
\fill[lightgray!50] (0,-.5)  -- (0,0) .. controls ++(0,.5) and ++(0,-.5) .. (1,1) -- (1,1.5) -- (-.3,1.5) -- (-.3,-.5) -- cycle;
\fill[pattern=primeddots] (1,-.5) -- (1,0) .. controls ++(0,.5) and ++(0,-.5) .. (0,1) -- (0,1.5) -- (-.3,1.5) -- (-.3,-.5) -- cycle;
\fill[pattern=primedbox] (1,-.5) -- (1,0) .. controls ++(0,.5) and ++(0,-.5) .. (0,1) -- (0,1.5) -- (1.3,1.5) -- (1.3,-.5) -- cycle;
\end{scope}
\draw[thick] (1,-.5) node [below] {$\scriptstyle \alpha_b$} -- (1,0) .. controls ++(0,.5) and ++(0,-.5) .. (0,1) -- (0,1.5) node [above] {$\scriptstyle \alpha_a$};
\draw[thick,\XsColor] (0,-.5) node [below] {$\scriptstyle FX$}-- (0,0) .. controls ++(0,.5) and ++(0,-.5) .. (1,1) -- (1,1.5) node [above] {$\scriptstyle GX$};
\filldraw[thick,fill=white] (.5,.5) circle (.1cm);
}
\right\}_{X\colon a\to b}
\right)
$$
and modifications $m\colon \alpha\Rrightarrow \beta$ are denoted by coupons along those textured strings.
$$
\beta
=
\tikzmath{
\begin{scope}
\clip[rounded corners=5pt] (0,0) rectangle (.6,.6);
\fill[primedregion=white] (0,0) rectangle (.3,.6);    
\fill[boxregion=white] (.3,0) rectangle (.6,.6);    
\end{scope}
\draw[dotted, rounded corners=5pt] (0,0) rectangle (.6,.6);
\draw[thick, snake] (.3,0) -- (.3,.6);
}
\qquad\rightsquigarrow\qquad
\tikzmath{
\begin{scope}
\clip[rounded corners=5pt] (-.6,-.6) rectangle (.6,.6);
\fill[primedregion=white] (-.6,-.6) rectangle (0,.6);    
\fill[boxregion=white] (0,-.6) rectangle (.6,.6);    
\end{scope}
\draw[dotted, rounded corners=5pt] (-.6,-.6) rectangle (.6,.6);
\draw[thick] (0,-.6) -- (0,-.3);
\draw[thick, snake] (0,.3) -- (0,.6);
\roundNbox{fill=white}{(0,0)}{.3}{0}{0}{$m$}
}
=
\left\{
m_a =
\tikzmath{
\begin{scope}
\clip[rounded corners=5pt] (-.6,-.6) rectangle (.6,.6);
\fill[primedregion=lightgray!50] (-.6,-.6) rectangle (0,.6);    
\fill[boxregion=lightgray!50] (0,-.6) rectangle (.6,.6);    
\end{scope}
\draw[dotted, rounded corners=5pt] (-.6,-.6) rectangle (.6,.6);
\draw[thick] (0,-.6) -- (0,-.3);
\draw[thick, snake] (0,.3) -- (0,.6);
\roundNbox{fill=white}{(0,0)}{.3}{0}{0}{$m_a$}
}
\right\}_{a\in\fX}
$$
\end{nota}

\subsection{The isometric unitary Yoneda lemma internal to a 3-Hilbert space}
\label{sec:UnitaryYoneda}

We now categorify \eqref{eq:2HilbsAreHilbEnriched} to 3-Hilbert spaces using the unitary adjunctions 
${}_bY^\vee \otimes_a - \dashv^\dag {}_aY\otimes_b -$
and
$-\otimes_b Z_c \dashv^\dag -\otimes_c Z^\vee_b$:
$$
\fX({}_bY^\vee \otimes_a X_c \Rightarrow {}_bZ_c) \isomeq
\fX({}_aX_c \Rightarrow {}_aY\otimes_b Z_c) \isomeq
\fX({}_a X \otimes_c Z^\vee_b \Rightarrow {}_aY_b).
$$
In Lemma \ref{lem:uaf-reps-unitary-adjs} below, we show that the second unitary adjunction implies that the representable functors $\fX(-\to x)\colon \fX^{1\op}\to 2\Hilb$ are $\dag$ and UAF-preserving,
in which case the first unitary adjunction implies
that the Yoneda embedding $\fX\hookrightarrow \Hom(\fX^{1\op}\to 2\Hilb)$ is also $\dag$, UAF-preserving, and isometric.
This first unitary adjunction will be proven in \ref{D:YonedaEmbeddingIsometric} below once we define the UAF and spherical weight on $\Hom(\fX^{1\op}\to 2\Hilb)$.

\begin{lem}\label{lem:uaf-reps-unitary-adjs}
For a 1-morphism ${}_aY_b\in \fX$,
and any $c \in \fX$, the functors 
${}_aY\otimes_b - \colon \fX(b \to c) \to \fX(a \to c)$ and ${}_bY^\vee \otimes_a - \colon \fX(a \to c)\to \fX(b \to c)$, with adjunction data coming from the UAF $(-)^\vee$, are unitary adjoints. 
In other words, $\fX(-\to c)\colon \fX^{1\op}\to 2\Hilb$ is $\dag$ and UAF-preserving, i.e., a 1-morphism of 3-Hilbert spaces.

Similarly, $-\otimes_b Z_c \dashv^\dag -\otimes_c Z^\vee_b$ and each $\fX(a\to -)\colon \fX\to 2\Hilb$ is a 1-morphism of 3-Hilbert spaces.
\end{lem}
\begin{proof}
Given ${}_bZ_c$ and ${}_aX_c$, we must show that 
$$
\mate\colon \fX( {}_aX_c \Rightarrow {}_aY \otimes_b Z_c ) \xrightarrow{\sim} \fX({}_bY^\vee \otimes_a X_c \Rightarrow {}_bZ_c)$$ 
under the unitary adjunction is unitary. 
Indeed, as $\vee$ is a UAF on $\fX$, for $f\colon {}_a X_c \Rightarrow {}_a Y \otimes_b Z_c$,  $\mate(f^\dag)=\mate(f)^\dag$.
We thus observe
\[
\|f\|^2 =
\Psi_c^\fX\left(\,
\tikzmath{
\fill[lightgray, rounded corners=5pt] (-1.7,-1.2) rectangle (.7,2.2);
\fill[lightgray!65] (-.2,1.3) arc(0:180:.2cm) -- (-.6,-.3) arc(-180:0:.2cm) -- (.2,-.3) arc(0:-180:.6cm) -- (-1,1.3) arc(180:0:.6cm);
\fill[lightgray!30] (0,.3) -- (0,.7) -- (-.2,1.3) arc(0:180:.2cm) -- (-.6,-.3) arc(-180:0:.2cm);
\draw[thick] (-.2,1.3) arc(0:180:.2cm) --node[right]{$\scriptstyle Y^\vee$} (-.6,-.3) arc(-180:0:.2cm);
\draw[thick] (.2,1.3) arc(0:180:.6cm) --node[left]{$\scriptstyle Z^\vee$} (-1,-.3) arc(-180:0:.6cm);
\draw[thick] (0,.3) --node[right]{$\scriptstyle X$} (0,.7);
\roundNbox{fill=white}{(0,1)}{.3}{.1}{.1}{$f$}
\roundNbox{fill=white}{(0,0)}{.3}{.1}{.1}{$f^\dag$}
}\,
\right)
=
\Psi_c^\fX\left(\,
\tikzmath{
\fill[lightgray, rounded corners=5pt] (-1.9,-1.2) rectangle (1.2,2.2);
\fill[lightgray!65] (0,1.3) arc(0:180:.6cm) -- (-1.2,-.3) arc(-180:0:.6cm);
\fill[lightgray!30] (-.4,.3) rectangle (.4,.7);
\draw[thick] (0,1.3) arc(0:180:.6cm) --node[left]{$\scriptstyle Z^\vee$} (-1.2,-.3) arc(-180:0:.6cm);
\draw[thick] (-.4,.3) --node[right]{$\scriptstyle Y^\vee$} (-.4,.7);
\draw[thick] (.4,.3) --node[right]{$\scriptstyle X$} (.4,.7);
\roundNbox{fill=white}{(0,1)}{.3}{.55}{.55}{$\mate(f)$}
\roundNbox{fill=white}{(0,0)}{.3}{.55}{.55}{$\mate(f)^\dag$}
}\,
\right)
=
\|\!\mate(f)\|^2,
\]
so $\mate$ is an invertible isometry, that is, a unitary.
The proof of the other unitary adjunction is similar, capping on the right hand side.
\end{proof}

The following corollary is essential for a representing object of a functor to be well-defined up to isometric equivalence.

\begin{cor}\label{cor:isometric-yoneda}
Let $a,b \in \fX$ for a 3-Hilbert space $\fX$. 
If there is a natural unitary equivalence 
$$\fX(-\to a) \isomeq \fX(-\to b)$$
that is pointwise isometric (that is, its components are isometric equivalences of 2-Hilbert spaces), then the representing equivalence $a \simeq b$ is an isometric equivalence in $\fX$.
\end{cor}
\begin{proof}
Let ${}_aX_b$ be the representing equivalence. 
By Lemma \ref{lem:uaf-reps-unitary-adjs}, the unitary adjoint of $- \otimes_a X _b \colon \fX(-\to a) \to \fX(-\to b)$ is represented by ${}_bX^\vee_a$ and the adjunction data of the UAF. 
In particular, since this data gives an isometric equivalence of 2-Hilbert spaces, this means $\ev_X$ and $\coev_X$ are both unitary isomorphisms, which is equivalent to $X$ being an isometric equivalence.
\end{proof}

\begin{cor}
\label{cor:YonedaLinear}
Taking hom categories in a 3-Hilbert space $\fX$ is compatible with Hilbert direct sums, i.e., we have
\begin{align*}
\fX(a\boxplus b\to c)
&\isomeq
\fX(a\to c)\boxplus \fX(b\to c)
\\
\fX(a\to b\boxplus c)
&\isomeq
\fX(a\to b)\boxplus \fX(a\to c)
\end{align*}
as 2-Hilbert spaces.
\end{cor}
\begin{proof}
Since each $\fX(-\to c)$ and $\fX(a\to -)$ preserve UAFs, they also preserve (co)isometries and Hilbert direct sums by \cite[Lem.~4.23]{MR5078555}; see Remark \ref{fact:IsometricPreservesStuff}.
\end{proof}

\section{The 3-Hilbert space toolbox}

Recall that for a 2-Hilbert space $(\cC,\Tr^\cC)$, $\pi_0\cC$ is a set of representatives for the simple objects of $\cC$.
In light of the unitary Yoneda results for 2-Hilbert spaces \cite{MR4750417}, recast in terms of the $\Hilb$-valued inner product in \cite{2411.01678}, one should view $\pi_0\cC$ as an ONB, as it gives a canonical resolution of the identity
$$
\id_\cC = \bigoplus_{b\in \pi_0\cC} | b\rangle\langle b|.
$$
Here, $|b\rangle\langle b| = \cC(b \to -) \odot_{\Omega_b} b$, meaning that for each $c\in \cC$, \cite[Eq.~(5)]{MR4750417} or \cite[(17)]{2411.01678} gives a canonical unitary isomorphism
$$
c \cong^\dag \bigoplus_{b\in\pi_0\cC} \ket b\bra b c =\bigoplus_{b\in\pi_0\cC} \cC(b\to c) \odot_{\Omega_b} b
=
\bigoplus_{b\in\pi_0\cC} \langle b| c\rangle^\cC_{\Hilb} \odot_{\Omega_b} b.
$$
With this in mind, we now turn our attention to ONBs and resolutions of the identity for 3-Hilbert spaces.
Our main tools for this section will be the unitary Yoneda Lemma for 3-Hilbert spaces, as well as a new generalized scalar multiplication operation.

\subsection{Coefficients for objects in a 3-Hilbert space}
\label{subsec:coefficients}
We mimic the construction $H\odot_A c$ for 2-Hilbert spaces from Construction \ref{cstr:rel-tsr-prod-over-A} in the context of 3-Hilbert spaces via the universal property of completion.

\begin{defn}\label{defn:coefficientsforobjects}
Suppose $\fX$ is a 3-Hilbert space, $c\in\fX$, and $\cM_{\Omega_c}\in \Mod^\dag(\Omega_c)$.
We define the object
$
\cM \boxdot_{\Omega_c} c
\in\fX
$
as the object corresponding to $\cM\in \Mod^\dag(\Omega_c)$ under the isometric embedding
$$
\begin{tikzcd}[row sep =0]
\cM\arrow[r, mapsto]
&
\cM\boxdot_{\Omega_c} c
\\
\Mod^\dag(\Omega_c)  
\arrow[r, hook]
& \fX
\\
\mbox{}
\\
\mbox{}
\\
\rmB \Omega_c \arrow[uuur, hook]
\arrow[uuu, hook]
\end{tikzcd}
$$
Here, the top arrow $\Mod^\dag(\Omega_c)\to \fX$ is the unique extension of the map $\rmB\Omega_c\hookrightarrow \fX$ 
given by $*\mapsto c$
afforded by the universal property of completion
\eqref{eq:UniversalPropertyOfCompletion}
together with the isometric equivalence
$(\rmB\Omega_c)^\cent \cong \Mod^\dag(\Omega_c)$ from Examples \ref{ex:CompleteH*mFC} and \ref{ex:CompleteBLoops}.
\end{defn}

\begin{rem}\label{rem:RelativeScalarActionIsDagVeePreserving}
As $\rmB \Omega_c \hookrightarrow \cC$ is trivially $\dag,\vee$-preserving, this action $- \boxdot_{\Omega_c} c$ by relative scalars is automatically $\dag,\vee$-preserving by the universal property of completion
\eqref{eq:UniversalPropertyOfCompletion}. 
\end{rem}

\begin{ex}
Consider $\fX=\Mod^\dag(\cC)$ for an $\rmH^*$-multifusion category $\cC$.
Since we have a canonical equivalence of $\rmH^*$-multifusion categories $\Omega_{\cC_\cC}\cong \cC$,
we see that the unique extension
$\Mod^\dag(\Omega_{\cC_\cC}) \hookrightarrow \fX=\Mod^\dag(\cC)$ is the identity functor.
We may thus identify 
$$
\cM_\cC\boxdot_{\Omega_{\cC_\cC}} \cC_\cC = 
\cM_\cC.
$$
\end{ex}

\begin{rem}
Suppose $c\in\fX$ and $d\in \fX$ is another object in the same component as $c$.
Then $\fX(c\to d)\in \Mod^\dag(\Omega_c)$
and
we have a canonical isometric equivalence 
\begin{equation}
\label{eq:ExpressInBasis}    
\fX(c\to d) \boxdot_{\Omega_c} c \cong d.
\end{equation}
Indeed, setting $\cC\coloneq\Omega_c$,
by \cite[Cor.~4.52]{MR5078555},
we may assume $\fX\cong \Mod^\dag(\cC)$
(which is the only part of $\fX$ that $c$ will recover), under which $c\mapsto \cC_\cC$
and $d\mapsto \cM_\cC$.
We have an isometric equivalences of $\cC$-modules
$$
\Hom(\cC_\cC \to \cM_{\cC}) 
\underset{\text{\cite[Prop.~3.9]{MR5078555}}}{\cong} 
\cM_{\cC}
\underset{\text{\cite[Cor.~3.12]{MR5078555}}}{\cong} 
\cM_\cC \boxtimes_\cC \cC_\cC.
$$
We thus see that 
$$
\cM_\cC
\cong
\cM \boxtimes_\cC \cC_\cC 
\cong 
\Hom(\cC_\cC \to \cM_\cC) \boxtimes_\cC \cC.
$$
\end{rem}

Similar to \cite[Eq.~(11)]{2411.01678} one categorical level down, we may also discuss more general coefficients in $\Mod^\dag(\cA)$ where $\cA\isometry \Omega_c$ is an isometric dagger tensor functor.

\begin{defn}\label{D:reltsrprodH*monad}
Suppose $\fX$ is a 3-Hilbert space, $c\in\fX$, 
$\cA$ is an $\rmH^*$-multifusion category,
and $F\colon  \cA\isometry \Omega_c$ is an isometric dagger tensor functor.
For 
$\cM_{\cA}\in \Mod^\dag(\cA)$, we define the object
$
\cM \boxdot_{\cA} c
\in\fX
$
as the object corresponding to $\cM\in \Mod^\dag(\cA)$ under the isometric embedding
$$
\begin{tikzcd}[row sep =0]
\cM 
\arrow[r, mapsto]
&
\cM \boxdot_{\cA} c
\\
\Mod^\dag(\cA)  
\arrow[r, hook, harpoon]
& 
\fX
\\
\mbox{}
\\
\mbox{}
\\
\rmB \cA
\arrow[uuu, hook]
\arrow["\mathrm{B}F", r, hook, harpoon]
&
\rmB \Omega_c \arrow[uuu, hook]
& 
\end{tikzcd}
$$
(We note that, technically speaking, the symbol $\cM \boxdot_{\cA} c$ should also involve $F$; we suppress it to ease the notation.)

As above, the map $\Mod^\dag(\cA)\hookrightarrow \fX$ is the unique extension $\rmB\cA\hookrightarrow \fX$ given by $*\mapsto c$ afforded by the universal property of completion
\eqref{eq:UniversalPropertyOfCompletion} together with the isometric equivalence
$(\rmB\cA)^\cent \cong \Mod^\dag(\cA)$.
\end{defn}

\begin{rem}\label{rem:ComputeRelativeScalar}
We can explicitly spell out how to compute $\cM \boxdot_\cA c$ via the definition above.
Given $\cM \in \Mod^\dag(\cA)$, there is an $\rmH^*$-algebra $A \in \cA$ such that $\cM \cong^\dag \Mod_\cA(A)$.
Now $FA$ is an $\rmH^*$-monad on $c$ in $\fX$, so we can split $FA$.
This splitting is the object $\cM\boxdot_\cA c$.
This computation follows directly from the isometric equivalence $\Mod^\dag(\cA) \cong \HstarAlg(\cA)$.
\end{rem}

\begin{ex}
Definition \ref{D:reltsrprodH*monad} gives every 3-Hilbert space $\fX$ a natural tensored structure over $2\Hilb$.
For any object $x \in \fX$, the unit functor $\Hilb \isometry \Omega_x$ is an isometric dagger tensor functor, and thus for any 2-Hilbert space $\cH$, this defines an object $\cH\boxdot x \coloneq \cH\boxdot_{\Hilb}x$.
Explicitly, if $\cH = \bigboxplus (\Hilb,r\Tr)^{\boxplus n_r}$, then $\cH\boxdot x \cong^\dag \bigboxplus \left((\Hilb,r\Tr)\boxdot x\right)^{\boxplus n_r}$, where $(\Hilb,r\Tr)\boxdot x$ is given by splitting the $\rmH^*$-monad $(1_x,r^{-1/2}\id_{1_x},r^{1/2}\id_{1_x})$.
(This is the monad obtained from following the calculation in Remark \ref{rem:ComputeRelativeScalar}).
\end{ex}

\begin{rem}
Similar to \cite[proof of Lem.~3.36]{2411.01678},
the universal property of completion allows us to factorize the map $\cM\mapsto \cM\boxdot_\cA c$ as
$$
\begin{tikzcd}[row sep =0]
\cM 
\arrow[r, mapsto]
& 
\cM \boxtimes_{\cA} \Omega_c
\arrow[r, mapsto]
&
\cM \boxdot_{\cA} c
\\
\Mod^\dag(\cA)  
\arrow[r, hook, harpoon]
& \Mod^\dag(\Omega_c) 
\arrow[r, hook]
& 
\fX.
\\
\mbox{}
\\
\mbox{}
\\
\rmB \cA
\arrow[uuu, hook]
\arrow["\mathrm{B}F", r, hook, harpoon]
&
\rmB \Omega_c \arrow[uuu, hook] \arrow[uuur, hook] 
& 
\end{tikzcd}
$$
Using this factorization, one can construct isometric equivalences
$$
(\cM\boxtimes_\cB \cN)\boxdot_\cA c
\xrightarrow{\simeq}
\cM\boxdot_\cB (\cN\boxdot_\cA c)
$$
by importing the isometric equivalence from $3\Hilb$
$$
(\cM\boxtimes_\cB \cN)\boxtimes_\cA \cA
\xrightarrow{\simeq}
\cM\boxtimes_\cB (\cN\boxtimes_\cA \cA).
$$
One also obtains in this way a pentagonator (corresponding to $\cM_1 \boxtimes \cM_2 \boxtimes \cM_3 \boxtimes \cM_4$) as well as three triangulators (corresponding to $1 \boxtimes \cM \boxtimes \cN$, $\cM \boxtimes 1 \boxtimes \cN$, and $\cM \boxtimes \cN \boxtimes 1$) that satisfy the usual coherences for 3-categories (an associahedra axiom corresponding to $\cM_1 \boxtimes \cM_2 \boxtimes \cM_3 \boxtimes \cM_4 \boxtimes \cM_5$ and two unitality axioms corresponding to $\cM_1 \boxtimes 1 \boxtimes \cM_3 \boxtimes \cM_4$ and $\cM_1 \boxtimes \cM_2 \boxtimes 1 \boxtimes \cM_4$). We refer the interested reader to \cite{MR3076451}.
\end{rem}

\begin{lem}
\label{lem:1MorphismsAreLinear}
Suppose $\fX,\fY$ are 3-Hilbert spaces, and $c,\cA,\cM$ are as above.
For any 1-morphism $G\colon \fX\to \fY$, we have a canonical natural isometric equivalence
$$
G(\cM \boxdot_\cA c) \isomeq \cM \boxdot_\cA G(c).
$$
\end{lem}
\begin{proof}
This follows from the universal property of completion via the following diagram.
$$
\begin{tikzcd}
&
\Mod^\dag(\cA)
\arrow[dl,hook,harpoon, "\cM\mapsto \cM\boxdot_\cA c", swap]
\arrow[dr,hook,harpoon, "\cM\mapsto \cM\boxdot_\cA G(c)"]
\\
\fX 
\arrow[rr,near start, "G"]
&&
\fY.
\\
&
\rmB\cA
\arrow[uu, hook, near start, "\iota", crossing over]
\arrow[ur, hook, harpoon ,"*\mapsto G(c)", swap]
\arrow[ul, hook, harpoon, "*\mapsto c"]
\end{tikzcd}
$$
Since
$
G(\cA\boxdot_\cA c)
=
G(c)
=
\cA\boxdot_\cA G(c),
$
the functors $\Mod^\dag(\cA)\to \fY$ given by
$G(-\boxdot_\cA c)$ and $-\boxdot_\cA G(c)$ are then pointwise isometrically isomorphic
as they are both the unique extension of $\rmB\cA\to \fY$ given by $*\mapsto G(c)$.
\end{proof}

The reason we view this construction as generalized scalars for $c\in \fX$ is that scalars `pull out' of the categorified $2\Hilb$-valued hom inner product.

\begin{prop}
\label{prop:2HilbertScalarsPullOut}
For all $b,c\in \fX$ and $\cM\in \Mod^\dag(\cA)$ where $\cA\hookrightarrow \Omega_c$,
we have canonical isometric equivalences of 2-Hilbert spaces
\begin{align*}
\fX(b\to \cM\boxdot_{\cA} c) 
&\isomeq
\cM\boxtimes_{\cA}\fX(b\to c)
\\
\fX(\cM\boxdot_{\cA} c\to b) 
&\isomeq
\fX(c\to b)\boxtimes_{\cA} \cM^{\op}.
\end{align*}
Thus $\cM\boxdot_\cA c$ isometrically represents the functor $\cM\boxtimes_\cA\fX(-\to c)\colon  \fX^{\op}\to 2\Hilb$.
\end{prop}
\begin{proof}
For the first statement, apply Lemma \ref{lem:1MorphismsAreLinear} to 
$\fX(b\to -)\colon \fX\to 2\Hilb$, which is a 1-morphism in $3\Hilb$ by Lemma \ref{lem:uaf-reps-unitary-adjs}.
The second statement is similar.
\end{proof}

\subsection{Orthonormal bases for a 3-Hilbert space}
\label{subsec:ONBS}
Recall from \cite[Def.~1.2.22]{1812.11933} that two simple objects $a,b\in\fX$ are in the same \emph{component} if there is a non-zero (adjointable) 1-morphism ${}_aX_b\in \fX$.
Being in the same component is an equivalence relation on simples, denoted $a\sim b$.

\begin{defn}\label{defn:ONB}
An \emph{orthonormal basis} (ONB; see Footnote \ref{Footnote:ONB})
for a 3-Hilbert space $\mathfrak X$ is a choice of simple object $b_i$ for each component of $\fX$.
We will denote an ONB by $\pi_0\fX$.
\end{defn}

\begin{rem}
By 
\cite[Cor.~4.52]{MR5078555},
an ONB is equivalent to the data of a finite list $\{\cC_i\}$ of $\rmH^*$-fusion categories and an isometric equivalence of 3-Hilbert spaces $\bigboxplus_i \Mod^\dag(\cC_{i}) \cong \fX$. 
\end{rem}

\begin{rem}
\label{rem:DeterminedByONB}
Fix an ONB $\pi_0\fX=\{b_i\}$ of $\fX$, and observe that $(\rmB\Omega_{\bigboxplus b_i})^\cent = \fX$.
By the universal property of completion, for every 3-Hilbert space $\fY$,
precomposition with the isometric inclusion $\iota\colon \rmB\Omega_{\bigboxplus b_i}\hookrightarrow \fX$ is a pointwise isometric equivalence
\[
\Hom(\fX\to \fY)
\xrightarrow[\simeq]{\iota\circ -}
\Hom(\rmB\Omega_{\bigboxplus b_i}\to \fY).
\]
It immediately follows that
\begin{itemize}
\item 
every $F\in\Hom(\fX\to \fY)$ is completely determined (up to contractible choice of pointwise isometric equivalence) by its restriction to $\bigboxplus \rmB\Omega_{b_i} \isomeq \rmB \Omega_{\bigboxplus b_i}$,
\item 
a transformation $\alpha\colon  F\Rightarrow G$ is completely determined by 
the 1-morphism $\alpha_{\bigboxplus b_i}: F(\bigboxplus b_i)\rightarrow G(\bigboxplus b_i)$ 
together with its half-braidings with $\Omega_{\boxplus b_i}$, i.e., for each $X\in \Omega_{\boxplus b_i}$, the unitaries $\alpha_X: F(X)\otimes_{F(\boxplus b_i)} \alpha_{\boxplus b_i} \to \alpha_{\boxplus b_i}\otimes_{G(\boxplus b_i)} G(X)$ which satisfy monoidality and unitality coherences as well as
\[
\tikzmath{
\begin{scope}
\clip[rounded corners=5] (-.4,-.6) rectangle (1.4,1.6);
\fill[primedregion=\AColor] (-.4,-.6) -- (1,-.6) -- (1,0) .. controls ++(0,.5) and ++(0,-.5) .. (0,1) -- (0,1.6) -- (-.4,1.6);
\fill[boxregion=\AColor] (1.4,-.6) -- (1,-.6) -- (1,0) .. controls ++(0,.5) and ++(0,-.5) .. (0,1) -- (0,1.6) -- (1.4,1.6);
\end{scope}
\draw[thick] (1,-.6) node [below] {$\scriptstyle \alpha_{\boxplus b_i}$} -- (1,0) .. controls ++(0,.5) and ++(0,-.5) .. (0,1) -- (0,1.6) node [above] {$\scriptstyle \alpha_{\boxplus b_i}$};
\draw[thick,\XsColor] (0,-.6) node [below] {$\scriptstyle FX$}-- (0,0) .. controls ++(0,.5) and ++(0,-.5) .. (1,1) -- (1,1.6) node [above] {$\scriptstyle GX$};
\filldraw[thick,fill=white] (.5,.5) circle (.1cm);
\roundNbox{primedregion=white, draw=black}{(0,-.2)}{.2}{0}{0}{$\scriptstyle f$};
}
=
\tikzmath{
\begin{scope}
\clip[rounded corners=5] (-.4,-.6) rectangle (1.4,1.6);
\fill[primedregion=\AColor] (-.4,-.6) -- (1,-.6) -- (1,0) .. controls ++(0,.5) and ++(0,-.5) .. (0,1) -- (0,1.6) -- (-.4,1.6);
\fill[boxregion=\AColor] (1.4,-.6) -- (1,-.6) -- (1,0) .. controls ++(0,.5) and ++(0,-.5) .. (0,1) -- (0,1.6) -- (1.4,1.6);
\end{scope}
\draw[thick] (1,-.6) node [below] {$\scriptstyle \alpha_{\boxplus b_i}$} -- (1,0) .. controls ++(0,.5) and ++(0,-.5) .. (0,1) -- (0,1.6) node [above] {$\scriptstyle \alpha_{\boxplus b_i}$};
\draw[thick,\XsColor] (0,-.6) node [below] {$\scriptstyle FX$}-- (0,0) .. controls ++(0,.5) and ++(0,-.5) .. (1,1) -- (1,1.6) node [above] {$\scriptstyle GX$};
\filldraw[thick,fill=white] (.5,.5) circle (.1cm);
\roundNbox{boxregion=white, draw=black}{(1,1.2)}{.2}{0}{0}{$\scriptstyle f$};
}
\qquad\qquad\qquad
\forall f\in \End_\fX(X).
\]
\item 
Finally, by an argument similar to \cite[Prop.~4.8]{MR4369356}, a modification $m\colon \alpha\Rrightarrow \beta$ is completely determined by the $m_b\colon  \alpha_b\Rightarrow \beta_b$ for $b\in\pi_0\fX$. Indeed, since modifications $m,n : \alpha \Rrightarrow \beta$ are uniquely determined by their component on $\bigboxplus b$, if $m_b = n_b$ for $b \in \pi_0 \fX$ then
\[
\tikzmath{
\begin{scope}
\clip[rounded corners=5pt] (-.6,-.6) rectangle (.6,.6);
\fill[primedregion=lightgray!50] (-.6,-.6) rectangle (0,.6);    
\fill[boxregion=lightgray!50] (0,-.6) rectangle (.6,.6);    
\end{scope}
\draw[dotted, rounded corners=5pt] (-.6,-.6) rectangle (.6,.6);
\draw[thick] (0,-.6) -- (0,-.3);
\draw[thick, snake] (0,.3) -- (0,.6);
\roundNbox{fill=white}{(0,0)}{.3}{0.1}{0.1}{$m_{\boxplus b}$}
}
=
\bigboxplus_{b \in \pi_0 \fX}
\tikzmath{
\begin{scope}
\clip[rounded corners=5pt] (-1.2,-.6) rectangle (.6,1.2);
\fill[lightgray!50] (-1.2,-.6) rectangle (0,1.2);    
\fill[lightgray!50] (0,-.6) rectangle (.6,1.2); 
\draw[thick,red,fill=\BColor] (-.6,.75) circle (.3);
\fill[primedregion] (-1.2,-.6) rectangle (0,1.2);    
\fill[boxregion] (0,-.6) rectangle (.6,1.2);    
\end{scope}
\draw[dotted, rounded corners=5pt] (-1.2,-.6) rectangle (.6,1.2);
\draw[thick] (0,-.6) -- (0,-.3);
\draw[thick, snake] (0,.3) -- (0,1.2);
\roundNbox{fill=white}{(0,0)}{.3}{0.1}{0.1}{$m_{\boxplus b}$}
}
=
\bigboxplus_{b \in \pi_0 \fX}
\tikzmath{
\begin{scope}
\clip[rounded corners=5pt] (-.6,-.6) rectangle (.6,1.2);
\fill[lightgray!50] (-.6,-.6) rectangle (0,1.2);    
\fill[lightgray!50] (0,-.6) rectangle (.6,1.2); 
\draw[thick,red,fill=\BColor] (0,.75) circle (.3);
\fill[primedregion] (-.6,-.6) rectangle (0,1.2);    
\fill[boxregion] (0,-.6) rectangle (.6,1.2);    
\end{scope}
\draw[dotted, rounded corners=5pt] (-.6,-.6) rectangle (.6,1.2);
\draw[thick] (0,-.6) -- (0,-.3);
\draw[thick, snake] (0,.3) -- (0,1.2);
\roundNbox{fill=white}{(0,0)}{.3}{0.1}{0.1}{$m_{\boxplus b}$};
\draw[thick,fill=white] (0,.75+.3) circle (.05);
\draw[thick,fill=white] (0,.75-.3) circle (.05);
}
=
\bigboxplus_{b \in \pi_0 \fX}
\tikzmath{
\begin{scope}
\clip[rounded corners=5pt] (-.9,-.9) rectangle (.9,.9);
\fill[lightgray!50] (-.9,-.9) rectangle (0,.9);    
\fill[lightgray!50] (0,-.9) rectangle (.9,.9); 
\draw[thick,red,fill=\BColor] (0,0) circle (.6);
\fill[primedregion] (-.9,-.9) rectangle (0,.9);    
\fill[boxregion] (0,-.9) rectangle (.9,.9);    
\end{scope}
\draw[dotted, rounded corners=5pt] (-.9,-.9) rectangle (.9,.9);
\draw[thick] (0,-.9) -- (0,-.3);
\draw[thick, snake] (0,.3) -- (0,.9);
\roundNbox{fill=white}{(0,0)}{.3}{0}{0}{$m_{b}$};
\draw[thick,fill=white] (0,.6) circle (.05);
\draw[thick,fill=white] (0,-.6) circle (.05);
}
=
\bigboxplus_{b \in \pi_0 \fX}
\tikzmath{
\begin{scope}
\clip[rounded corners=5pt] (-.9,-.9) rectangle (.9,.9);
\fill[lightgray!50] (-.9,-.9) rectangle (0,.9);    
\fill[lightgray!50] (0,-.9) rectangle (.9,.9); 
\draw[thick,red,fill=\BColor] (0,0) circle (.6);
\fill[primedregion] (-.9,-.9) rectangle (0,.9);    
\fill[boxregion] (0,-.9) rectangle (.9,.9);    
\end{scope}
\draw[dotted, rounded corners=5pt] (-.9,-.9) rectangle (.9,.9);
\draw[thick] (0,-.9) -- (0,-.3);
\draw[thick, snake] (0,.3) -- (0,.9);
\roundNbox{fill=white}{(0,0)}{.3}{0}{0}{$n_{b}$};
\draw[thick,fill=white] (0,.6) circle (.05);
\draw[thick,fill=white] (0,-.6) circle (.05);
}
=\cdots=
\tikzmath{
\begin{scope}
\clip[rounded corners=5pt] (-.6,-.6) rectangle (.6,.6);
\fill[primedregion=lightgray!50] (-.6,-.6) rectangle (0,.6);    
\fill[boxregion=lightgray!50] (0,-.6) rectangle (.6,.6);    
\end{scope}
\draw[dotted, rounded corners=5pt] (-.6,-.6) rectangle (.6,.6);
\draw[thick] (0,-.6) -- (0,-.3);
\draw[thick, snake] (0,.3) -- (0,.6);
\roundNbox{fill=white}{(0,0)}{.3}{0.1}{0.1}{$n_{\boxplus b}$}
}
\]
and hence $m = n$.
\end{itemize}
\end{rem}

\begin{prop}
\label{prop:FourierExpansion}
Let $\pi_0\fX=\{b_i\}$ be an ONB for $\mathfrak X$.
For any $c\in\fX$, there is a natural isometric isomorphism $c \isomeq \bigboxplus_{b\in\pi_0\fX} \fX(b\to c) \boxdot_{\Omega_b} b$.
\end{prop}
\begin{proof}
Immediate from \eqref{eq:ExpressInBasis} by taking Hilbert direct sums.
\end{proof}

\begin{cor}
For any simple $s\in \fX$, $|s\rangle\langle s| \in \End(\fX)$ given by $x\mapsto \fX(s\to x) \boxdot_{\Omega_s} s$ is projection onto the $s$-component of $\fX$.
\end{cor}
\begin{proof}
Extend $\{s\}$ to an ONB for $\fX$ and apply Proposition \ref{prop:FourierExpansion}. 
\end{proof}

\begin{cor}
\label{cor:AllPresheafsRepresentable}
A presheaf $F\colon \fX^{1\op}\to 2\Hilb$ 
is isometrically representable by the object
$$
\bigboxplus_{b\in \pi_0\fX} F(b)\boxdot_{\Omega_b} b
\in \fX.
$$
\end{cor}
\begin{proof}
For $c\in\fX$ in a single component, let $a\in \pi_0\fX$ such that $c\sim a$.
Then we have
\begin{align*}
\fX\left(
c\to 
\bigboxplus_{b} F(b)\boxdot_{\Omega_b} b
\right)
&\isomeq
\bigboxplus_{b}\, \fX\left(
c\to 
F(b)\boxdot_{\Omega_b} b
\right)
&&
\text{(Cor.~\ref{cor:YonedaLinear})}
\\&\isomeq
\bigboxplus_{b}
F(b)
\boxtimes_{\Omega_b}
\fX(c\to b)
&&
\text{(Prop.~\ref{prop:2HilbertScalarsPullOut})}
\\&\isomeq
F(a) \boxtimes_{\Omega_a}
\fX(c\to a)
&&\text{($c\sim a$)}
\\&\isomeq
F(\fX(c\to a)^{\op}\boxdot_{\Omega_a} a)
&&
\text{(Prop.~\ref{prop:2HilbertScalarsPullOut})}
\\&\isomeq
F(\fX(a\to c)\boxdot_{\Omega_a} a)
&&\text{(adjoints)}
\\&\isomeq
F(c).
&&\eqref{eq:ExpressInBasis}
\end{align*}
The result now follows by Corollary \ref{cor:isometric-yoneda} and taking Hilbert direct sums.
\end{proof}

\subsection{Unitary 2-adjunction}\label{subsec:Unitary2Adjunction}
As discussed in Example \ref{ex:UnitaryAdjunction} above, given two 2-Hilbert spaces $\cA,\cB$, two dagger functors $F\colon \cA\to \cB$ and $F^*\colon \cB\to \cA$ are said to be \emph{unitary adjoints}, denoted $F\dashv^\dag F^*$, if there is a family of natural unitary isomorphisms
$$
\cB(F(a)\to b)\isomeq \cA(a\to F^*(b))
\qquad\qquad
\forall\,a\in \cA,\, b\in\cB.
$$
As an application of ONBs, we get unitary adjoints for $\dag$-2-functors between 3-Hilbert spaces.
We first recall the definition of (non-dagger!) adjoint 2-functors from \cite{MR371990}.

\begin{defn}\label{defn:2-adjoint}
Suppose $F\colon \fX\to \fY$ and $G\colon  \fY\to \fX$ are (non-dagger) 2-functors between 3-vector spaces (finite semisimple linear 2-categories).
We say that $F$ is \emph{left adjoint} to $G$, denoted $F\dashv G$, if there is a collection of equivalences of categories 
$$
\langle Fx|y\rangle 
\coloneq 
\fY(F(x) \to y) 
\cong 
\fX(x \to G(y))
\eqcolon
\langle x| Gy\rangle
$$ 
that is natural in $x$ and $y$.
That is, for every 1-cell $f\colon x' \to x$ there are isomorphisms 
\begin{equation}\label{eq:UnitaryAdjunctionData1cells}
\begin{tikzcd}
\langle Fx | y \rangle 
\arrow[r, "\sim"] 
\arrow[d,"\langle Ff|"'] 
& 
\langle x | Gy \rangle 
\arrow[d, "\langle f |"]
\\
\langle Fx' | y \rangle
\arrow[ur, Rightarrow, "\sim_f"description]
\arrow[r, "\sim"']
& 
\langle x' | Gy \rangle
\end{tikzcd}
\qquad=\qquad
\tikzmath{
\begin{scope}
\clip[rounded corners=5] (-.4,-.6) rectangle (1.4,1.6);
\fill[\AColor] (-.4,-.6) -- (0,-.6) -- (0,0) .. controls ++(0,.5) and ++(0,-.5) .. (1,1) -- (1,1.6)-- (-.4,1.6) -- cycle;
\fill[\BColor] (1.4,-.6) -- (0,-.6) -- (0,0) .. controls ++(0,.5) and ++(0,-.5) .. (1,1) -- (1,1.6) -- (1.4,1.6) -- cycle;
\fill[primedregion] (-.4,-.6) -- (1,-.6) -- (1,0) .. controls ++(0,.5) and ++(0,-.5) .. (0,1) -- (0,1.6) -- (-.4,1.6);
\fill[boxregion] (1.4,-.6) -- (1,-.6) -- (1,0) .. controls ++(0,.5) and ++(0,-.5) .. (0,1) -- (0,1.6) -- (1.4,1.6);
\end{scope}
\draw[thick] (1,-.6) node [below] {$\scriptstyle \sim$} -- (1,0) .. controls ++(0,.5) and ++(0,-.5) .. (0,1) -- (0,1.6) node [above] {$\scriptstyle \sim$};
\draw[thick,\XsColor] (0,-.6) node [below] {$\scriptstyle \langle Ff|y \rangle$}-- (0,0) .. controls ++(0,.5) and ++(0,-.5) .. (1,1) -- (1,1.6) node [above] {$\scriptstyle \langle f|Gy \rangle$};
\filldraw[thick,fill=white] (.5,.5) circle (.1cm);
}
\qquad\qquad
\begin{aligned}
\tikzmath{
\fill[lightgray!50, rounded corners=5pt] (0,0) rectangle (.6,.6);    
}
&=
x\threehilbtimes y 
\\
\tikzmath{
\fill[lightgray, rounded corners=5pt] (0,0) rectangle (.6,.6);    
}
&=
x'\threehilbtimes y 
\\
\tikzmath{
\filldraw[primedregion=white, rounded corners=5pt, draw=black, dotted] (0,0) rectangle (.6,.6);    
}
&=
\Hom_\fY(F(-)\to -)
\\
\tikzmath{
\fill[boxregion=white, rounded corners=5pt, draw=black, dotted] (0,0) rectangle (.6,.6);    
}
&=
\Hom_\fX(-\to G(-))
\end{aligned}
\end{equation}
and similarly for $g\colon y \to y'$.
Above, we use the overlay graphical calculus for the functors
\[
\Hom_\fY(F(-)\to -), \,\Hom_\fX(-\to G(-)) \colon \fX^{1\op}\threehilbtimes\fY\longrightarrow 2\Vect,
\]
where $\threehilbtimes$ is the Deligne product for 3-vector spaces defined in \cite{MR4677193}.
These isomorphisms are required to satisfy the condition that for every 2-cell $\alpha\colon f \Rightarrow f'\colon x' \to x$,
\begin{equation}\label{eq:UnitaryAdjunctionData2cells}
\tikzmath{
\begin{scope}
\clip[rounded corners=5] (-.4,-.6) rectangle (1.4,1.6);
\fill[\AColor] (-.4,-.6) -- (0,-.6) -- (0,0) .. controls ++(0,.5) and ++(0,-.5) .. (1,1) -- (1,1.6)-- (-.4,1.6) -- cycle;
\fill[\BColor] (1.4,-.6) -- (0,-.6) -- (0,0) .. controls ++(0,.5) and ++(0,-.5) .. (1,1) -- (1,1.6) -- (1.4,1.6) -- cycle;
\fill[primedregion] (-.4,-.6) -- (1,-.6) -- (1,0) .. controls ++(0,.5) and ++(0,-.5) .. (0,1) -- (0,1.6) -- (-.4,1.6);
\fill[boxregion] (1.4,-.6) -- (1,-.6) -- (1,0) .. controls ++(0,.5) and ++(0,-.5) .. (0,1) -- (0,1.6) -- (1.4,1.6);
\end{scope}
\draw[thick] (1,-.6) node [below] {$\scriptstyle \sim $} -- (1,0) .. controls ++(0,.5) and ++(0,-.5) .. (0,1) -- (0,1.6) node [above] {$\scriptstyle \sim $};
\draw[thick,\XsColor] (0,-.6) node [below] {$\scriptstyle \langle Ff| y \rangle$} -- (0,0);
\draw[thick,blue] (0,0) .. controls ++(0,.5) and ++(0,-.5) .. (1,1) -- (1,1.6) node [above] {$\scriptstyle \langle f'|Gy \rangle $};
\filldraw[thick,fill=white] (.5,.5) circle (.1cm);
\roundNbox{fill=white,draw=black}{(0,-.2)}{.2}{.15}{.25}{$\scriptstyle \langle F\alpha |$};
}
=
\begin{tikzcd}[column sep=25,row sep=45]
\langle Fx | y \rangle  
\arrow[r, "\sim"] 
\arrow[d,bend right=30, near start, "\langle Ff|"'] 
\arrow[d, bend left=30, near start, "{\langle Ff'|}"]
\arrow[d,phantom,"\overset{\langle F\alpha|}{\Rightarrow}"]
& 
\langle x | Gy \rangle 
\arrow[d, "{\langle f' |}"]
\\
\langle Fx' | y \rangle
\arrow[ur, Rightarrow, "{\sim_{f'}}"description]
\arrow[r, "\sim"']
& 
\langle x' | Gy \rangle
\end{tikzcd}
=
\begin{tikzcd}[column sep=25,row sep=45]
\langle Fx | y \rangle
\arrow[r, "\sim"] 
\arrow[d,"\langle Ff |"'] 
& 
\langle x | Gy \rangle 
\arrow[d, bend right=30, near end, "{\langle f |}"']
\arrow[d, bend left=30, near end, "{\langle f'|}"]
\arrow[d,phantom,"\overset{\langle \alpha |}{\Rightarrow}"]
\\
\langle Fx' | y \rangle
\arrow[ur, Rightarrow, "\sim_f"description]
\arrow[r, "\sim"']
& 
\langle x' | Gy \rangle
\end{tikzcd}
=
\tikzmath{
\begin{scope}
\clip[rounded corners=5] (-.4,-.6) rectangle (1.4,1.6);
\fill[\AColor] (-.4,-.6) -- (0,-.6) -- (0,0) .. controls ++(0,.5) and ++(0,-.5) .. (1,1) -- (1,1.6)-- (-.4,1.6) -- cycle;
\fill[\BColor] (1.4,-.6) -- (0,-.6) -- (0,0) .. controls ++(0,.5) and ++(0,-.5) .. (1,1) -- (1,1.6) -- (1.4,1.6) -- cycle;
\fill[primedregion] (-.4,-.6) -- (1,-.6) -- (1,0) .. controls ++(0,.5) and ++(0,-.5) .. (0,1) -- (0,1.6) -- (-.4,1.6);
\fill[boxregion] (1.4,-.6) -- (1,-.6) -- (1,0) .. controls ++(0,.5) and ++(0,-.5) .. (0,1) -- (0,1.6) -- (1.4,1.6);
\end{scope}
\draw[thick] (1,-.6) node [below] {$\scriptstyle \sim$} -- (1,0) .. controls ++(0,.5) and ++(0,-.5) .. (0,1) -- (0,1.6) node [above] {$\scriptstyle \sim$};
\draw[thick,\XsColor] (0,-.6) node [below] {$\scriptstyle \langle Ff| y \rangle $}-- (0,0) .. controls ++(0,.5) and ++(0,-.5) .. (1,1);
\draw[thick,blue] (1,1) -- (1,1.6) node [above] {$\scriptstyle \langle f' | Gy \rangle $};
\filldraw[thick,fill=white] (.5,.5) circle (.1cm);
\roundNbox{fill=white, draw=black}{(1,1.2)}{.2}{.1}{.1}{$\scriptstyle \langle \alpha | $};
}\!.
\end{equation}
Equivalently, there are natural transformations $\eta\colon \id_{\fX} \Rightarrow GF$ and $\varepsilon\colon FG \Rightarrow \id_{\fY}$, and invertible modifications $m\colon (\varepsilon F\circ F\eta) \Rrightarrow \id_F$ and $n\colon \id_G \Rrightarrow (G\varepsilon\circ \eta G)$ that satisfy the swallowtail equations \cite[\S 2.1.2]{1812.11933}.
\end{defn}

\begin{defn}
Let $F\in\Hom(\fX \to \fY)$.
We call a (not necessarily $\dag,\vee$-preserving) linear functor $F^*\colon\fY \to \fX$ its \emph{unitary 2-adjoint} if there is a collection of isometric equivalences of 2-Hilbert spaces 
$$
\fY(F(c)\to d) \isomeq \fX(c\to F^*(d))
$$
that are natural in $c$ and $d$ (as in Definition \ref{defn:2-adjoint}).
\end{defn}

\begin{ex}
For $x\in\fX$ and $\cM\in \Mod^\dag(\Omega_x)$, recall that
$|x\rangle_\Omega \cM := \cM \boxdot_{\Omega_x} x$
and
${}_\Omega\langle x|y := \fX(x \to y)_{\Omega_x}$.
We see that these functors are unitary adjoints by computing
\begin{align*}    
\Hom_{\Omega_x}(\fX(x \to y) \to \cM)
&\cong^\dag 
\cM \boxtimes_{\Omega_x} \fX(x \to y)^{\op}
&&\text{(\cite[Prop.~3.14]{MR5078555})}
\\&\cong^\dag
\cM \boxtimes_{\Omega_x} \fX(y \to x) 
\\&\cong^\dag
\fX(y \to \cM \boxdot_{\Omega_x} x)
&&\text{(Prop.~\ref{prop:2HilbertScalarsPullOut}).}
\end{align*}
\end{ex}

\begin{prop}
Unitary adjoints exist 
and are unique up to unique isometric equivalence
for every 1-morphism in $3\Hilb$.
In particular, the unitary adoint of $F\in \Hom(\fX\to \fY)$ automatically lies in $\Hom(\fY\to \fX)$.
\end{prop}
\begin{proof}
For a 1-morphism $F\colon  \fX\to \fY$,
the presheaf $\fX^{\op}\to 2\Hilb$ given by
$$
c\mapsto \fY(F(c)\to d)
$$
is isometrically representable by Corollary \ref{cor:AllPresheafsRepresentable}, which gives an explicit formula for $F^*$ by 
\begin{equation}
\label{eq:Unitary2AdjointExists}    
F^*(d) = \bigboxplus_{b\in\Irr \fX}\, \fY(F(b)\to d) \boxdot_{\Omega_b} b.
\end{equation}
This formula manifestly defines a functor in $\Hom(\fY\to \fX)$ as each $\fY(F(b) \to - )$ is $\dag,\vee$-preserving by Lemma \ref{lem:uaf-reps-unitary-adjs} and the relative scalar action $- \boxdot_{\Omega_b} b$ is automatically $\dag,\vee$-preserving (see Remark \ref{rem:RelativeScalarActionIsDagVeePreserving}).
By Corollary \ref{cor:isometric-yoneda}, the unitary 2-adjoint is unique up to unique isometric equivalence.
\end{proof}

We end this section with the following proposition, which can be viewed as an analog of the Closed Graph Theorem for operators between 3-Hilbert spaces, cf.~\cite[Rem.~2.4]{MR4750417}.

\begin{prop}\label{prop:adjointableimpliesdagpreserving}
If $F\colon  \fX\to \fY$ and $G\colon  \fY\to \fX$ are (not necessarily $\dag,\vee$-preserving) linear functors which are unitary adjoints,
then $F,G$ are automatically $\dag$-preserving.
\end{prop}

\begin{proof}
For every 2-cell $\alpha\colon f \Rightarrow f'\colon x' \to x$, and $y \in \fY$, observe
\[
\def\tkzscaletemp{0.7}
\scalebox{\tkzscaletemp}{$\tikzmath{
\begin{scope}
\clip[rounded corners=5] (-.4,-2.2) rectangle (2,1.6);
\fill[\AColor] (-.4,-2.22) -- (0.8,-2.2) -- (0.8,1.6) -- (-.4,1.6) ;
\fill[\BColor] (1.4,-2.2)  -- (0.8,-2.2) -- (0.8,1.6) -- (2,1.6) -- (2,-2.2);
\fill[primedregion] (-.4,-2.2) rectangle (2,1.6);
\end{scope}
\draw[thick,blue] (0.8,-2.2) node [below] {$\scriptstyle \langle Ff'| y \rangle $} -- (0.8,-.3);
\draw[thick,\XsColor] (0.8,-.3) -- (0.8,1.6) node [above] {$\scriptstyle \langle Ff| y \rangle $};
\roundNbox{fill=white,draw=black}{(0.8-.175,-.3)}{.25}{.1}{.45}{$\scriptstyle \langle (F\alpha)^\dag |$};
}$}
=
\scalebox{\tkzscaletemp}{$\tikzmath{
\begin{scope}
\clip[rounded corners=5] (-.4,-2.2) rectangle (2,1.6);
\fill[\AColor] (-.4,-2.22) -- (0,-2.2) -- (0,1.6) -- (-.4,1.6) ;
\fill[\BColor] (1.4,-2.2)  -- (0,-2.2) -- (0,1.6) -- (2,1.6) -- (2,-2.2);
\fill[primedregion] (-.4,-2.2) rectangle (2,1.6);
\fill[\BColor] (1.6,-1.3) arc (360:180:.6) -- (.4,.3) arc(180:0:.6) -- cycle;
\fill[boxregion] (1.6,-1.3) arc (360:180:.6) -- (.4,.3) arc(180:0:.6) -- cycle;
\end{scope}
\draw[thick,blue] (0,-2.2) node [below] {$\scriptstyle \langle Ff'| y \rangle $} -- (0,1.2);
\draw[thick,\XsColor] (0,1.2) -- (0,1.6) node [above] {$\scriptstyle \langle Ff| y \rangle $};
\draw[thick] (1.6,-1.3) arc (360:180:.6) -- (.4,.3) arc(180:0:.6) -- cycle;
\roundNbox{fill=white,draw=black}{(0,1.2)}{.25}{.1}{.45}{$\scriptstyle \langle (F\alpha)^\dag |$};
}$}
=
\scalebox{\tkzscaletemp}{$
\tikzmath{
\begin{scope}
\clip[rounded corners=5] (-.4,-2.2) rectangle (2,1.6);
\fill[\AColor] (-.4,-2.22) -- (0,-2.2) -- (0,-1.6) .. controls ++(0,.5) and ++(0,-.5) .. (1,-.6) -- (1,0) .. controls ++(0,.5) and ++(0,-.5) .. (0,1) -- (0,1.6) -- (-.4,1.6) ;
\fill[\BColor] (1.4,-2.2) -- (0,-2.2) -- (0,-1.6) .. controls ++(0,.5) and ++(0,-.5) .. (1,-.6) -- (1,0) .. controls ++(0,.5) and ++(0,-.5) .. (0,1) -- (0,1.6) -- (2,1.6) -- (2,-2.2);
\fill[primedregion] (-.4,-2.2) rectangle (2,1.6);
\begin{scope}
\clip (1.6,-1.6) arc (360:180:.3) -- (1,-1.6) .. controls ++(0,.5) and ++(0,-.5) .. (0,-.6)-- (0,0) .. controls ++(0,.5) and ++(0,-.5) .. (1,1) arc(180:0:.3) -- cycle;
\fill[\AColor] (-.4,-2.22) -- (0,-2.2) -- (0,-1.6) .. controls ++(0,.5) and ++(0,-.5) .. (1,-.6) -- (1,0) .. controls ++(0,.5) and ++(0,-.5) .. (0,1) -- (0,1.6) -- (-.4,1.6) ;
\fill[\BColor] (1.4,-2.2) -- (0,-2.2) -- (0,-1.6) .. controls ++(0,.5) and ++(0,-.5) .. (1,-.6) -- (1,0) .. controls ++(0,.5) and ++(0,-.5) .. (0,1) -- (0,1.6) -- (2,1.6) -- (2,-2.2);
\fill[boxregion] (1.6,-1.6) arc (360:180:.3) -- (1,-1.6) .. controls ++(0,.5) and ++(0,-.5) .. (0,-.6)-- (0,0) .. controls ++(0,.5) and ++(0,-.5) .. (1,1) arc(180:0:.3) -- cycle;
\end{scope}

\end{scope}
\draw[thick,blue] (0,-2.2) node [below] {$\scriptstyle \langle Ff'| y \rangle $} -- (0,-1.6) .. controls ++(0,.5) and ++(0,-.5) .. (1,-.6) -- (1,0) .. controls ++(0,.5) and ++(0,-.5) .. (0,1) -- (0,1.2);
\draw[thick,\XsColor] (0,1.2) -- (0,1.6) node [above] {$\scriptstyle \langle Ff| y \rangle $};
\draw[thick] (1.6,-1.6) arc (360:180:.3) -- (1,-1.6) .. controls ++(0,.5) and ++(0,-.5) .. (0,-.6)-- (0,0) .. controls ++(0,.5) and ++(0,-.5) .. (1,1) arc(180:0:.3) -- cycle;
\filldraw[thick,fill=white] (.5,.5) circle (.1cm);
\filldraw[thick,fill=white] (.5,-1.1) circle (.1cm);
\roundNbox{fill=white,draw=black}{(0,1.2)}{.25}{.1}{.45}{$\scriptstyle \langle (F\alpha)^\dag |$};
}$}
\underset{\eqref{eq:UnitaryAdjunctionData2cells}^\dag}{=}
\scalebox{\tkzscaletemp}{$\tikzmath{
\begin{scope}
\clip[rounded corners=5] (-.4,-2.2) rectangle (2,1.6);
\fill[\AColor] (-.4,-2.22) -- (0,-2.2) -- (0,-1.6) .. controls ++(0,.5) and ++(0,-.5) .. (1,-.6) -- (1,0) .. controls ++(0,.5) and ++(0,-.5) .. (0,1) -- (0,1.6) -- (-.4,1.6) ;
\fill[\BColor] (1.4,-2.2) -- (0,-2.2) -- (0,-1.6) .. controls ++(0,.5) and ++(0,-.5) .. (1,-.6) -- (1,0) .. controls ++(0,.5) and ++(0,-.5) .. (0,1) -- (0,1.6) -- (2,1.6) -- (2,-2.2);
\fill[primedregion] (-.4,-2.2) rectangle (2,1.6);
\begin{scope}
\clip (1.6,-1.6) arc (360:180:.3) -- (1,-1.6) .. controls ++(0,.5) and ++(0,-.5) .. (0,-.6)-- (0,0) .. controls ++(0,.5) and ++(0,-.5) .. (1,1) arc(180:0:.3) -- cycle;
\fill[\AColor] (-.4,-2.22) -- (0,-2.2) -- (0,-1.6) .. controls ++(0,.5) and ++(0,-.5) .. (1,-.6) -- (1,0) .. controls ++(0,.5) and ++(0,-.5) .. (0,1) -- (0,1.6) -- (-.4,1.6) ;
\fill[\BColor] (1.4,-2.2) -- (0,-2.2) -- (0,-1.6) .. controls ++(0,.5) and ++(0,-.5) .. (1,-.6) -- (1,0) .. controls ++(0,.5) and ++(0,-.5) .. (0,1) -- (0,1.6) -- (2,1.6) -- (2,-2.2);
\fill[boxregion] (1.6,-1.6) arc (360:180:.3) -- (1,-1.6) .. controls ++(0,.5) and ++(0,-.5) .. (0,-.6)-- (0,0) .. controls ++(0,.5) and ++(0,-.5) .. (1,1) arc(180:0:.3) -- cycle;
\end{scope}
\end{scope}
\draw[thick,blue] (0,-2.2) node [below] {$\scriptstyle \langle Ff'| y \rangle $} -- (0,-1.6) .. controls ++(0,.5) and ++(0,-.5) .. (1,-.6) -- (1,-.2);
\draw[thick,\XsColor] (1,-.2) -- (1,0) .. controls ++(0,.5) and ++(0,-.5) .. (0,1) -- (0,1.6) node [above] {$\scriptstyle \langle Ff| y \rangle $};
\draw[thick] (1.6,-1.6) arc (360:180:.3) -- (1,-1.6) .. controls ++(0,.5) and ++(0,-.5) .. (0,-.6)-- (0,0) .. controls ++(0,.5) and ++(0,-.5) .. (1,1) arc(180:0:.3) -- cycle;
\filldraw[thick,fill=white] (.5,.5) circle (.1cm);
\filldraw[thick,fill=white] (.5,-1.1) circle (.1cm);
\roundNbox{fill=white, draw=black}{(1,-.2)}{.25}{.05}{.05}{$\scriptstyle \langle \alpha^\dag | $};
}$}
\underset{\eqref{eq:UnitaryAdjunctionData2cells}}{=}
\scalebox{\tkzscaletemp}{$\tikzmath{
\begin{scope}
\clip[rounded corners=5] (-.4,-2.2) rectangle (2,1.6);
\fill[\AColor] (-.4,-2.22) -- (0,-2.2) -- (0,-1.6) .. controls ++(0,.5) and ++(0,-.5) .. (1,-.6) -- (1,0) .. controls ++(0,.5) and ++(0,-.5) .. (0,1) -- (0,1.6) -- (-.4,1.6) ;
\fill[\BColor] (1.4,-2.2) -- (0,-2.2) -- (0,-1.6) .. controls ++(0,.5) and ++(0,-.5) .. (1,-.6) -- (1,0) .. controls ++(0,.5) and ++(0,-.5) .. (0,1) -- (0,1.6) -- (2,1.6) -- (2,-2.2);
\fill[primedregion] (-.4,-2.2) rectangle (2,1.6);
\begin{scope}
\clip (1.6,-1.6) arc (360:180:.3) -- (1,-1.6) .. controls ++(0,.5) and ++(0,-.5) .. (0,-.6)-- (0,0) .. controls ++(0,.5) and ++(0,-.5) .. (1,1) arc(180:0:.3) -- cycle;
\fill[\AColor] (-.4,-2.22) -- (0,-2.2) -- (0,-1.6) .. controls ++(0,.5) and ++(0,-.5) .. (1,-.6) -- (1,0) .. controls ++(0,.5) and ++(0,-.5) .. (0,1) -- (0,1.6) -- (-.4,1.6) ;
\fill[\BColor] (1.4,-2.2) -- (0,-2.2) -- (0,-1.6) .. controls ++(0,.5) and ++(0,-.5) .. (1,-.6) -- (1,0) .. controls ++(0,.5) and ++(0,-.5) .. (0,1) -- (0,1.6) -- (2,1.6) -- (2,-2.2);
\fill[boxregion] (1.6,-1.6) arc (360:180:.3) -- (1,-1.6) .. controls ++(0,.5) and ++(0,-.5) .. (0,-.6)-- (0,0) .. controls ++(0,.5) and ++(0,-.5) .. (1,1) arc(180:0:.3) -- cycle;
\end{scope}
\end{scope}
\draw[thick,blue] (0,-2.2) node [below] {$\scriptstyle \langle Ff'| y \rangle $} --  (0,-1.8);
\draw[thick,\XsColor] (0,-1.8) -- (0,-1.6) .. controls ++(0,.5) and ++(0,-.5) .. (1,-.6) -- (1,0) .. controls ++(0,.5) and ++(0,-.5) .. (0,1) -- (0,1.6) node [above] {$\scriptstyle \langle Ff| y \rangle $};
\draw[thick] (1.6,-1.6) arc (360:180:.3) -- (1,-1.6) .. controls ++(0,.5) and ++(0,-.5) .. (0,-.6)-- (0,0) .. controls ++(0,.5) and ++(0,-.5) .. (1,1) arc(180:0:.3) -- cycle;
\filldraw[thick,fill=white] (.5,.5) circle (.1cm);
\filldraw[thick,fill=white] (.5,-1.1) circle (.1cm);
\roundNbox{fill=white, draw=black}{(0,-1.8)}{.25}{.1}{.45}{$\scriptstyle \langle F(\alpha^\dag) | $};
}$}
=
\scalebox{\tkzscaletemp}{$\tikzmath{
\begin{scope}
\clip[rounded corners=5] (-.4,-2.2) rectangle (2,1.6);
\fill[\AColor] (-.4,-2.22) -- (0,-2.2) -- (0,1.6) -- (-.4,1.6) ;
\fill[\BColor] (1.4,-2.2)  -- (0,-2.2) -- (0,1.6) -- (2,1.6) -- (2,-2.2);
\fill[primedregion] (-.4,-2.2) rectangle (2,1.6);
\fill[\BColor] (1.6,-.9) arc (360:180:.6) -- (.4,.7) arc(180:0:.6) -- cycle;
\fill[boxregion] (1.6,-.9) arc (360:180:.6) -- (.4,.7) arc(180:0:.6) -- cycle;
\end{scope}
\draw[thick,blue] (0,-2.2) node [below] {$\scriptstyle \langle Ff'| y \rangle $} --  (0,-1.8);
\draw[thick,\XsColor] (0,-1.8) -- (0,1.6) node [above] {$\scriptstyle \langle Ff| y \rangle $};
\draw[thick] (1.6,-.9) arc (360:180:.6) -- (.4,.7) arc(180:0:.6) -- cycle;
\roundNbox{fill=white,draw=black}{(0,-1.8)}{.25}{.1}{.45}{$\scriptstyle \langle F(\alpha^\dag) |$};
}$}
=
\scalebox{\tkzscaletemp}{$\tikzmath{
\begin{scope}
\clip[rounded corners=5] (-.4,-2.2) rectangle (2,1.6);
\fill[\AColor] (-.4,-2.22) -- (0.8,-2.2) -- (0.8,1.6) -- (-.4,1.6) ;
\fill[\BColor] (1.4,-2.2)  -- (0.8,-2.2) -- (0.8,1.6) -- (2,1.6) -- (2,-2.2);
\fill[primedregion] (-.4,-2.2) rectangle (2,1.6);
\end{scope}
\draw[thick,blue] (0.8,-2.2) node [below] {$\scriptstyle \langle Ff'| y \rangle $} -- (0.8,-.3);
\draw[thick,\XsColor] (0.8,-.3) -- (0.8,1.6) node [above] {$\scriptstyle \langle Ff| y \rangle $};
\roundNbox{fill=white,draw=black}{(0.8-.175,-.3)}{.25}{.1}{.45}{$\scriptstyle \langle F(\alpha^\dag) |$};
}$}
\]
and hence $(F\alpha)^\dag = F(\alpha^\dag)$ as the Yoneda embedding $\fY \hookrightarrow \Hom(\fY^{1\op} \to 2\Hilb)$ is faithful. 
\end{proof}

\subsection{Application: the space of ONBs is contractible}

As an application of unitary adjunction, we prove that given any two ONBs, there is a canonical transformation between them, and moreover, the space of ONBs is contractible.

\begin{defn}
Given a 3-Hilbert space $\fX$, ONBs form a 3-groupoid as follows.
\begin{itemize}
\item 
Objects consist of a list of $\rmH^*$-fusion categories $\{\cC_i\}$ and an isometric equivalence $F\colon  \bigboxplus_i \Mod^\dag(\cC_i)\to \fX$, which can be viewed as an object in the slice category $3\Hilb_{/\fX}$.
\item 
Given objects $(\{\cC_i\},F)$ and $(\{\cD_i\},G)$, a 1-morphism is an isometric equivalence $H\colon \bigboxplus_i \Mod^\dag(\cC_i) \to \bigboxplus_i \Mod^\dag(\cD_i)$ together with a pointwise isometric equivalence
$\eta\colon  F \Rightarrow GH$. In diagrams,
\[
\tikzmath{
\begin{scope}
\clip [rounded corners=5] (-.7,-.7) rectangle (.7,.7);
\fill[\AColor] (-.7,-.7) rectangle (.7,.7);
\fill[\Asplitcolor] (0,-.7) rectangle (.7,.7);
\fill[\ccol] (-.3,0) rectangle (.3,.7);
\end{scope}
\draw[thick] (-.3,.3) -- (-.3,.7) node[above] {$\scriptstyle G$};
\draw[thick] (.3,.3) -- (.3,.7) node[above] {$\scriptstyle H$};
\draw[thick] (0,-.3) -- (0,-.7) node[below] {$\scriptstyle F$};
\roundNbox{fill=white}{(0,0)}{.3}{.2}{.2}{$\eta$};
}
\qquad
\text{where}
\qquad
\begin{array}{r}
\fX = 
\tikzmath{
\begin{scope}
\clip [rounded corners=5] (-.3,-.3) rectangle (.3,.3);
\fill[\AColor] (-.3,-.3) rectangle (.3,.3);
\end{scope}
}~
\\[.5em]
\bigboxplus_i \Mod^\dag(\cC_i) = 
\tikzmath{
\begin{scope}
\clip [rounded corners=5] (-.3,-.3) rectangle (.3,.3);
\fill[\Asplitcolor] (-.3,-.3) rectangle (.3,.3);
\end{scope}
}~
\\[.5em]
\bigboxplus_i \Mod^\dag(\cD_i) = 
\tikzmath{
\begin{scope}
\clip [rounded corners=5] (-.3,-.3) rectangle (.3,.3);
\fill[\ccol] (-.3,-.3) rectangle (.3,.3);
\end{scope}
}\phantom{.}
\end{array}
\]
\item 
Given parallel 1-morphisms $(H,\eta)$ and $(K,\kappa)$ from $(\{\cC_i\},F)\to(\{\cD_i\},G)$,
a 2-morphism is an isometric natural equivalence $\alpha\colon H \Rightarrow K$ together with a unitary modification
\[
\mu\colon
\quad
\tikzmath{
\begin{scope}
\clip [rounded corners=5] (-.7,-.7) rectangle (.7,.7);
\fill[\AColor] (-.7,-.7) rectangle (.7,.7);
\fill[\Asplitcolor] (0,-.7) rectangle (.7,.7);
\fill[\ccol] (-.2,0) rectangle (.2,.7);
\end{scope}
\draw[thick] (-.2,.3) -- (-.2,.7) node[above] {$\scriptstyle G$};
\draw[thick] (.2,.3) -- (.2,.7) node[above] {$\scriptstyle K$};
\draw[thick] (0,-.3) -- (0,-.7) node[below] {$\scriptstyle F$};
\roundNbox{fill=white}{(0,0)}{.3}{.1}{.1}{$\kappa$};
}
\quad
\Rrightarrow
\quad
\tikzmath{
\begin{scope}
\clip [rounded corners=5] (-.7,-.7) rectangle (.7,1.7);
\fill[\AColor] (-.7,-.7) rectangle (.7,1.7);
\fill[\Asplitcolor] (0,-.7) rectangle (.7,1.7);
\fill[\ccol] (-.3,0) rectangle (.3,1.7);
\end{scope}
\draw[thick] (-.3,.3) -- (-.3,1.7) node[above] {$\scriptstyle G$};
\draw[thick] (.3,.3) -- node[right] {$\scriptstyle H$}  (.3,.7) ;
\draw[thick] (.3,1.3) -- (.3,1.7) node[above] {$\scriptstyle K$};
\draw[thick] (0,-.3) -- (0,-.7) node[below] {$\scriptstyle F$};
\roundNbox{fill=white}{(0,0)}{.3}{.2}{.2}{$\eta$};
\roundNbox{fill=white}{(.3,1)}{.3}{0}{0}{$\alpha$};
}
.
\]
\item 
Given parallel 2-morphisms $(\alpha,\mu)$ and $(\beta,\nu)$ from $(H,\eta) \to (K,\kappa)$,
a 3-morphism is a unitary modification $\xi\colon\alpha \Rrightarrow \beta$ satisfying $\mu\circ(\kappa G\xi) = \nu$.
\begin{equation}
\label{eq:Icecream}    
\begin{tikzcd}[row sep=-0.5in]
\tikzmath{
\begin{scope}
\clip [rounded corners=5] (-.7,-.7) rectangle (.7,.7);
\fill[\AColor] (-.7,-.7) rectangle (.7,.7);
\fill[\Asplitcolor] (0,-.7) rectangle (.7,.7);
\fill[\ccol] (-.2,0) rectangle (.2,.7);
\end{scope}
\draw[thick] (-.2,.3) -- (-.2,.7) node[above] {$\scriptstyle G$};
\draw[thick] (.2,.3) -- (.2,.7) node[above] {$\scriptstyle K$};
\draw[thick] (0,-.3) -- (0,-.7) node[below] {$\scriptstyle F$};
\roundNbox{fill=white}{(0,0)}{.3}{.1}{.1}{};
\node at (0,-.1) {$\kappa$};
}
\arrow[dr,"\mu"]
\arrow[rr,"\nu"]
&&
\tikzmath{
\begin{scope}
\clip [rounded corners=5] (-.7,-.7) rectangle (.7,1.7);
\fill[\AColor] (-.7,-.7) rectangle (.7,1.7);
\fill[\Asplitcolor] (0,-.7) rectangle (.7,1.7);
\fill[\ccol] (-.3,0) rectangle (.3,1.7);
\end{scope}
\draw[thick] (-.3,.3) -- (-.3,1.7) node[above] {$\scriptstyle G$};
\draw[thick] (.3,.3) -- node[right] {$\scriptstyle H$}  (.3,.7) ;
\draw[thick] (.3,1.3) -- (.3,1.7) node[above] {$\scriptstyle K$};
\draw[thick] (0,-.3) -- (0,-.7) node[below] {$\scriptstyle F$};
\roundNbox{fill=white}{(0,0)}{.3}{.2}{.2}{};
\node at (0,-.1) {$\eta$};
\roundNbox{fill=white}{(.3,1)}{.3}{0}{0}{};
\node at (.3,.9) {$\beta$};
}
\\
&
\tikzmath{
\begin{scope}
\clip [rounded corners=5] (-.7,-.7) rectangle (.7,1.7);
\fill[\AColor] (-.7,-.7) rectangle (.7,1.7);
\fill[\Asplitcolor] (0,-.7) rectangle (.7,1.7);
\fill[\ccol] (-.3,0) rectangle (.3,1.7);
\end{scope}
\draw[thick] (-.3,.3) -- (-.3,1.7) node[above] {$\scriptstyle G$};
\draw[thick] (.3,.3) -- node[right] {$\scriptstyle H$}  (.3,.7) ;
\draw[thick] (.3,1.3) -- (.3,1.7) node[above] {$\scriptstyle K$};
\draw[thick] (0,-.3) -- (0,-.7) node[below] {$\scriptstyle F$};
\roundNbox{fill=white}{(0,0)}{.3}{.2}{.2}{};
\node at (0,-.1) {$\eta$};
\roundNbox{fill=white}{(.3,1)}{.3}{0}{0}{};
\node at (.3,.9) {$\alpha$};
\draw[thick, rounded corners, dashed] (-.1,.6) rectangle (.7,1.4);
}
\arrow[ur,"\xi"]
\end{tikzcd}
\end{equation}
where we apply $\xi$ locally within the dashed box.
\end{itemize}
\end{defn}

\begin{thm}\label{spaceofONBscontractible}
The 3-groupoid of ONBs for $\fX$ is contractible.
\end{thm}
\begin{proof}
It is enough to show that there exists a 1-morphism between every pair of objects, 
a 2-morphism between every pair of parallel 1-morphisms, 
and a unique 3-morphism between every pair of parallel 2-morphisms.
It is a standard fact in category theory that the core of the full subcategory of the slice category over $\fX$ on isomorphisms is contractible (for any $n$-category).
Indeed, there is at most $\xi$ which can possibly satisfy \eqref{eq:Icecream} above.
In our case, we must check that the unique morphisms are isometric/unitary.

First, given objects $(\{\cC_i\},F)$ and $(\{\cD_i\},G)$, we can take unitary adjoints to get a functor $G^*F: (\{\cC_i\},F)\to (\{\cD_i\},G)$.
Since $F,G$ are isometric equivalences, so is $G^*F$.
Now we set $\eta: F\Rightarrow GG^*F$ as $\ev_G^\vee\circ \id_F$, i.e.,
$$
\tikzmath{
\begin{scope}
\clip [rounded corners=5] (-.7,-.7) rectangle (.7,.7);
\fill[\AColor] (-.7,-.7) rectangle (.7,.7);
\fill[\Asplitcolor] (.7,.7) -- (.3,.7) -- (.3,.3) .. controls ++(0,-.3) and ++(0,.3) .. (0,-.3) -- (0,-.7) -- (.7,-.7);
\fill[\ccol] (-.4,.7) -- (-.4,.3) arc(180:360:.2) -- (0,.3) -- (0,.7);
\end{scope}
\draw[thick] (-.4,.7) node[above] {$\scriptstyle G$} -- (-.4,.3) arc(180:360:.2) -- (0,.3) -- (0,.7) node[above] {$\scriptstyle G^*$};
\draw[thick] (.3,.3) -- (.3,.7) node[above] {$\scriptstyle F$};
\draw[thick] (.3,.7) -- (.3,.3) .. controls ++(0,-.3) and ++(0,.3) .. (0,-.3) -- (0,-.7) node[below] {$\scriptstyle F$};
}
$$
(Note here that functor composition is read \emph{right-to-left}, opposite to our usual left-to-right composition rule.)

Second, given two parallel 1-morphisms
$(H,\eta), (K,\kappa) :(\{\cC_i\},F)\to (\{\cD_i\},G)$,
we have the 2-morphism
$$
\tikzmath{
\begin{scope}
\clip [rounded corners=5] (-1.5,-.3) rectangle (.7,2.5);
\fill[\ccol] (-1.5,-.3) rectangle (.7,2.5);
\fill[\AColor] (-0.3,0.3) arc(0:-180:.3)--++(0,1.65) arc(180:0:.3);
\fill[\AColor] (-0.3,1.6) rectangle (.3,.6);
\fill[\Asplitcolor] (0,-.3) rectangle (.7,2.5);
\end{scope}
\draw[thick] (0,-.3)--node[right]{\scriptsize$F$}node[right,pos=.925]{\scriptsize$K$}node[right,pos=.075]{\scriptsize$H$} (0,2.5);
\draw[thick] (-0.3,0.3) arc(0:-180:.3)--node[left]{\scriptsize$G^*$}++(0,1.65) arc(180:0:.3);
\roundNbox{fill=white}{(-.15,.6)}{.3}{.3}{.3}{$\eta^*$};
\roundNbox{fill=white}{(-.15,1.65)}{.3}{.3}{.3}{$\kappa$};
}
:
\;
H\longrightarrow K
$$
which is automatically an isometric natural isomorphism as it is built from composites of whiskerings of isometric natural isomorphisms.
This natural isomorphism admits the unitary modification
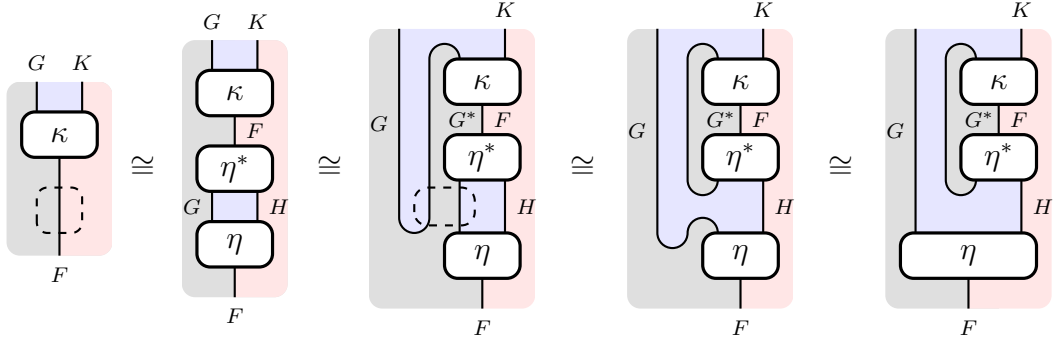
\[
\tikzmath{
\begin{scope}
\clip [rounded corners=5] (-.7,-1.6) rectangle (.7,.7);
\fill[\AColor] (-.7,-1.6) rectangle (.7,.7);
\fill[\Asplitcolor] (0,-1.6) rectangle (.7,.7);
\fill[\ccol] (-.3,0) rectangle (.3,.7);
\end{scope}
\draw[thick] (-.3,.3) -- (-.3,.7) node[above] {$\scriptstyle G$};
\draw[thick] (.3,.3) -- (.3,.7) node[above] {$\scriptstyle K$};
\draw[thick] (0,-.3) -- (0,-1.6) node[below] {$\scriptstyle F$};
\roundNbox{fill=white}{(0,0)}{.3}{.2}{.2}{$\kappa$};
\draw[thick, dashed, rounded corners] (-.3,-1.3) rectangle (.3,-.7);
}
\;
\cong
\;
\tikzmath{
\begin{scope}
\clip [rounded corners=5] (-.7,-.7) rectangle (.7,2.7);
\fill[\AColor] (-.7,-.7) rectangle (.7,2.7);
\fill[\Asplitcolor] (0,-.7) rectangle (.7,2.7);
\fill[\ccol] (-.3,.3) rectangle (.3,.7);
\fill[\ccol] (-.3,2.3) rectangle (.3,2.7);
\end{scope}
\draw[thick] (-.3,.3) -- node[left] {$\scriptstyle G$} (-.3,.7) ;
\draw[thick] (.3,.3) -- node[right] {$\scriptstyle H$} (.3,.7) ;
\draw[thick] (0,-.3) -- (0,-.7) node[below] {$\scriptstyle F$};
\draw[thick] (0,1.3) -- node[right] {$\scriptstyle F$} (0,1.7) ;
\draw[thick] (-.3,2.3) -- (-.3,2.7) node[above] {$\scriptstyle G$};
\draw[thick] (.3,2.3) -- (.3,2.7) node[above] {$\scriptstyle K$};
\roundNbox{fill=white}{(0,0)}{.3}{.2}{.2}{$\eta$};
\roundNbox{fill=white}{(0,1)}{.3}{.2}{.2}{$\eta^*$};
\roundNbox{fill=white}{(0,2)}{.3}{.2}{.2}{$\kappa$};
}
\;
\cong
\;
\tikzmath{
\begin{scope}
\clip [rounded corners=5] (-1.5,-1) rectangle (.7,2.7);
\fill[\AColor] (-1.5,-1) rectangle (0,2.7);
\fill[\Asplitcolor] (0,-1) rectangle (.7,2.7);
\fill[\ccol] (-.3,0) rectangle (.3,.7);
\fill[\ccol] (-.3,2.3)  arc(0:180:.2cm) -- (-.7,.2) arc(0:-180:.2cm) -- (-1.1,2.7) -- (.3,2.7) -- (.3,2.3);
\end{scope}
\draw[thick] (-.3,0) -- (-.3,.7) ;
\draw[thick] (.3,0) -- node[right] {$\scriptstyle H$} (.3,.7) ;
\draw[thick] (0,-.6) -- (0,-1) node[below] {$\scriptstyle F$};
\draw[thick] (0,1.3) -- node[right] {$\scriptstyle F$} (0,1.7);
\draw[thick] (.3,2.3) -- (.3,2.7) node[above] {$\scriptstyle K$};
\draw[thick] (-.3,2.3) arc(0:180:.2cm) --node[right, yshift=.25cm,xshift=.1cm] {$\scriptstyle G^*$} (-.7,.2) arc(0:-180:.2cm) --node[left] {$\scriptstyle G$} (-1.1,2.7);
\roundNbox{fill=white}{(0,-.3)}{.3}{.2}{.2}{$\eta$};
\roundNbox{fill=white}{(0,1)}{.3}{.2}{.2}{$\eta^*$};
\roundNbox{fill=white}{(0,2)}{.3}{.2}{.2}{$\kappa$};
\draw[thick,dashed] [rounded corners=5] (-0.9,0.1) rectangle (-.1,.6);
}
\;\,
\cong
\;\,
\tikzmath{
\begin{scope}
\clip [rounded corners=5] (-1.5,-1) rectangle (.7,2.7);
\fill[\AColor] (-1.5,-1) rectangle (0,2.7);
\fill[\Asplitcolor] (0,-1) rectangle (.7,2.7);
\fill[\ccol] (-.3,0) arc(0:180:.2cm) arc(0:-180:.2cm) -- (-1.1,2.7) -- (.3,2.7) -- (.3,2.3) -- (0,1.7) -- (0,1.3) -- (.3,.7) -- (.3,0);
\fill[\AColor] (-.3,2.3) arc(0:180:.2cm) -- (-.7,.7) arc(-180:0:.2cm) -- (0,1.3) -- (0,1.7);
\end{scope}
\draw[thick] (.3,0) -- node[right] {$\scriptstyle H$} (.3,.7) ;
\draw[thick] (0,-.6) -- (0,-1) node[below] {$\scriptstyle F$};
\draw[thick] (0,1.3) -- node[right] {$\scriptstyle F$} (0,1.7);
\draw[thick] (.3,2.3) -- (.3,2.7) node[above] {$\scriptstyle K$};
\draw[thick] (-.3,2.3) arc(0:180:.2cm) --node[right, xshift=.1cm] {$\scriptstyle G^*$} (-.7,.7) arc(-180:0:.2cm);
\draw[thick] (-.3,0) arc(0:180:.2cm) arc(0:-180:.2cm) --node[left] {$\scriptstyle G$} (-1.1,2.7);
\roundNbox{fill=white}{(0,-.3)}{.3}{.2}{.2}{$\eta$};
\roundNbox{fill=white}{(0,1)}{.3}{.2}{.2}{$\eta^*$};
\roundNbox{fill=white}{(0,2)}{.3}{.2}{.2}{$\kappa$};
}
\;\,
\cong
\;\,
\tikzmath{
\begin{scope}
\clip [rounded corners=5] (-1.5,-1) rectangle (.7,2.7);
\fill[\AColor] (-1.5,-1) rectangle (0,2.7);
\fill[\Asplitcolor] (-.4,-1) rectangle (.7,2.7);
\fill[\ccol] (-1.1,0) -- (-1.1,2.7) -- (.3,2.7) -- (.3,2.3) -- (0,1.7) -- (0,1.3) -- (.3,.7) -- (.3,0);
\fill[\AColor] (-.3,2.3) arc(0:180:.2cm) -- (-.7,.7) arc(-180:0:.2cm) -- (0,1.3) -- (0,1.7);
\end{scope}
\draw[thick] (.3,0) -- node[right] {$\scriptstyle H$} (.3,.7) ;
\draw[thick] (-.4,-.6) -- (-.4,-1) node[below] {$\scriptstyle F$};
\draw[thick] (0,1.3) -- node[right] {$\scriptstyle F$} (0,1.7);
\draw[thick] (.3,2.3) -- (.3,2.7) node[above] {$\scriptstyle K$};
\draw[thick] (-.3,2.3) arc(0:180:.2cm) --node[right, xshift=.1cm] {$\scriptstyle G^*$} (-.7,.7) arc(-180:0:.2cm);
\draw[thick] (-1.1,0) --node[left] {$\scriptstyle G$} (-1.1,2.7);
\roundNbox{fill=white}{(0,-.3)}{.3}{1}{.2}{$\eta$};
\roundNbox{fill=white}{(0,1)}{.3}{.2}{.2}{$\eta^*$};
\roundNbox{fill=white}{(0,2)}{.3}{.2}{.2}{$\kappa$};
}
\]
where unitarity follows from the fact that $G$, its evaluations and coevaluations, and $\eta$ are isometric equivalences/isomorphisms.

Finally, one checks that the unique $\xi$ satisfying \eqref{eq:Icecream} is manifestly unitary, as it is built from isometric equivalences and unitaries.
\end{proof}

\section{3Hilb is self-enriched}
\label{sec:3HilbSelfEnriched}
We now use the notion of an ONB for a 3-Hilbert space $\fX$ to endow $\Hom(\fX\to \fY)$, the $\dag,\vee$-preserving 2-functors $\fX\to \fY$, with a UAF and spherical weight, promoting the $\rmC^*$ 2-category $\Hom(\fX\to \fY)$ to a 3-Hilbert space.
We then prove our definition satisfies several desiderata for the `correct' 3-Hilbert space structure.

\subsection{Self-enrichment is well-defined}
\label{subsec:selfenrichmentwelldefined}
\begin{defn} \label{defn:trace and UAF on Hom}
Suppose $\fX,\fY$ are 3-Hilbert spaces.
On $\Hom(\fX\to \fY)$, we define the weight
\begin{equation}
\label{eq:3HilbOnFun}
\Psi_F^{\Hom}(m\colon \id_F\Rrightarrow \id_F)
\coloneq
\sum_{b\in\pi_0\fX}
\frac{d_b}{D_{\Omega_b}}
\Psi^\fY_{F(b)}(m_b)
\qquad\qquad
\forall\,F\colon \fX\to \fY.
\end{equation}
Here, $d_b\coloneq d_{1_b} = \Psi(\id_{1_b})$ 
and
$D_{\Omega_b} \coloneq \sum_{X\in \pi_0\Omega_b} d_X^2$.
The UAF $\vee$ on $\Hom(\fX \to \fY)$ is determined on a transformation $\alpha \colon F \Rightarrow G$ as follows.
\begin{itemize}
\item For $a \in \fX$, we define the component of $\alpha^\vee $ at $a$ by 
$( \alpha^\vee)_a \coloneqq (\alpha_a)^\vee \in \fY(Ga \to Fa)$. 
\item For ${}_aX_b \in \fX(a \to b)$, we define the naturator unitary $(\alpha^\vee)_X\coloneq(\alpha_X^\vee)^\dag$ associated to ${}_aX_b$:
\[
\tikzmath{
\begin{scope}
\clip[rounded corners=5] (-.3,-.5) rectangle (1.3,1.5);
\fill[lightgray] (0,-.5)  -- (0,0) .. controls ++(0,.5) and ++(0,-.5) .. (1,1) -- (1,1.5) -- (1.3,1.5) -- (1.3,-.5) -- cycle;
\fill[lightgray!50] (0,-.5)  -- (0,0) .. controls ++(0,.5) and ++(0,-.5) .. (1,1) -- (1,1.5) -- (-.3,1.5) -- (-.3,-.5) -- cycle;
\fill[pattern=primedbox] (1,-.5) -- (1,0) .. controls ++(0,.5) and ++(0,-.5) .. (0,1) -- (0,1.5) -- (-.3,1.5) -- (-.3,-.5) -- cycle;
\fill[pattern=primeddots] (1,-.5) -- (1,0) .. controls ++(0,.5) and ++(0,-.5) .. (0,1) -- (0,1.5) -- (1.3,1.5) -- (1.3,-.5) -- cycle;
\end{scope}
\draw[thick] (1,-.5) node [below] {$\scriptstyle \alpha^\vee_b$} -- (1,0) .. controls ++(0,.5) and ++(0,-.5) .. (0,1) -- (0,1.5) node [above] {$\scriptstyle \alpha^\vee_a$};
\draw[thick,\XsColor] (0,-.5) node [below] {$\scriptstyle GX$}-- (0,0) .. controls ++(0,.5) and ++(0,-.5) .. (1,1) -- (1,1.5) node [above] {$\scriptstyle FX$};
\roundNbox{fill=white}{(.5,.5)}{.3}{0}{0}{$\scriptstyle\alpha^\vee_X$};
}
\;\coloneqq\;
\left(
\def\tkzscaletemp{0.7}
\scalebox{\tkzscaletemp}{$
\tikzmath{
\begin{scope}
\clip[rounded corners=5] (-1.3,-1) rectangle (2.3,2);
\fill[lightgray!50] (0,-1) -- (0,0) .. controls ++(0,.5) and ++(0,-.5) .. (1,1) -- (1,2) -- (-1.3,2) -- (-1.3,-1) -- cycle;
\fill[lightgray] (0,-1) -- (0,0) .. controls ++(0,.5) and ++(0,-.5) .. (1,1) -- (1,2) -- (2.3,2) -- (2.3,-1) -- cycle;
\fill[pattern=primedbox] (2,2) -- (2,0) arc (360:180:.5) .. controls ++(0,.5) and ++(0,-.5) .. (0,1) arc (0:180:.5) -- (-1,-1) -- (-1.3,-1) -- (-1.3,2) --cycle;
\fill[pattern=primeddots] (2,2) -- (2,0) arc (360:180:.5) .. controls ++(0,.5) and ++(0,-.5) .. (0,1) arc (0:180:.5) -- (-1,-1) -- (2.3,-1) -- (2.3,2) --cycle;
\end{scope}
\draw[very thick, \XsColor] (0,-1) node [below] {$FX$}-- (0,0) .. controls ++(0,.5) and ++(0,-.5) .. (1,1) -- (1,2) node [above] {$GX$};
\draw[very thick] (2,2) node [above] {$ \alpha^\vee_b$} -- (2,0) arc (360:180:.5) .. controls ++(0,.5) and ++(0,-.5) .. (0,1) arc (0:180:.5) -- (-1,-1) node [below] {$ \alpha^\vee_a$};
\roundNbox{fill=white,ultra thick}{(.5,.5)}{.4}{0}{0}{$\alpha_X$};
}$}
\right)^\dag
\;=\;
\def\tkzscaletemp{0.7}
\scalebox{\tkzscaletemp}{$
\tikzmath{
\tikzset{yscale=-1}
\begin{scope}
\clip[rounded corners=5] (-1.3,-1) rectangle (2.3,2);
\fill[lightgray!50] (0,-1) -- (0,0) .. controls ++(0,.5) and ++(0,-.5) .. (1,1) -- (1,2) -- (-1.3,2) -- (-1.3,-1) -- cycle;
\fill[lightgray] (0,-1) -- (0,0) .. controls ++(0,.5) and ++(0,-.5) .. (1,1) -- (1,2) -- (2.3,2) -- (2.3,-1) -- cycle;
\fill[pattern=primedbox] (2,2) -- (2,0) arc (360:180:.5) .. controls ++(0,.5) and ++(0,-.5) .. (0,1) arc (0:180:.5) -- (-1,-1) -- (-1.3,-1) -- (-1.3,2) --cycle;
\fill[pattern=primeddots] (2,2) -- (2,0) arc (360:180:.5) .. controls ++(0,.5) and ++(0,-.5) .. (0,1) arc (0:180:.5) -- (-1,-1) -- (2.3,-1) -- (2.3,2) --cycle;
\end{scope}
\draw[very thick, \XsColor] (0,-1) node [above] {$FX$}-- (0,0) .. controls ++(0,.5) and ++(0,-.5) .. (1,1) -- (1,2) node[below] {$GX$};
\draw[very thick] (2,2) node [below] {$\alpha^\vee_b$} -- (2,0) arc (360:180:.5) .. controls ++(0,.5) and ++(0,-.5) .. (0,1) arc (0:180:.5) -- (-1,-1) node [above] {$\alpha^\vee_a$};
\roundNbox{fill=white,ultra thick}{(.5,.5)}{.4}{0}{0}{$\alpha_X^\dag$};
}$}
.
\]

\item 
The evaluation $\ev_\alpha\colon  \alpha^\vee\otimes \alpha\Rrightarrow \id_{G}$
and coevaluation 
$\coev_\alpha \colon  \id_F\Rrightarrow \alpha\otimes \alpha^\vee$ 
modifications are given componentwise by
$(\ev_\alpha)_a\coloneq \ev_{\alpha_a}$ 
and 
$(\coev_\alpha)_a\coloneq\coev_{\alpha_a}$.
These formulas define modifications by the monoidality transformation axiom for $\alpha$.
The zig-zag axioms are verified componentwise.
\end{itemize}
\end{defn}

\begin{lem}
$\vee, \Psi^{\Hom}$ are well-defined.
\end{lem}
\begin{proof}
\item[\underline{Step 1:}]
The formula \eqref{eq:3HilbOnFun} for $\Psi_F^\Fun$ is independent of the choice of ONB for $\fX$.
\begin{proof}
Suppose $\{c\}$ is another choice of ONB for $\fX$.
If $b\sim c$, then 
$$
d_b^{-1} \Psi^\fY_{F(b)}(\mu_b)
=
d_c^{-1} \Psi^\fY_{F(c)}(\mu_c).
$$
By \cite[p16, before (8)]{MR5078555}, since $\Omega_b$ is fusion,
$$
d_b^{-1} D_{\Omega_b}
=
\FPdim(\Omega_b)\cdot d_b.
$$
Since $\Omega_b$ and $\Omega_c$ 
are Morita equivalent $\rmH^*$-fusion categories
as $b\sim c$, we then have
\begin{align*}
\frac{d_b}{D_{\Omega_b}}
\Psi^\fY_{F(b)}(\mu_b)
&=
\frac{1}{d_b\cdot \FPdim \Omega_b}
\Psi^\fY_{F(b)}(\mu_b)
\\&=
\frac{1}{d_c\cdot \FPdim \Omega_{c}}
\Psi^\fY_{F(c)}(\mu_c)
\\&=
\frac{d_c}{D_{\Omega_c}}
\Psi^\fY_{F(c)}(\mu_c).
\qedhere
\end{align*}
\end{proof}

\item[\underline{Step 2:}]
$\vee$ is a UAF on $\Hom(\fX\to \fY)$.
\vspace*{.2cm}\\
\noindent\emph{Proof.}
We must check that 
(a) for all modifications $m\colon  \alpha\Rrightarrow \beta$ for parallel transformations $\alpha,\beta\colon  F\Rightarrow G$, $m^{\vee\dag}=m^{\dag\vee}$, 
and
(b) the canonical tensorator $\nu_{\alpha,\beta}\colon \alpha^\vee\beta^\vee\Rrightarrow (\beta\alpha)^\vee$
is unitary.
These are both verified componentwise, 
and both follow as $\ev,\coev$ for transformations were defined componentwise.
\end{proof}

\begin{lem}
$(\vee, \Psi^{\Hom})$ endow $\Hom(\fX\to \fY)$ with the structure of a pre-3-Hilbert space.
\end{lem}
\begin{proof}
We must show $\Psi^{\Hom}$ is a spherical weight for $\vee$.
Suppose $F,G\in \Hom(\fX\to \fY)$, $\alpha\colon  F\Rightarrow G$.
We must prove that
$$
\Psi_{G}^{\Hom}\!\left(
\tikzmath{
\fill[boxregion=white, rounded corners=5pt, dotted, draw=black] (-.9,-.9) rectangle (.6,.9);
\filldraw[primedregion=white, draw=black, thick] (0,.3) arc(0:180:.3cm) -- (-.6,-.3) arc(-180:0:.3cm);
\roundNbox{fill=white}{(0,0)}{.3}{0}{0}{$m$}
}
\right)
\overset{?}{=}
\Psi_{F}^{\Hom}\!\left(
\tikzmath{
\fill[primedregion=white, rounded corners=5pt, dotted, draw=black] (-.6,-.9) rectangle (.9,.9);
\filldraw[boxregion=white, draw=black, thick] (0,.3) arc(180:0:.3cm) -- (.6,-.3) arc(0:-180:.3cm);
\roundNbox{fill=white}{(0,0)}{.3}{0}{0}{$m$}
}
\right)
\qquad\qquad\qquad
\forall\, m\in \Omega_\alpha.
$$
By definition \eqref{eq:3HilbOnFun}, the above equation is equivalent to
$$
0
\overset{?}{=}
\sum_{b \in \pi_0\fX} \frac{d_b}{D_{\Omega_b}}
\left(
\Psi_{G(b)}^{\fY}\!\left(
\tikzmath{
\fill[boxregion=lightgray, rounded corners=5pt] (-.9,-.9) rectangle (.6,.9);
\filldraw[primedregion=lightgray, draw=black, thick] (0,.3) arc(0:180:.3cm) -- (-.6,-.3) arc(-180:0:.3cm);
\roundNbox{fill=white}{(0,0)}{.3}{0}{0}{$m_b$}
}
\right)
-
\Psi_{F(b)}^{\fY}\!\left(
\tikzmath{
\fill[primedregion=lightgray, rounded corners=5pt] (-.6,-.9) rectangle (.9,.9);
\filldraw[boxregion=lightgray, draw=black, thick] (0,.3) arc(180:0:.3cm) -- (.6,-.3) arc(0:-180:.3cm);
\roundNbox{fill=white}{(0,0)}{.3}{0}{0}{$m_b$}
}
\right)
\right),
$$
which clearly holds by sphericality of $\Psi^\fY$ and the definition of $\alpha^\vee$.
\end{proof}

\begin{prop}
\label{prop:EnrichmentWellDefined}
$(\vee,\Psi_F^{\Hom})$ endows $\Hom(\fX\to \fY)$ with the structure of a 3-Hilbert space.
\end{prop}
\begin{proof}
We prove that $\Hom(\fX\to \fY)$ admits Hilbert direct sums and is $\rmH^*$-monad complete.

\item[\underline{Admits $\boxplus$:}]
For $F,G \in \Hom(\fX \to \fY)$,
we define $F \boxplus G$ as follows.
\begin{itemize}
\item 
Since $\fY$ admits Hilbert direct sums, 
for $a\in\fX$, we may define
$$
(F\boxplus G)(a)
\coloneq 
F(a) \boxplus G(a).
$$
In the overlay graphical calculus, we denote
$$
F=
\tikzmath{
\filldraw[primedregion=white, rounded corners=5pt, draw=black, dotted] (0,0) rectangle (.6,.6);    
}
\qquad\qquad
G=
\tikzmath{
\fill[boxregion=white, rounded corners=5pt, draw=black, dotted] (0,0) rectangle (.6,.6);    
}
\qquad\qquad
F\boxplus G=
\tikzmath{
\fill[plusregion=white, rounded corners=5pt, draw=black, dotted] (0,0) rectangle (.6,.6);    
}\,.
$$
For the sequel, we denote by
$I_a\colon  F(a)\hookrightarrow F(a)\boxplus G(a)$
and
$J_a\colon  G(a)\hookrightarrow F(a)\boxplus G(a)$
the isometries witnessing this Hilbert direct sum.
Graphically, we write
$$
I_a
=
\tikzmath{
\begin{scope}
\clip[rounded corners=5pt] (0,0) rectangle (.6,.6);
\fill[primedregion=lightgray!50] (0,0) rectangle (.3,.6);    
\fill[plusregion=lightgray!50] (.3,0) rectangle (.6,.6);    
\end{scope}
\draw[thick] (.3,0) -- (.3,.6);
}
\qquad\qquad
I_a^\vee
=
\tikzmath{
\begin{scope}
\clip[rounded corners=5pt] (0,0) rectangle (.6,.6);
\fill[plusregion=lightgray!50] (0,0) rectangle (.3,.6);    
\fill[primedregion=lightgray!50] (.3,0) rectangle (.6,.6);    
\end{scope}
\draw[thick] (.3,0) -- (.3,.6);
}
\qquad\qquad
J_a
=
\tikzmath{
\begin{scope}
\clip[rounded corners=5pt] (0,0) rectangle (.6,.6);
\fill[boxregion=lightgray!50] (0,0) rectangle (.3,.6);    
\fill[plusregion=lightgray!50] (.3,0) rectangle (.6,.6);    
\end{scope}
\draw[thick] (.3,0) -- (.3,.6);
}
\qquad\qquad
J_a^\vee
=
\tikzmath{
\begin{scope}
\clip[rounded corners=5pt] (0,0) rectangle (.6,.6);
\fill[plusregion=lightgray!50] (0,0) rectangle (.3,.6);    
\fill[boxregion=lightgray!50] (.3,0) rectangle (.6,.6);    
\end{scope}
\draw[thick] (.3,0) -- (.3,.6);
}
$$
\item 
For a 1-morphism ${}_aX_b$, we define
$$
(F \boxplus G)({}_aX_b) 
=
\tikzmath{
\begin{scope}
\clip[rounded corners=5pt] (0,0) rectangle (.6,.6);
\fill[plusregion=lightgray!50] (0,0) rectangle (.3,.6);    
\fill[plusregion=lightgray] (.3,0) rectangle (.6,.6);    
\end{scope}
\draw[thick, \XsColor] (.3,0) -- (.3,.6);
}
\coloneq
\tikzmath{
\begin{scope}
\clip[rounded corners=5pt] (-.6,0) rectangle (.6,.6);
\fill[plusregion=lightgray!50] (-.3,0) rectangle (-.6,.6);    
\fill[primedregion=lightgray!50] (0,0) rectangle (-.3,.6);    
\fill[primedregion=lightgray] (0,0) rectangle (.3,.6);    
\fill[plusregion=lightgray] (.3,0) rectangle (.6,.6);    
\end{scope}
\draw[thick] (.3,0) -- (.3,.6);
\draw[thick, \XsColor] (0,0) -- (0,.6);
\draw[thick] (-.3,0) -- (-.3,.6);
}
\oplus
\tikzmath{
\begin{scope}
\clip[rounded corners=5pt] (-.6,0) rectangle (.6,.6);
\fill[plusregion=lightgray!50] (-.3,0) rectangle (-.6,.6);    
\fill[boxregion=lightgray!50] (0,0) rectangle (-.3,.6);    
\fill[boxregion=lightgray] (0,0) rectangle (.3,.6);    
\fill[plusregion=lightgray] (.3,0) rectangle (.6,.6);    
\end{scope}
\draw[thick] (.3,0) -- (.3,.6);
\draw[thick, \XsColor] (0,0) -- (0,.6);
\draw[thick] (-.3,0) -- (-.3,.6);
}
=
(I^{\vee}_{a} \otimes F(X) \otimes I_{b}) \oplus (J^{\vee}_{a} \otimes G(X) \otimes J_{b}).
$$
\item 
For a 2-morphism $f\in \fX( {}_aX_b\Rightarrow {}_aY_b)$, we define
$$
(F\boxplus G)(f)
=
\tikzmath{
\begin{scope}
\clip[rounded corners=5pt] (-.6,-.6) rectangle (.6,.6);
\fill[plusregion=lightgray!50] (-.6,-.6) rectangle (0,.6);    
\fill[plusregion=lightgray] (0,-.6) rectangle (.6,.6);    
\end{scope}
\draw[dotted, rounded corners=5pt] (-.6,-.6) rectangle (.6,.6);
\draw[thick, \XsColor] (0,-.6) -- (0,-.3);
\draw[thick, \YsColor] (0,.6) -- (0,.3);
\roundNbox{plusregion=white, draw=black}{(0,0)}{.3}{0}{0}{$f$}
}
\coloneq
\tikzmath{
\begin{scope}
\clip[rounded corners=5pt] (-.9,-.6) rectangle (.9,.6);
\fill[plusregion=lightgray!50] (-.9,-.6) rectangle (-.6,.6);    
\fill[primedregion=lightgray!50] (-.6,-.6) rectangle (0,.6);    
\fill[primedregion=lightgray] (0,-.6) rectangle (.6,.6);  
\fill[plusregion=lightgray] (.9,-.6) rectangle (.6,.6);  
\end{scope}
\draw[dotted, rounded corners=5pt] (-.6,-.6) rectangle (.6,.6);
\draw[thick] (-.6,-.6) -- (-.6,.6);
\draw[thick] (.6,-.6) -- (.6,.6);
\draw[thick, \XsColor] (0,-.6) -- (0,-.3);
\draw[thick, \YsColor] (0,.6) -- (0,.3);
\roundNbox{primedregion=white, draw=black}{(0,0)}{.3}{0}{0}{$f$}
}
\oplus
\tikzmath{
\begin{scope}
\clip[rounded corners=5pt] (-.9,-.6) rectangle (.9,.6);
\fill[plusregion=lightgray!50] (-.9,-.6) rectangle (-.6,.6);    
\fill[boxregion=lightgray!50] (-.6,-.6) rectangle (0,.6);    
\fill[boxregion=lightgray] (0,-.6) rectangle (.6,.6);  
\fill[plusregion=lightgray] (.9,-.6) rectangle (.6,.6);  
\end{scope}
\draw[dotted, rounded corners=5pt] (-.6,-.6) rectangle (.6,.6);
\draw[thick] (-.6,-.6) -- (-.6,.6);
\draw[thick] (.6,-.6) -- (.6,.6);
\draw[thick, \XsColor] (0,-.6) -- (0,-.3);
\draw[thick, \YsColor] (0,.6) -- (0,.3);
\roundNbox{boxregion=white, draw=black}{(0,0)}{.3}{0}{0}{$f$}
}
=
(I^{\vee}_{a} \otimes F(f) \otimes I_{b}) \oplus (J^{\vee}_{a} \otimes G(f) \otimes J_{b}).
$$

\item
For each $a\in \fX$, we define $(F\boxplus G)^0_a \colon  1_a\Rightarrow 1_{(F\boxplus G)(a)}$
by
$$
(F\boxplus G)^0_a
\coloneq
\tikzmath{
\begin{scope}
\clip[rounded corners=5pt] (-.6,-.3) rectangle (.6,.6);
\fill[plusregion=lightgray!50] (-.6,.6) -- (-.3,.6) -- (-.3,.3) arc(-180:0:.3cm) -- (.3,.6) -- (.6,.6) -- (.6,-.3) -- (-.6,-.3);    
\fill[primedregion=lightgray!50] (-.3,.6) -- (-.3,.3) arc(-180:0:.3cm) -- (.3,.6);       
\end{scope}
\draw[thick] (-.3,.6) -- (-.3,.3) arc(-180:0:.3cm) -- (.3,.6);
}
\oplus
\tikzmath{
\begin{scope}
\clip[rounded corners=5pt] (-.6,-.3) rectangle (.6,.6);
\fill[plusregion=lightgray!50] (-.6,.6) -- (-.3,.6) -- (-.3,.3) arc(-180:0:.3cm) -- (.3,.6) -- (.6,.6) -- (.6,-.3) -- (-.6,-.3);    
\fill[boxregion=lightgray!50] (-.3,.6) -- (-.3,.3) arc(-180:0:.3cm) -- (.3,.6);       
\end{scope}
\draw[thick] (-.3,.6) -- (-.3,.3) arc(-180:0:.3cm) -- (.3,.6);
}
=
\ev_{I_{a}}^\dag
\oplus
\ev_{J_{a}}^\dag
$$
where we suppress the unitaries 
$F^0_a\colon  1_a\Rightarrow 1_{F(a)}$
and
$G^0_a\colon  1_a\Rightarrow 1_{G(a)}$.
For composable 1-morphisms ${}_aX_b$ and ${}_bX_c$ in $\fX$, we define
$(F\boxplus G)^2_{X,Y} \colon  (F\boxplus G)({}_aX_b)\otimes_{(F\boxplus G)(b)} (F\boxplus G)({}_bY_c)\Rightarrow (F\boxplus G)({}_aX\otimes_b Y_c)$
by
$$
(F\boxplus G)^2_{X,Y}
\coloneq
\tikzmath{
\begin{scope}
\clip[rounded corners=5pt] (-.6,0) rectangle (1.8,.6);
\fill[plusregion=lightgray!50] (-.3,0) rectangle (-.6,.6);    
\fill[primedregion=lightgray!50] (-.3,0) -- (0,0) to[out=90,in=-90] (.3,.6) -- (-.3,.6);    
\fill[primedregion=lightgray] (0,0) -- (.3,0) arc (180:0:.3cm) -- (1.2,0) to[out=90,in=-90] (.9,.6) -- (.3,.6) to[out=-90,in=90] (0,0);
\fill[plusregion=lightgray] (.3,0) arc (180:0:.3cm);
\fill[primedregion=gray] (1.2,0) to[out=90,in=-90] (.9,.6) -- (1.5,.6) -- (1.5,0);    
\fill[plusregion=gray] (1.5,0) rectangle (1.8,.6);    
\end{scope}
\draw[thick] (1.5,0) -- (1.5,.6);
\draw[thick, \YsColor] (1.2,0) to[out=90,in=-90] (.9,.6);
\draw[thick] (.3,0) arc (180:0:.3cm);
\draw[thick, \XsColor] (0,0) to[out=90,in=-90] (.3,.6);
\draw[thick] (-.3,0) -- (-.3,.6);
}
\oplus
\tikzmath{
\begin{scope}
\clip[rounded corners=5pt] (-.6,0) rectangle (1.8,.6);
\fill[plusregion=lightgray!50] (-.3,0) rectangle (-.6,.6);    
\fill[boxregion=lightgray!50] (-.3,0) -- (0,0) to[out=90,in=-90] (.3,.6) -- (-.3,.6);    
\fill[boxregion=lightgray] (0,0) -- (.3,0) arc (180:0:.3cm) -- (1.2,0) to[out=90,in=-90] (.9,.6) -- (.3,.6) to[out=-90,in=90] (0,0);
\fill[plusregion=lightgray] (.3,0) arc (180:0:.3cm);
\fill[boxregion=gray] (1.2,0) to[out=90,in=-90] (.9,.6) -- (1.5,.6) -- (1.5,0);    
\fill[plusregion=gray] (1.5,0) rectangle (1.8,.6);    
\end{scope}
\draw[thick] (1.5,0) -- (1.5,.6);
\draw[thick, \YsColor] (1.2,0) to[out=90,in=-90] (.9,.6);
\draw[thick] (.3,0) arc (180:0:.3cm);
\draw[thick, \XsColor] (0,0) to[out=90,in=-90] (.3,.6);
\draw[thick] (-.3,0) -- (-.3,.6);
}
$$
where we suppress the unitaries
$F^2_{X,Y}\colon  {}_{F(a)}F(X)\otimes_{F(b)} F(Y)_{F(c)}\Rightarrow {}_{F(a)}F(X\otimes_{b} Y)_{F(c)}$
and $G^2_{X,Y}$.
Note that the caps above are $\coev_{I_b}^\dag$ and $\coev_{J_b}^\dag$ respectively.
\end{itemize}
One then checks unitarity of $(F\boxplus G)^0_a$ and $(F\boxplus G)^2_{X,Y}$ using the isometry relations from Definition \ref{defn:HilbertDirectSum} and the necessary coherences follow via diagrammatic calculus.

We now define an isometry $I\colon  F\Rightarrow F\boxplus G$ by
$$
I=
\tikzmath{
\begin{scope}
\clip[rounded corners=5pt] (0,0) rectangle (.6,.6);
\fill[primedregion=white] (0,0) rectangle (.3,.6);    
\fill[plusregion=white] (.3,0) rectangle (.6,.6);    
\end{scope}
\draw[dotted, rounded corners=5pt] (0,0) rectangle (.6,.6);
\draw[thick] (.3,0) -- (.3,.6);
}
=
\left(
\left\{
I_a
=
\tikzmath{
\begin{scope}
\clip[rounded corners=5pt] (0,0) rectangle (.6,.6);
\fill[primedregion=lightgray!50] (0,0) rectangle (.3,.6);    
\fill[plusregion=lightgray!50] (.3,0) rectangle (.6,.6);    
\end{scope}
\draw[thick] (.3,0) -- (.3,.6);
}
\right\}_{a\in \fX}
,
\left\{
I_X
=
\tikzmath{
\begin{scope}
\clip[rounded corners=5pt] (-1.2,-.6) rectangle (.6,.6);
\fill[plusregion=lightgray!50] (-.3,.6) -- (-.3,.3) arc (0:-180:.3cm) -- (-.9,.6);    
\fill[primedregion=lightgray!50] (-.6,-.6) to[out=90,in=-90] (0,0) -- (0,.6)  -- (-.3,.6) -- (-.3,.3) arc (0:-180:.3cm) -- (-.9,.6) -- (-1.2,.6) -- (-1.2,-.6);    
\fill[primedregion=lightgray] (-.6,-.6) to[out=90,in=-90] (0,0) -- (0,.6) -- (.3,.6) -- (.3,-.6);    
\fill[plusregion=lightgray] (.3,-.6) rectangle (.6,.6);    
\end{scope}
\draw[thick] (.3,-.6) node[below]{$\scriptstyle I_b$} -- (.3,.6);
\draw[thick, \XsColor] (-.6,-.6) node[below]{$\scriptstyle F(X)$} to[out=90,in=-90] (0,0) -- (0,.6);
\draw[thick] (-.3,.6) -- (-.3,.3) arc (0:-180:.3cm) -- (-.9,.6) node[above]{$\scriptstyle I_a$};
}
\right\}_{X\colon a\to b}
\right)
$$
and similar isometry $J\colon G\Rightarrow F\boxplus G$.
One checks the isometry relations componentwise.
\\

\item[\underline{$\rmH^*$-monads split:}]
Suppose $(\alpha\colon  F \Rightarrow F, m\colon  \alpha^2 \Rrightarrow \alpha, i\colon  \id_F \Rrightarrow \alpha )$
is an $\rmH^*$-monad in $\Hom(\fX\to \fY)$.
In diagrams, this means:
\begin{itemize}
\item
\ref{H:Frobenius}
$m^\dag$ is an $\alpha$–$\alpha$-bimodule map, i.e.,
\[
\tikzmath{
    \filldraw[primedregion=white, rounded corners=5pt, draw=black, dotted] (-.3,-.6) rectangle (1.5,.6);
    \draw[thick] (0,-.6) -- (0,0) arc (180:0:.3cm) arc (-180:0:.3cm) -- (1.2,.6);
    \draw[thick] (.3,.3) -- (.3,.6);
    \draw[thick] (.9,-.3) -- (.9,-.6);
}
=
\tikzmath{
    \filldraw[primedregion=white, rounded corners=5pt, draw=black, dotted] (-.3,0) rectangle (.9,1.2);
    \draw[thick] (0,0) arc (180:0:.3cm);
    \draw[thick] (0,1.2) arc (-180:0:.3cm);
    \draw[thick] (.3,.3) -- (.3,.9);
}
=
\tikzmath{
    \filldraw[primedregion=white, rounded corners=5pt, draw=black, dotted] (-.3,.6) rectangle (1.5,-.6);
    \draw[thick] (0,.6) -- (0,0) arc (-180:0:.3cm) arc (180:0:.3cm) -- (1.2,-.6);
    \draw[thick] (.3,-.3) -- (.3,-.6);
    \draw[thick] (.9,.3) -- (.9,.6);
}
\quad\text{where}\;
F=
\tikzmath{
\filldraw[primedregion=white, rounded corners=5pt, draw=black, dotted] (0,0) rectangle (.6,.6);    
},
\;
{}_F\alpha_F
=
\tikzmath{
\begin{scope}
\clip[rounded corners=5pt] (0,0) rectangle (.6,.6);
\fill[primedregion=white] (0,0) rectangle (.3,.6);    
\fill[primedregion=white] (.3,0) rectangle (.6,.6);    
\end{scope}
\draw[dotted, rounded corners=5pt] (0,0) rectangle (.6,.6);
\draw[thick] (.3,0) -- (.3,.6);
},
\;
m
=
\tikzmath{
    \tikzset{scale=1.25,yscale=-1}
    \filldraw[primedregion=white,rounded corners, draw=black, dotted](-.05,0)rectangle(.65,.5);
    \draw[thick](0.3,0)--(0.3,.25);
    \draw[thick](0.05,.5) arc(-180:0:.25);
},
\;
m^\dag=
 \tikzmath{
    \tikzset{scale=1.25}
    \filldraw[primedregion=white,rounded corners,draw=black, dotted](-.05,0)rectangle(.65,.5);
    \draw[thick](0.3,0)--(0.3,.25);
    \draw[thick](0.05,.5) arc(-180:0:.25);
}.
\]
\item
\ref{H:Separable} 
The endomorphism $mm^\dag$ of ${}_F\alpha_F$ is invertible, i.e.,
\[
\tikzmath{
    \tikzset{scale=0.3}
    \coordinate (O) at (0,0);
    \filldraw[primedregion=white,rounded corners, draw=black, dotted] (O) rectangle ++(3,4);
    \draw[thick] (O) ++(1.5,0) -- ++(0,1);
    \draw[thick] (O) ++(1.5,3) -- ++(0,1);
    \draw[thick](O)++(1.5,2) circle(1);
}
\in\End_{\Hom(\fX\to \fY)}({}_F\alpha_F)^\times.
\]
\item
\ref{H:Standard}  
Denoting $
\tikzmath{
    \tikzset{scale=1.25}
    \fill[primedregion=white, rounded corners=5pt, draw=black, dotted](0,0)rectangle(.5,.5);
    \draw[thick](0.25,.5)--(.25,.25);
    \filldraw(.25,.25)circle(0.05cm);
}
=i
$
and
$
\tikzmath{
    \tikzset{scale=1.25}
    \tikzset{yscale=-1}
    \fill[primedregion=white, rounded corners=5pt, draw=black, dotted](0,0)rectangle(.5,.5);
    \draw[thick](0.25,.5)--(.25,.25);
    \filldraw(.25,.25)circle(0.05cm);
}
=i^\dag
$,
the morphisms
$$
\tikzmath{
\fill[primedregion=white, rounded corners=5pt, draw=black, dotted] (1.8,-.3) rectangle (-.5,1.2);
\draw[thick] (0,-.3) --node[left]{$\scriptstyle A$} (0,.3) arc (180:0:.3cm) arc (-180:0:.3cm) --node[right]{$\scriptstyle A^\vee$} (1.2,1.2);
\draw[thick] (.3,.9) -- (.3,.6);
\filldraw (.3,.9) circle (.05cm);
}
\qquad\qquad\text{and}\qquad\qquad
\tikzmath{
\fill[primedregion=white, rounded corners=5pt, draw=black, dotted] (-1.8,-.3) rectangle (.5,1.2);
\draw[thick] (0,-.3) --node[right]{$\scriptstyle A$} (0,.3) arc (0:180:.3cm) arc (0:-180:.3cm) --node[left]{$\scriptstyle A^\vee$} (-1.2,1.2);
\draw[thick] (-.3,.9) -- (-.3,.6);
\filldraw (-.3,.9) circle (.05cm);
}
$$
are unitary (and thus equal).

\end{itemize}
Evaluating at $a\in \fX$, we overlay the above diagrams on an $a$-shaded region.
Observe that the overlaid diagrams equip each ${}_{F(a)}(\alpha_a)_{F(a)}\in \Omega_{F(a)}$ with the structure of an $\rmH^*$-monad.

We now construct a splitting of ${}_F\alpha_F$, which consists of a functor $G: \fX\to \fY$, a transformation $\beta: F\Rightarrow G$, and a unitary modification $\gamma: {}_F\alpha_F \Rrightarrow {}_F \beta \otimes_G \beta^\vee_F$ that is compatible with the canonical monad structure on ${}_F\beta \otimes_G \beta^\vee_F$.
\begin{itemize}
\item 
For each $a\in \fX$, 
we choose a splitting $({}_{F(a)}Y_{c}, \gamma_a)$ for the $\rmH^*$-monad $(\alpha_a, m_a, i_a)$, i.e., 
$\ev_{Y}\ev_{Y}^\dag \in \Omega_{c}^2$ is invertible, and $\gamma_a\colon  {}_{F(a)}(\alpha_a)_{F(a)} \Rightarrow {}_{F(a)}Y\otimes_c Y^\vee_{F(a)}$ is a unitary monad isomorphism.
We define $G(a)\coloneq c$.
$$
a=
\tikzmath{
\filldraw[lightgray!50, rounded corners=5pt] (0,0) rectangle (.6,.6);    
}
\qquad\qquad
F(a)=
\tikzmath{
\filldraw[primedregion=lightgray!50, rounded corners=5pt] (0,0) rectangle (.6,.6);    
}
\qquad\qquad
G(a)=
\tikzmath{
\filldraw[boxregion=lightgray!50, rounded corners=5pt] (0,0) rectangle (.6,.6);    
}
\qquad\qquad
{}_{F(a)}Y_{G(a)}=
\tikzmath{
\begin{scope}
\clip[rounded corners=5pt] (0,0) rectangle (.6,.6);
\fill[primedregion=lightgray!50] (0,0) rectangle (.3,.6);   
\fill[boxregion=lightgray!50] (.3,0) rectangle (.6,.6);   
\end{scope}
\draw[thick,blue] (.3,0) -- (.3,.6);
}
\qquad\qquad
\gamma_a=
\tikzmath{
\fill[rounded corners=5pt, primedregion=lightgray!50] (0,0) rectangle (1.2,.6);
\draw[preaction={boxregion=lightgray!50},thick,blue] (.3,.6) arc(-180:0:.3cm);
\draw[thick] (.6,.3) -- (.6,0);
}
$$
\item 
For a 1-morphism ${}_{a_1}X_{a_2}$, we define $G(X)$ as the splitting of the following idempotent:
\[
G(X)=
\tikzmath{
\begin{scope}
    \clip[rounded corners=5pt] (-1,-.8) rectangle (1,.8);
    \fill[boxregion=lightgray!50] (-1,-.8) rectangle (-.7,.8);      \fill[primedregion=lightgray!50] (-.7,-.8) rectangle (0,.8); 
    \fill[primedregion=lightgray] (0,-.8) rectangle (.7,.8);
    \fill[boxregion=lightgray] (.7,-.8) rectangle (1,.8);
\end{scope}
\draw[thick,blue] (-.7,-.8) node[below]{$\scriptstyle Y_1^\vee$} --(-.7,.8) node[above]{$\scriptstyle Y_1^\vee$};
\draw[thick,blue] (.7,-.8) node[below]{$\scriptstyle Y_2$} --(.7,.8) node[above]{$\scriptstyle Y_2$};
\draw[thick,\XsColor] (0,-.8) node[below]{$\scriptstyle F(X)$} --(0,.8) node[above]{$\scriptstyle F(X)$};
\draw[thick](-.7,0)-- (.7,0);
}
\coloneq
\tikzmath{
\begin{scope}
    \clip[rounded corners=5pt] (-1,-.8) rectangle (1.3,.8);
    \fill[boxregion=lightgray!50] (-1,-.8) rectangle (-.7,.8); 
    \fill[primedregion=lightgray!50] (-.7,-.8) rectangle (0,.8);
    \fill[primedregion=lightgray] (0,-.8) rectangle (1,.8);
    \fill[boxregion=lightgray] (1,-.8) rectangle (1.3,.8);
\end{scope}
\draw[thick,blue] (-.7,-.8) node[below]{$\scriptstyle Y_1^\vee$} --(-.7,.8) node[above]{$\scriptstyle Y_1^\vee$};
\draw[thick,blue] (1,-.8) node[below]{$\scriptstyle Y_2$} --(1,.8) node[above]{$\scriptstyle Y_2$};
\draw[thick,\XsColor] (0,-.8) node[below]{$\scriptstyle F(X)$} --(0,.8) node[above]{$\scriptstyle F(X)$};
\draw[thick] (.3,-.1) arc (180:-180:.1cm) node[above, xshift=.4cm,yshift = -.1cm] {$\scriptstyle -1$};
\draw[thick](.4,-.2) to[out = -90,in=150] (1,-.6);
\draw[thick](.4,0) to[out = 90,in=-60] (-.7,.6);
}
\qquad\qquad\qquad
\begin{aligned}
a_1&=
\tikzmath{
\filldraw[lightgray!50, rounded corners=5pt] (0,0) rectangle (.6,.6);    
}
&
a_2&=
\tikzmath{
\filldraw[lightgray, rounded corners=5pt] (0,0) rectangle (.6,.6);    
}
\\
{}_{a_1}X_{a_2}&=
\tikzmath{
\begin{scope}
\clip[rounded corners=5pt] (0,0) rectangle (.6,.6);
\fill[lightgray!50] (0,0) rectangle (.3,.6);   
\fill[lightgray] (.3,0) rectangle (.6,.6);   
\end{scope}
\draw[thick,\XsColor] (.3,0) -- (.3,.6);
}
\\
\alpha_{a_1}&=
\tikzmath{
\begin{scope}
\clip[rounded corners=5pt] (0,0) rectangle (.6,.6);
\fill[primedregion=lightgray!50] (0,0) rectangle (.6,.6);   
\end{scope}
\draw[thick] (.3,0) -- (.3,.6);
}
&
\alpha_{a_2}&=
\tikzmath{
\begin{scope}
\clip[rounded corners=5pt] (0,0) rectangle (.6,.6);
\fill[primedregion=lightgray] (0,0) rectangle (.6,.6);   
\end{scope}
\draw[thick] (.3,0) -- (.3,.6);
}
\end{aligned}
\]
where $({}_{F(a_i)}(Y_i)_{G(a_i)},\gamma_i)$ is our chosen splitting for $\alpha_{a_i}$ for $i=1,2$.
(Recall that the right $\alpha_{a_1}$-action on $Y_1^\vee$ and the left $\alpha_{a_2}$-action on $Y_2$ are as in Remark \ref{rem:ModuleFromSplitting}.)

\item 
For a 2-morphism $f \in \fX( {}_{a_1}X_{a_2} \Rightarrow {}_{a_1}Z_{a_2})$, we define
\[
G(f)=
\tikzmath{
\begin{scope}
    \clip[rounded corners=5pt] (-1,-.8) rectangle (1,.8);
    \fill[boxregion=lightgray!50] (-1,-.8) rectangle (-.7,.8);      
    \fill[primedregion=lightgray!50] (-.7,-.8) rectangle (0,.8); 
    \fill[primedregion=lightgray] (0,-.8) rectangle (.7,.8);
    \fill[boxregion=lightgray] (.7,-.8) rectangle (1,.8);
\end{scope}
\draw[thick,blue] (-.7,-.8) node[below]{$\scriptstyle Y_1^\vee$} --(-.7,.8) node[above]{$\scriptstyle Y_1^\vee$};
\draw[thick,blue] (.7,-.8) node[below]{$\scriptstyle Y_2$} --(.7,.8) node[above]{$\scriptstyle Y_2$};
\draw[thick,\XsColor] (0,-.8) node[below]{$\scriptstyle F(X)$} --(0,.8) node[above]{$\scriptstyle F(Z)$};
\draw[thick](-.7,-.5)-- (.7,-.5);
\roundNbox{primedregion=white, draw=black}{(0,.1)}{.3}{.2}{.2}{$F(f)$};
}
=
\tikzmath{
\begin{scope}
    \clip[rounded corners=5pt] (-1,-.8) rectangle (1,.8);
    \fill[boxregion=lightgray!50] (-1,-.8) rectangle (-.7,.8);      \fill[primedregion=lightgray!50] (-.7,-.8) rectangle (0,.8); 
    \fill[primedregion=lightgray] (0,-.8) rectangle (.7,.8);
    \fill[boxregion=lightgray] (.7,-.8) rectangle (1,.8);
\end{scope}
\draw[thick,blue] (-.7,-.8) node[below]{$\scriptstyle Y_1^\vee$} --(-.7,.8) node[above]{$\scriptstyle Y_1^\vee$};
\draw[thick,blue] (.7,-.8) node[below]{$\scriptstyle Y_2$} --(.7,.8) node[above]{$\scriptstyle Y_2$};
\draw[thick,\XsColor] (0,-.8) node[below]{$\scriptstyle F(X)$} --(0,.8) node[above]{$\scriptstyle F(Z)$};
\draw[thick](-.7,.5)-- (.7,.5);
\roundNbox{primedregion=white, draw=black}{(0,-.1)}{.3}{.2}{.2}{$F(f)$};
}
\]

\item 
For each $a \in \fX$, we define $G^0_a \colon 1_{G(a)} \Rightarrow G(1_a)$ by
$$
G^0_a
\coloneq
\tikzmath{
\begin{scope}
\clip[rounded corners=5pt] (-1,-.9) rectangle (1,.6);
\fill[boxregion=lightgray!50] (-1,.6) -- (-.7,.6) -- (-.7,0) arc(-180:0:.7cm) -- (.7,.6) -- (1,.6) -- (1,-.8) -- (-1,-.9);    
\fill[primedregion=lightgray!50] (-.7,.6) -- (.7,.6) --(.7,0) arc(0:-180:.7cm) -- (-.7,0);       
\end{scope}
\draw[thick,blue] (-.7,.6) node[above]{$\scriptstyle Y^\vee$} -- (-.7,0) arc(-180:0:.7cm) -- (.7,.6) node[above]{$\scriptstyle Y$};
\draw[thick, \XsColor] (0,.6) node[above]{$\scriptstyle F(1_a)$} -- (0,.1);
\filldraw[\XsColor] (0,.1) circle (.05cm);
\draw[thick](-.7,.3) -- (.7,.3);
\draw[thick] (-.7,.1) to[out=-45,in=90] (0,-.2);
\draw[thick] (0,-.3) circle (.1cm);
\node at (.3,-.1) {$\scriptstyle -1/2$};
}$$
For composable 1-morphisms ${}_{a_1} X_{a_2}$ and $ {}_{a_2} W_{a_3}$ in $\fX,$ we define $G^2_{X,W} \colon G(X) \otimes_{G(a_2)} G(W) \Rightarrow G(X \otimes_{a_2} W)$ 
by
$$
G^2_{X,Y} \coloneq
\tikzmath{
\begin{scope}
\clip[rounded corners=5pt] (-2,-.5) rectangle (2,1.2);
\fill[boxregion=lightgray!50] (-2,-.5) -- (-1.7,-.5) -- (-1.7,-.2) to[out = 90, in = -90] (-.9,1) -- (-.9,1.2) -- (-2,1.2) -- (-2,-.5) ;    
\fill[primedregion=lightgray!50] (-1.7,-.5) -- (-1.7,-.2) to[out = 90, in = -90] (-.9,1) -- (-.9,1.2) -- (-.05,1.2) -- (-.05,.8) to[out = -90, in=90] (-1,-.3) -- (-1,-.5) -- (-1.7,-.5);   
\fill[primedregion=lightgray] (.05,1.2) -- (.05,.8) to[out = -90, in=90] (1,-.3) -- (1,-.5) -- (.3,-.5)  -- (.3,-.3) to[out =90,in =0] (0,0) to[out = 180, in =90] (-.3,-.3) -- (-.3,-.5) -- (-1,-.5) -- (-1,-.3) to[out=90,in=-90] (-.05,.8) -- (-.05,1.2) -- (.05,1.2);
\fill[boxregion=lightgray] (.3,-.5)  -- (.3,-.3) to[out =90,in =0] (0,0) to[out = 180, in =90] (-.3,-.3) -- (-.3,-.5) -- (.3,-.5);
\fill[primedregion=lightgray!150](.05,1.2) -- (.05,.8) to[out = -90, in=90] (1,-.3) -- (1,-.5) -- (1.7,-.5) -- (1.7,-.2) to[out = 90, in = -90] (.9,1) -- (.9,1.2) -- (.05,1.2);
\fill[boxregion=lightgray!150] (1.7,-.5) -- (1.7,-.2) to[out = 90, in = -90] (.9,1) -- (.9,1.2) -- (2,1.2) -- (2,-.5) -- (1.7,-.5);
\end{scope}
\draw[thick,blue] (-1.7,-.5) node[below]{$\scriptstyle Y_1^\vee$} -- (-1.7,-.2) to[out = 90, in = -90] (-.9,1) -- (-.9,1.2) node[above]{$\scriptstyle Y_1^\vee$};
\draw[thick,blue] (1.7,-.5) node[below]{$\scriptstyle Y_3$} -- (1.7,-.2) to[out = 90, in = -90] (.9,1) -- (.9,1.2) node[above]{$\scriptstyle Y_3$};
\draw[thick, \XsColor] (1,-.5) node[below]{$\scriptstyle F(W)$} -- (1,-.3) to[out=90,in=-90] (.05,.8) -- (.05,1.2);
\draw[thick, \XsColor] (-1,-.5) node[below]{$\scriptstyle F(X)$} -- (-1,-.3) to[out=90,in=-90] (-.05,.8) -- (-.05,1.2) node[above,xshift = .05cm]{$\scriptstyle F(X \otimes W)$};
\draw[thick,blue] (.3,-.5) node[below]{$\scriptstyle Y_2^\vee$} -- (.3,-.3) to[out =90,in =0] (0,0) to[out = 180, in =90] (-.3,-.3) -- (-.3,-.5) node[below]{$\scriptstyle Y_2$};
\draw[thick] (-1.7,-.3) -- (-.3,-.3);
\draw[thick] (.3,-.3) -- (1.7,-.3);
\draw[thick] (-.9,1) -- (.9,1);
\draw[thick] (0,0) -- (0,.125);
\draw[thick] (0,.2) circle (.075cm);
\node at (.3,.1) {$\scriptstyle 1/2$};
}
=
\tikzmath{
\begin{scope}
\clip[rounded corners=5pt] (-2,-.5) rectangle (2,1.2);
\fill[boxregion=lightgray!50] (-2,-.5) -- (-1.7,-.5) -- (-1.7,-.2) to[out = 90, in = -90] (-.9,1) -- (-.9,1.2) -- (-2,1.2) -- (-2,-.5) ;    
\fill[primedregion=lightgray!50] (-1.7,-.5) -- (-1.7,-.2) to[out = 90, in = -90] (-.9,1) -- (-.9,1.2) -- (-.05,1.2) -- (-.05,.8) to[out = -90, in=90] (-1,-.3) -- (-1,-.5) -- (-1.7,-.5);   
\fill[primedregion=lightgray] (.05,1.2) -- (.05,.8) to[out = -90, in=90] (1,-.3) -- (1,-.5) -- (.3,-.5)  -- (.3,-.3) to[out =90,in =0] (0,0) to[out = 180, in =90] (-.3,-.3) -- (-.3,-.5) -- (-1,-.5) -- (-1,-.3) to[out=90,in=-90] (-.05,.8) -- (-.05,1.2) -- (.05,1.2);
\fill[boxregion=lightgray] (.3,-.5)  -- (.3,-.3) to[out =90,in =0] (0,0) to[out = 180, in =90] (-.3,-.3) -- (-.3,-.5) -- (.3,-.5);
\fill[primedregion=lightgray!150](.05,1.2) -- (.05,.8) to[out = -90, in=90] (1,-.3) -- (1,-.5) -- (1.7,-.5) -- (1.7,-.2) to[out = 90, in = -90] (.9,1) -- (.9,1.2) -- (.05,1.2);
\fill[boxregion=lightgray!150] (1.7,-.5) -- (1.7,-.2) to[out = 90, in = -90] (.9,1) -- (.9,1.2) -- (2,1.2) -- (2,-.5) -- (1.7,-.5);
\end{scope}
\draw[thick,blue] (-1.7,-.5) node[below]{$\scriptstyle Y_1^\vee$} -- (-1.7,-.2) to[out = 90, in = -90] (-.9,1) -- (-.9,1.2) node[above]{$\scriptstyle Y_1^\vee$};
\draw[thick,blue] (1.7,-.5) node[below]{$\scriptstyle Y_3$} -- (1.7,-.2) to[out = 90, in = -90] (.9,1) -- (.9,1.2) node[above]{$\scriptstyle Y_3$};
\draw[thick, \XsColor] (1,-.5) node[below]{$\scriptstyle F(W)$} -- (1,-.3) to[out=90,in=-90] (.05,.8) -- (.05,1.2);
\draw[thick, \XsColor] (-1,-.5) node[below]{$\scriptstyle F(X)$} -- (-1,-.3) to[out=90,in=-90] (-.05,.8) -- (-.05,1.2) node[above,xshift = .05cm]{$\scriptstyle F(X \otimes W)$};
\draw[thick,blue] (.3,-.5) node[below]{$\scriptstyle Y_2^\vee$} -- (.3,-.3) to[out =90,in =0] (0,0) to[out = 180, in =90] (-.3,-.3) -- (-.3,-.5) node[below]{$\scriptstyle Y_2$};
\draw[thick] (-1.7,-.3) -- (-.3,-.3);
\draw[thick] (.3,-.3) -- (1.7,-.3);
\draw[thick] (0,0) -- (0,.125);
\draw[thick] (0,.2) circle (.075cm);
\node at (.3,.1) {$\scriptstyle 1/2$};
}
$$

\end{itemize}

We now define a transformation $\beta: F\Rightarrow G$ and a unitary $\gamma: {}_F\alpha_F\Rrightarrow {}_F\beta\otimes_G \beta^\vee_F$ as follows:
\begin{itemize}
\item 
For each $a \in \fX,$ define $\beta_a \coloneq Y \colon F(a) \to G(a)$, and for a 1-morphism ${}_{a_1}X_{a_2}$, set
$$\beta_X \coloneq
\tikzmath{
\begin{scope}
\clip[rounded corners=5pt] (-.3,0) rectangle (2.6,1.3);
\fill[primedregion=lightgray!50] (-.3,0) -- (.3,0) -- (.3,.2) to[out = 90,in=-90] (1.6,1.1) -- (1.6,1.3) -- (.9,1.3) -- (.9,1.1) to[out =-90,in=0] (.6,.8) to[out = 180, in=-90] (.3,1.1) -- (.3,1.3) -- (-.3,1.3) -- (-.3,0);    
\fill[boxregion=lightgray!50] (.9,1.3) -- (.9,1.1) to[out =-90,in=0](.6,.8) to[out = 180, in=-90] (.3,1.1) -- (.3,1.3) -- (.9,1.3);   
\fill[primedregion=lightgray] (.3,0) -- (.3,.2) to[out = 90,in=-90] (1.6,1.1) -- (1.6,1.3) -- (2.3,1.3) -- (2.3,0) -- (.3,0) ;
\fill[boxregion=lightgray] (2.3,0) -- (2.6,0) -- (2.6,1.3) -- (2.3,1.3);
\end{scope}
\draw[thick, \XsColor] (.3,0) node[below]{$\scriptstyle F(X)$} -- (.3,.2) to[out = 90,in=-90] (1.6,1.1) -- (1.6,1.3) node[above]{$\scriptstyle F(X)$};
\draw[thick,blue] (2.3,0) node[below]{$\scriptstyle Y_2$} -- (2.3,1.3) node[above]{$\scriptstyle Y_2$};
\draw[thick,blue] (.9,1.3) node[above]{$\scriptstyle Y_1^\vee$}-- (.9,1.1) to[out =-90,in=0] (.6,.8) to[out = 180, in=-90] (.3,1.1) -- (.3,1.3) node[above]{$\scriptstyle Y_1$};
\draw[thick] (.9,1.1) -- (2.3,1.1);
\draw[thick] (.3,1.1) to[out=-135,in=90] (0,.8);
\draw[thick] (0,.7) circle (.1cm);
\node at (.35,.7) {$\scriptstyle 1/2$};
}\,
\qquad\text{which implies}\qquad
\beta^\vee_X =
\tikzmath{
\begin{scope}
\clip[rounded corners=5pt] (.3,0) rectangle (-2.6,-1.3);
\fill[primedregion=lightgray] (.3,0) -- (-.3,0) -- (-.3,-.2) to[out = -90,in=90] (-1.6,-1.1) -- (-1.6,-1.3) -- (-.9,-1.3) -- (-.9,-1.1) to[out =90,in=180] (-.6,-.8) to[out = 0, in=90] (-.3,-1.1) -- (-.3,-1.3) -- (.3,-1.3) -- (.3,0);    
\fill[boxregion=lightgray] (-.9,-1.3) -- (-.9,-1.1) to[out =90,in=180](-.6,-.8) to[out = 0, in=90] (-.3,-1.1) -- (-.3,-1.3) -- (-.9,-1.3);   
\fill[primedregion=lightgray!50] (-.3,0) -- (-.3,-.2) to[out = -90,in=90] (-1.6,-1.1) -- (-1.6,-1.3) -- (-2.3,-1.3) -- (-2.3,0) -- (-.3,0) ;
\fill[boxregion=lightgray!50] (-2.3,0) -- (-2.6,0) -- (-2.6,-1.3) -- (-2.3,-1.3);
\end{scope}
\draw[thick, \XsColor] (-.3,0) node[above]{$\scriptstyle F(X)$} -- (-.3,-.2) to[out = -90,in=90] (-1.6,-1.1) -- (-1.6,-1.3) node[below]{$\scriptstyle F(X)$};
\draw[thick,blue] (-2.3,0) node[above]{$\scriptstyle Y_1^\vee$} -- (-2.3,-1.3) node[below]{$\scriptstyle Y_1^\vee$};
\draw[thick,blue] (-.9,-1.3) node[below]{$\scriptstyle Y_2$} -- (-.9,-1.1) to[out =90,in=180] (-.6,-.8) to[out = 0, in=90] (-.3,-1.1) -- (-.3,-1.3) node[below]{$\scriptstyle Y_2^\vee$};
\draw[thick] (-.9,-1.1) -- (-2.3,-1.1);
\draw[thick] (-.3,-1.1) to[out=135,in=-90] (0,-.8);
\draw[thick] (0,-.7) circle (.1cm);
\node at (-.35,-.7) {$\scriptstyle 1/2$};
}
$$

\item 
Define $\gamma \colon {}_F\alpha_F \Rrightarrow {}_F\beta\otimes_G \beta^\vee_F$ by letting its component at $a$ be $\gamma_a \colon  \alpha_a \Rightarrow Y\otimes_{G(a)} Y^\vee$.
That $\gamma$ is a modification follows from the equality
\[
\tikzmath{
\begin{scope}
\clip[rounded corners=5pt] (-.3,0) rectangle (3.5,2);
\fill[primedregion=lightgray!50] (-.3,0) -- (.3,0) -- (.3,.2) to[out = 90,in=-90] (1.6,1.1) -- (1.6,2) -- (.9,2) -- (.9,1.1) arc(0:-180:.3cm) -- (.3,2) -- (-.3,2) -- (-.3,0);    
\fill[boxregion=lightgray!50] (.9,2) -- (.9,1.1) arc(0:-180:.3cm) -- (.3,2);   
\fill[primedregion=lightgray] (.3,0) -- (.3,.2) to[out = 90,in=-90] (1.6,1.1) -- (1.6,2) -- (3.5,2) -- (3.5,0) -- (.3,0) ;
\end{scope}
\draw[thick, \XsColor] (.3,0) node[below]{$\scriptstyle F(X)$} -- (.3,.2) to[out = 90,in=-90] (1.6,1.1) -- (1.6,2) node[above]{$\scriptstyle F(X)$};
\draw[preaction={boxregion=lightgray}, thick,blue] (2.3,.6) -- (2.3,1.1) arc(180:0:.3cm) -- (2.9,.6) arc(0:-180:.3cm);
\draw[thick,blue] (.9,2) node[above]{$\scriptstyle Y_1^\vee$} -- (.9,1.1) arc(0:-180:.3cm) (.3,1.1) -- (.3,2) node[above]{$\scriptstyle Y_1$};
\draw[thick] (2.6,.3) -- (2.6,0) node[below]{$\scriptstyle \alpha_{a_2}$};
\draw[thick] (.9,1.1) -- (2.3,1.1);
\draw[thick] (.3,1.1) to[out=-135,in=90] (0,.8);
\draw[thick] (0,.7) circle (.1cm);
\node at (.35,.7) {$\scriptstyle 1/2$};
\draw[thick] (2.9,1.1) to[out=45,in=-90] (3.2,1.4);
\draw[thick] (3.2,1.5) circle (.1cm);
\node at (2.8,1.6) {$\scriptstyle 1/2$};
}
=
\tikzmath{
\begin{scope}
\clip[rounded corners=5pt] (0,0) rectangle (1.9,1.3);
\fill[primedregion=lightgray!50] (0,0) -- (.6,0) -- (.6,.2) to[out = 90,in=-90] (1.6,.8) -- (1.6,1.3) -- (.9,1.3) -- (.9,1.1) to[out =-90,in=0] (.6,.8) to[out = 180, in=-90] (.3,1.1) -- (.3,1.3) -- (0,1.3) -- (0,0);    
\fill[boxregion=lightgray!50] (.9,1.3) -- (.9,1.1) to[out =-90,in=0](.6,.8) to[out = 180, in=-90] (.3,1.1) -- (.3,1.3) -- (.9,1.3);   
\fill[primedregion=lightgray] (.6,0) -- (.6,.2) to[out = 90,in=-90] (1.6,.8) -- (1.6,1.3) -- (1.9,1.3) -- (1.9,0) -- (.3,0) ;
\end{scope}
\draw[thick, \XsColor] (.6,0) node[below]{$\scriptstyle F(X)$} -- (.6,.2) to[out = 90,in=-90] (1.6,.8) -- (1.6,1.3) node[above]{$\scriptstyle F(X)$};
\draw[thick,blue] (.9,1.3) node[above]{$\scriptstyle Y_1^\vee$} -- (.9,1.1) to[out =-90,in=0] (.6,.8) to[out = 180, in=-90] (.3,1.1) -- (.3,1.3) node[above]{$\scriptstyle Y_1$};
\draw[thick] (.6,.8) to[out=-90,in=90] (1.6,.2) -- (1.6,0) node[below]{$\scriptstyle \alpha_{a_2}$};
}\,.
\]
\end{itemize}
We define the modification $\ev_\beta$ as in Definition \ref{defn:trace and UAF on Hom}.
Notice $\ev_\beta \ev^\dag_\beta$ is invertible and $\gamma$ is a unitary monad isomorphism as these hold for each component by construction.
\end{proof}

\subsection{The self-enrichment satisfies the desiderata}
\label{subsec:selfenrichmentdesiderata}
We now show that $(\vee,\Psi^{\Hom})$ satisfy our desiderata for a 3-Hilbert space structure on $\Hom(\fX\to \fY)$.
\begin{enumerate}[label=(D\arabic*)]

\item 
The canonical equivalence $\Hom(2\Hilb\to \fX)\cong \fX$ 
given by evaluation at $\Hilb\in 2\Hilb$ is isometric.
\begin{proof}
Since $\{\Hilb\}$ is an ONB for $2\Hilb$, for any $F\colon 2\Hilb\to \fX$ and $\mu\colon \id_F\Rightarrow\id_F$,
\[
\Psi_F^\Fun(\mu)
=
\frac{d_{\Hilb}}{D_{\End_{2\Hilb}(\Hilb)}}
\Psi^\fX_{F(\Hilb)}(\mu_{\Hilb})
=
\Psi^\fX_{F(\Hilb)}(\mu_{\Hilb}).
\qedhere
\]
\end{proof}

\item 
\label{D:YonedaEmbeddingIsometric}
The Yoneda embedding $\yo\colon\fX \hookrightarrow \Hom(\fX^{\op}\to 2\Hilb)$ is an isometric equivalence of 3-Hilbert spaces.
\begin{proof}
We know the Yoneda embedding $\yo$ is fully faithful, and it is isometrically essentially surjective by Corollary \ref{cor:AllPresheafsRepresentable}.
Also, $\yo$ is $\dag$-preserving by \cite[Thm.~2.9]{2404.05193}, as it is just a feature of the underlying $\rmC^*$-2-categories,
and it is UAF-preserving by Lemma \ref{lem:uaf-reps-unitary-adjs}
It remains to prove $\yo$ is isometric;
by \cite[Rem.~4.50]{MR5078555}, it suffices to check this on an ONB $\pi_0\fX$.
Recall that by definition, if $b\in \pi_0\fX$, then
$$
(\yo \alpha_b)_X \coloneq
[
X\cong 1_b\otimes X \xrightarrow{\alpha\otimes \id_X} 1_b\otimes X \cong X
]
\qquad\qquad
\forall\, {}_bX_b.
$$
We then calculate
\begin{align*}
\Psi_{\yo b}^{\Hom}(\yo \alpha)
&=
\frac{d_b}{D_{\Omega_b}} \Psi^{2\Hilb}_{\Omega_b}(\yo\alpha_b)
\\&=
\frac{d_b}{D_{\Omega_b}} \sum_{X\in\pi_0\Omega_b} d_X \Psi^\fX_b\left(\tikzmath{
\fill[fill=lightgray!50,rounded corners=5] (-1,-.5) rectangle (1,.5);
\roundNbox{fill=white}{(-.6,0)}{.25}{0}{0}{\scriptsize$\alpha$};
\draw[thick,fill=lightgray] (.2,0) arc (180:360+180:.3);
\node at (0,0) {$\scriptstyle X$};
}\right)
\\&=
\frac{1}{D_{\Omega_b}} \sum_{X\in\pi_0\Omega_b} d_X \Psi^\fX_b\left(\tikzmath{
\fill[fill=lightgray!50,rounded corners=5] (-1,-.5) rectangle (1,.5);
\roundNbox{fill=white}{(-.6,0)}{.25}{0}{0}{\scriptsize$\alpha$};
\draw[thick,fill=lightgray] (.2,0) arc (180:360+180:.3);
\node at (0,0) {$\scriptstyle X$};
}\right)
\Psi^\fX_b\left(\tikzmath{
\fill[fill=lightgray!50,rounded corners=5] (-1/2,-1/2) rectangle (1/2,1/2);
}\right)
\\&=
\frac{1}{D_{\Omega_b}} \sum_{X\in\pi_0\Omega_b} d_X \Psi^\fX_b\left(\tikzmath{
\fill[fill=lightgray!50,rounded corners=5] (-1,-.5) rectangle (0,.5);
\roundNbox{fill=white}{(-.5,0)}{.25}{0}{0}{\scriptsize$\alpha$};
}\right)
\Psi^\fX_b\left(\tikzmath{
\fill[fill=lightgray!50,rounded corners=5] (-1/2,-1/2) rectangle (.6,1/2);
\draw[thick,fill=lightgray] (-.2,0) arc (180:360+180:.25);
\node at (-.35,0) {$\scriptstyle X$};
}\,\right)
\\&=
\frac{1}{D_{\Omega_b}}
\left(
\sum_{X\in\pi_0\Omega_b} d_X^2
\right)
\Psi^\fX_b(\alpha)
\\&=
\Psi^\fX_b(\alpha).
\qedhere
\end{align*}
\end{proof}

\item\label{fact:hom-is-bim}
For $\rmH^*$-multifusion categories $\cC,\cD$,
the $\dag$-equivalence 
$$
\Hom(\Mod^\dag(\cC)\to\Mod^\dag(\cD)) \cong \Bim^\dag(\cC,\cD) \cong \Mod^\dag(\cC^{\mathrm{mp}}\boxtimes \cD)  
$$
given by $F \mapsto {}_\cC F(\cC_\cC)_\cD$ is isometric, i.e., an equivalence of 3-Hilbert spaces.
\begin{proof}
Since this is an equivalence of categories, it is enough to show that this preserves the dimensions of simple objects.
Moreover, it is enough to consider the case where $\cC$ and $\cD$ are both fusion.
Indeed, any $\cC$--$\cD$-bimodule decomposes as a sum of $\cC_{ii}$--$\cD_{jj}$-bimodules, where $\cC_{ii} \coloneq 1_i\otimes\cC\otimes 1_i$ is a fusion corner of $\cC$, so any simple will only be supported at one fusion block on each side. 

Before proceeding, we remind the reader from Example \ref{ex:ModforH*mFC} of the notation $d_c\coloneq\psi_\cC(\tr_\cC(\id_c))$, $D_\cC\coloneq \sum_{c\in\pi_0\cC} d_c^2$, $n_\cC\coloneq\dim(\End_\cC(1_\cC))$, and $\delta_\cC\coloneq \psi_\cC(\id_{1_\cC})$. 

First, since $\Mod^\dag(\cC)$ is connected, $\{\cC_\cC\}$ is an ONB, so using Proposition \ref{prop:FourierExpansion}, we compute
$$
d_F^{\Hom}
=
\Psi^{\Hom}_F(\id_{1_F}) 
\underset{\eqref{eq:3HilbOnFun}}{=}
\frac{d_\cC^{\Mod^\dag(\cC)}}{\dim^{\Mod^\dag(\cC)}\left( \Omega_{\cC_\cC} \right) }\Psi^{\Mod^\dag(\cD)}(\id_{1_{F\cC}}).
$$
To calculate the numerator of the right hand side, we use the formula \eqref{modeeqn} for $\Psi^{\Mod^\dag(\cC)}_{\cC_\cC}$, $
d_\cC^{\Mod^\dag(\cC)}
=
\frac{\delta_\cC}{D_\cC}\sum_{c\in\pi_0\cC} (d^{\cC}_c)^2
=
\delta_\cC
$
and simplify the denominator as follows:
\begin{align*}
\dim^{\Mod^\dag(\cC)}\left( \Omega_{\cC_\cC}\right) &= \sum_{c\in \pi_0\left(\Omega_{\cC_\cC}\cong \cC\right)}\Psi^{\Mod^\dag(\cC)}(\ev_c\circ\ev_c^\dag)^2
\\
&= \sum_{c\in\pi_0\cC }\bigg(\frac{\textcolor{red}{\delta_\cC}}{\textcolor{blue}{D_\cC}}\sum_{c'\in\pi_0\cC}\textcolor{blue}{d^\cC_{c'}}\underbrace{
\Tr^\cC_{c'}\left(\tikzmath{
\draw[thick] (0,0) arc (0:360:.4);
\draw[thick] (-1.1,-.4) -- ++(0,.8);
\node at (-1.3,0) {{$\scriptstyle c'$}};
\node at (.15,0) {$\scriptstyle c$};}\right)
}_{=d^\cC_c\textcolor{blue}{d^\cC_{c'}}/\textcolor{red}{\delta_\cC}}\bigg)^2 \\
&= 
\sum_{c\in\pi_0\cC}\left(d_c^\cC\right)^2\\
&=
D_\cC
\end{align*}
to obtain $\frac{d_\cC^{\Mod^\dag(\cC)}}{\dim^{\Mod^\dag(\cC)}\left( \Omega_{\cC_\cC} \right) } = \frac{\delta_\cC}{D_\cC}$.
By a similar calculation, 
$$
\Psi^{\Mod^\dag(\cD)}(\id_{1_{F\cC}}) = \frac{\delta_\cD}{D_\cD}\sum_{m \in \pi_0F\cC}(d_m^{F\cC})^2,
$$ 
so the dimension of $F$ in $\Hom(\Mod^\dag(\cC) \to \Mod^\dag(\cD))$ is
\[
\frac{\delta_\cC\delta_\cD}{D_\cC D_\cD}\sum_{m\in\pi_0F\cC} \left(d_m^{F\cC}\right)^2.
\]

On the other hand, if we calculate the dimension of $F\cC$ as a $\cC$--$\cD$-bimodule in $\Bim^\dag(\cC,\cD) \cong \Mod^\dag(\cC^{\mathrm{mp}}\boxtimes\cD)$, we get
\[ \frac{\delta_{\cC^{\mathrm{mp}}\boxtimes\cD}}{D_{\cC^{\mathrm{mp}}\boxtimes\cD}}\sum_{m \in \pi_0F\cC}\left(d^{F\cC}_m\right)^2, \]
which is visibly equal to the above.
\end{proof}

\item \label{nattransisisometryiffpointwiseisometry}
A 1-morphism $\eta\colon  F\Rightarrow G$ in $\Hom(\fX\to \fY)$ is an isometry if and only if it is a pointwise isometry.
\begin{proof}
Recall from \cite[Rem.~4.12]{MR5078555} that $\eta$ is an isometry if and only if $\coev_\eta$ is unitary and $\ev_\eta$ is an isometry.
Since the $\dagger$-structure on $\Hom(\fX\to \fY)$ is given pointwise, this is true if and only if $(\coev_\eta)_x=\coev_{\eta_x}$ is unitary and $(\ev_\eta)_x=\ev_{\eta_x}$ is an isometry at each object $x \in \fX$.
\end{proof}

\item \label{isometriesin4Hilb} 
A functor $F\in\Hom(\fX\to\fY)$ is an isometry (isometric equivalence) of 3-Hilbert spaces if and only if
$\mathrm{coev}_F\colon \id_{\fX}\Rightarrow F^*F$\footnote{For categories of functors, the conventions for evaluation and coevaluation are swapped as composition is read from right-to-left.} is an isometric equivalence in $\End(\fX)$ and $\mathrm{coev}_F$ is an isometry (isometric equivalence) in $\End(\fY)$. 
\begin{proof}
Note that $F$ is an isometry of 3-Hilbert spaces if and only if the maps $\fX(x_1 \to x_2) \to \fY(Fx_1 \to Fx_2)$ are isometric equivalences of 2-Hilbert spaces, since $F$ being an isometry means it is $\Psi$-preserving and fully faithful.
By unitary adjunction, the composite isometric equivalence $$
\fX(x_1 \to x_2) \cong^\dag \fY(Fx_1 \to Fx_2) \cong^\dag \fX(x_1 \to F^*Fx_2)
$$ 
is given by composition with $\coev_F$.
This is true if and only if each component of $\coev_F$ is a pointwise isometric equivalence, which by \ref{nattransisisometryiffpointwiseisometry} means $\coev_F$ is an isometric equivalence.
If $F$ is an isometric equivalence, this means that $F^*$ is also an isometry, and thus implies that $\ev_F$ is also an isometric equivalence.
\end{proof}

    \item 
    \label{D:YonedaEmbeddingUAF}
    Unitary 2-adjunctions induce unitary 2-adjunctions between functor 2-categories.
That is, if $F\dashv^\dag F^*$ for $F\colon  \fX\to \fY$, then for all $\fW$, 
$F\circ - \colon  \Hom(\fW\to \fX) \to \Hom(\fW\to \fY)$ 
has unitary 2-adjoint $F^*\circ -\colon  \Hom(\fW\to \fY)\to \Hom(\fW\to \fX)$.
\begin{proof}
We wish to show that for each functor $G:\fW \to \fX$ and $H\colon\fW \to \fY$, the equivalences $\Hom_{\Hom(\fW\to\fY)}(F\circ G,H) \cong^\dag \Hom_{\Hom(\fW,\fX)}(G,F^*\circ H)$ given by $\ev_F$ and $\coev_F$ are isometric.
Given $\alpha\colon F\circ G\Rightarrow H$ and $m\colon\alpha \Rrightarrow \alpha$, we want $\Psi^{\Hom(\fW\to\fY)}(\Tr(m)) = \Psi^{\Hom(\fW\to\fX)}(\Tr(\mate(m)))$.
Note that by the definition of $\Psi^{\Hom}$, it suffices to show that $\Psi^\fY_{Hw}(\Tr(m_w)) = \Psi^{\fX}_{F^* Hw}(\Tr(\mate(m)_w))$.
Since $\mate(m)_w$ is $\mate(m_w)$, and since $F \dashv^\dag F^*$, this is immediate.
\end{proof}
\end{enumerate}

\subsection{The symmetric monoidal structure of \texorpdfstring{$3\Hilb$}{3Hilb}}\label{sec:symm-monoidal-struct}
The final desideratum for our 3-Hilbert space structure on $\Hom(\fX\to \fY)$ is that it satisfies hom-tensor adjunction for the unitary Deligne product on $3\Hilb$, which corresponds to the usual unitary Deligne product on $\rmH^*$-multifusion categories under the equivalence $3\Hilb \cong \mathsf{H^*mFC}$.

    \begin{defn}
    Let $\fX,\fY$ be two 3-Hilbert spaces.
    We define a 2-category with objects formal pairs $c\threehilbtimes d$ of objects in $\fX$ and $\fY$, Hom categories $\Hom(c\threehilbtimes d,c'\threehilbtimes d') \coloneq \Hom(c,c')\boxtimes\Hom(d,d')$, and composition given componentwise.
    This category has a spherical state $\Psi$ given by the Deligne product of 2-Hilbert spaces.
    Explicitly, on a generating 2-morphism $\alpha\threehilbtimes \beta \colon \id_{c\threehilbtimes d} \Rightarrow \id_{c\threehilbtimes d}$, we have $\Psi(\alpha\threehilbtimes \beta) = \Psi^\fX(\alpha)\Psi^\fX(\beta)$.
    The \emph{unitary Deligne product} of $\fX$ and $\fY$, denoted $\fX\threehilbtimes \fY$, is defined to be the completion of the 2-category above. This is the unitary version of the Deligne tensor product described in \cite{MR4677193}.
    \end{defn}
    
    \begin{ex}
    \label{ex:3HilbSymmetricMonoidal}
    If $\cC,\cD$ are $\rmH^*$-multifusion categories, 
    $\Mod^\dag(\cC) \threehilbtimes \Mod^\dag(\cD) \isomeq \Mod^\dag(\cC\boxtimes\cD)$.
    Note that there is a clear map from the formal 2-category of pairs $\cM_\cC\threehilbtimes\cN_\cD$ to $\Mod^\dag(\cC\boxtimes\cD)$ by $\cM_\cC\threehilbtimes\cN_\cD \mapsto (\cM\boxtimes\cN)_{\cC\boxtimes \cD}$.
    This map exhibits $\Mod^\dag(\cC\boxtimes\cD)$ as the completion, since it is (in particular) the completion of the full subcategory on the generating module $(\cC\boxtimes\cD)_{\cC\boxtimes\cD}$.
    \end{ex}

    \begin{prop}
    The unitary Deligne product of 3-Hilbert spaces satisfies an isometric version of hom-tensor adjunction.
    That is, for any 3-Hilbert spaces $\fX,\fY,\fZ$, the natural equivalence of 2-categories $\Hom(\fX\to\Hom(\fY{\to}\fZ)) \cong \Hom(\fX\threehilbtimes\fY\to\fZ)$ is an isometry.
    \end{prop}
    \begin{proof}
    Let $F\colon\fX \to \Hom(\fY\to\fZ)$.
    Then we have
    \begin{align*}
    \Psi(\id_{\id_F}) &= \sum_{x \in \pi_0\fX} \frac{d_x}{D_{\Omega_x}}\Psi^{\Hom(\fY{\to}\fZ)}_{Fx}(\id_{\id_{Fx}}) \\
    &= \sum_{\substack{x\in\pi_0\fX, \\ \!y\in\pi_0\fY}} \frac{d_xd_y}{D_{\Omega_x}D_{\Omega_y}}\Psi^{\fZ}_{Fxy}(\id_{\id_{Fxy}}).
    \end{align*}
    For the other side, we claim that objects in the form $x\threehilbtimes y$ where $x,y$ are a basis for $\fX$ and $\fY$, respectively, form a basis for $\fX\threehilbtimes \fY$.
    Note that $\Omega_{x\threehilbtimes y} = \Omega_x\boxtimes \Omega_y$ is a fusion category, so $x\threehilbtimes y$ is simple.
    And since $\fX \threehilbtimes \fY$ is the completion of objects in this form, this does indeed form a basis.
    Therefore, letting $F'$ denote the image of $F$ under the hom-tensor equivalence, we have
    \begin{align*}
    \Psi(\id_{\id_{F'}}) &= \sum_{x \in \pi_0\fX,y \in \pi_0\fY} \frac{d_{x\threehilbtimes y}}{D_{\Omega_{x\threehilbtimes y}}} \Psi^{\fZ}_{F'(x\threehilbtimes y)}(\id_{\id_{F'(x\threehilbtimes y)}}) \\
    &= \sum_{x \in \pi_0\fX,y\in\pi_0\fY} \frac{d_xd_y}{D_{\Omega_x}D_{\Omega_y}}\Psi^{\fZ}_{Fxy}(\id_{\id_{Fxy}}).
    \end{align*}
This verifies that the equivalence is $\Psi$-preserving; since it is clearly an equivalence of 2-categories, this means it is isometric.
    \end{proof}

We can now have the analogous result from Example \ref{ex:HomIsTensor} in $3\Hilb$.

\begin{cor}
There is an isometric equivalence $\Hom(\fX \to \fY) \cong^\dag \fY\threehilbtimes \fX^{1\op}$.
\end{cor}
\begin{proof}
Write $\fX \cong^\dag \Mod^\dag(\cC)$ and $\fY \cong^\dag \Mod^\dag(\cD)$ for some $\rmH^*$-multifusion categories $\cC,\cD$.
Then
\begin{align*}
\Hom(\Mod^\dag(\cC)\to \Mod^\dag(\cD))
&\cong^\dag
\Mod^\dag(\cD\boxtimes \cC^{\rm mp})
&&\text{\ref{fact:hom-is-bim}}
\\&
\cong^\dag
\Mod^\dag(\cD)\threehilbtimes \Mod^\dag(\cC^{\rm mp})
&&\text{(Ex.~\ref{ex:3HilbSymmetricMonoidal})}
\\&
\cong^\dag
\Mod^\dag(\cD)\threehilbtimes \Mod^\dag(\cC)^{1\op},
&&
\end{align*}
where the last isometric equivalence follows from
\[
\Mod^\dag(\cC^{\rm mp}) 
\cong^\dag 
\rmB(\cC^{\rm mp})^\cent 
\cong^\dag
((\rmB\cC)^{1\op})^{\cent } 
\cong^\dag
\Mod^\dag(\cC)^{1 \op}.
\qedhere
\]
\end{proof}

\bibliographystyle{alpha}
{\footnotesize{
\bibliography{../../bibliography/bibliography}

\newcommand{\etalchar}[1]{$^{#1}$}
\begin{thebibliography}{FHJF{\etalchar{+}}24}

\bibitem[Amb45]{MR13235}
Warren Ambrose.
\newblock Structure theorems for a special class of {B}anach algebras.
\newblock {\em Trans. Amer. Math. Soc.}, 57:364--386, 1945.
\newblock \mathscinet{MR13235} \doi{10.2307/1990182}.

\bibitem[Bae97]{MR1448713}
John~C. Baez.
\newblock Higher-dimensional algebra. {II}. {$2$}-{H}ilbert spaces.
\newblock {\em Adv. Math.}, 127(2):125--189, 1997.
\newblock \mathscinet{MR1448713} \doi{10.1006/aima.1997.1617}.

\bibitem[BD95]{MR1355899}
John~C. Baez and James Dolan.
\newblock Higher-dimensional algebra and topological quantum field theory.
\newblock {\em J. Math. Phys.}, 36(11):6073--6105, 1995.
\newblock \mathscinet{MR1355899} \arXiv{q-alg/9503002} \doi{10.1063/1.531236}.

\bibitem[CFH{\etalchar{+}}26]{MR5078555}
Quan Chen, Giovanni Ferrer, Brett Hungar, David Penneys, and Sean Sanford.
\newblock Manifestly unitary higher {H}ilbert spaces.
\newblock {\em J. Lond. Math. Soc. (2)}, 113(5):Paper No. e70532, 2026.
\newblock \mathscinet{MR5078555} \doi{10.1112/jlms.70532} \arxiv{2410.05120}.

\bibitem[CHPJP22]{MR4419534}
Quan Chen, Roberto Hern\'{a}ndez~Palomares, Corey Jones, and David Penneys.
\newblock Q-system completion for {$\rm C^*$} 2-categories.
\newblock {\em J. Funct. Anal.}, 283(3):Paper No. 109524, 2022.
\newblock \mathscinet{MR4419534} \doi{10.1016/j.jfa.2022.109524}
  \arxiv{2105.12010}.

\bibitem[CP22]{MR4369356}
Quan Chen and David Penneys.
\newblock Q-system completion is a 3-functor.
\newblock {\em Theory Appl. Categ.}, 38:Paper No. 4, 101--134, 2022.
\newblock \mathscinet{MR4369356} \doi{10.1002/num.22828} \arxiv{2106.12437}.

\bibitem[D{\'e}c23a]{MR4600461}
Thibault~D. D{\'e}coppet.
\newblock The {M}orita theory of fusion 2-categories.
\newblock {\em High. Struct.}, 7(1):234--292, 2023.
\newblock \mathscinet{MR4600461} \doi{10.21136/HS.2023.07} \arxiv{2208.08722}.

\bibitem[D{\'{e}}c23b]{2311.16827}
Thibault~D. D{\'{e}}coppet.
\newblock On the dualizability of fusion 2-categories.
\newblock \arxiv{2311.16827}, to appear Quantum Topol., 2023.

\bibitem[D{\'{e}}c24]{MR4677193}
Thibault~D. D{\'{e}}coppet.
\newblock The 2-{D}eligne tensor product.
\newblock {\em Kyoto J. Math.}, 64(1):1--29, 2024.
\newblock \mathscinet{MR4677193} \doi{10.1215/21562261-2023-0005}
  \arxiv{2103.16880}.

\bibitem[DHP24]{MR4814691}
Zachary Dell, Peter Huston, and David Penneys.
\newblock Unitary braided-enriched monoidal categories.
\newblock {\em Quantum Topol.}, 15(3):567--632, 2024.
\newblock \mathscinet{MR4814691} \doi{10.4171/qt/226} \arxiv{2208.14992}.

\bibitem[DR18]{1812.11933}
Christopher~L. Douglas and David~J. Reutter.
\newblock Fusion 2-categories and a state-sum invariant for 4-manifolds, 2018.
\newblock \arxiv{1812.11933}.

\bibitem[EJ25]{2507.05185}
David~E. Evans and Corey Jones.
\newblock An operator algebraic approach to fusion category symmetry on the
  lattice.
\newblock \arxiv{2507.05185}, 2025.

\bibitem[Fer24]{2404.05193}
Giovanni Ferrer.
\newblock Foundations for operator algebraic tricategories.
\newblock \arxiv{2404.05193}, 2024.

\bibitem[FH21]{MR4268163}
Daniel~S. Freed and Michael~J. Hopkins.
\newblock Reflection positivity and invertible topological phases.
\newblock {\em Geom. Topol.}, 25(3):1165--1330, 2021.
\newblock \mathscinet{MR4268163} \doi{10.2140/gt.2021.25.1165}
  \arxiv{1604.06527}.

\bibitem[FHJF{\etalchar{+}}24]{2403.01651}
Giovanni Ferrer, Brett Hungar, Theo Johnson-Freyd, Cameron Krulewski, Lukas
  M\"uller, Nivedita, David Penneys, David Reutter, Claudia Scheimbauer, Luuk
  Stehouwer, and Chetan Vuppulury.
\newblock Dagger $n$-categories.
\newblock \arxiv{2403.01651}, 2024.

\bibitem[FKP]{UQSLbook}
Giovanni Ferrer, Kyle Kawagoe, and David Penneys.
\newblock Unitary quantum symmetries lite.
\newblock In preparation. Some chapters available at
  \url{https://people.math.osu.edu/penneys.2/UQSL/UQSL.html}.

\bibitem[FLW02]{MR1910833}
Michael~H. Freedman, Michael Larsen, and Zhenghan Wang.
\newblock A modular functor which is universal for quantum computation.
\newblock {\em Comm. Math. Phys.}, 227(3):605--622, 2002.
\newblock \mathscinet{MR1910833} \doi{10.1007/s002200200645}
  \arxiv{quant-ph/0001108v2}.

\bibitem[FMPS26]{2606.11334}
Giovanni Ferrer, Lukas M\"uller, David Penneys, and Luuk Stehouwer.
\newblock The many faces of higher hilbert spaces.
\newblock \arxiv{2606.11334}, 2026.

\bibitem[FMT24]{MR4814695}
Daniel~S. Freed, Gregory~W. Moore, and Constantin Teleman.
\newblock Topological symmetry in quantum field theory.
\newblock {\em Quantum Topol.}, 15(3):779--869, 2024.
\newblock \mathscinet{MR4814695} \doi{10.4171/qt/223} \arxiv{2209.07471}.

\bibitem[GJF19]{1905.09566}
Davide Gaiotto and Theo Johnson-Freyd.
\newblock Condensations in higher categories, 2019.
\newblock \arxiv{1905.09566}.

\bibitem[GMP{\etalchar{+}}23]{MR4598730}
Pinhas Grossman, Scott Morrison, David Penneys, Emily Peters, and Noah Snyder.
\newblock The {E}xtended {H}aagerup fusion categories.
\newblock {\em Ann. Sci. \'{E}c. Norm. Sup\'{e}r. (4)}, 56(2):589--664, 2023.
\newblock \mathscinet{MR4598730} \doi{10.24033/asens.2541} \arxiv{1810.06076}.

\bibitem[Gra74]{MR371990}
John~W. Gray.
\newblock {\em Formal category theory: adjointness for {$2$}-categories}.
\newblock Lecture Notes in Mathematics, Vol. 391. Springer-Verlag, Berlin-New
  York, 1974.
\newblock \mathscinet{MR371990} \doi{doi:10.1007/BFb0061280}.

\bibitem[Gur13]{MR3076451}
Nick Gurski.
\newblock {\em Coherence in three-dimensional category theory}, volume 201 of
  {\em Cambridge Tracts in Mathematics}.
\newblock Cambridge University Press, Cambridge, 2013.
\newblock \mathscinet{MR3076451} \doi{10.1017/CBO9781139542333}.

\bibitem[HNP24]{2411.01678}
Andr\'e Henriques, Nivedita, and David Penneys.
\newblock Complete $\mathrm{W}^*$-categories.
\newblock \arxiv{2411.01678}, 2024.

\bibitem[HPT16]{MR3578212}
Andr\'e Henriques, David Penneys, and James Tener.
\newblock Categorified trace for module tensor categories over braided tensor
  categories.
\newblock {\em Doc. Math.}, 21:1089--1149, 2016.
\newblock \mathscinet{MR3578212} \doi{10.4171/DM/553} \arxiv{1509.02937}.

\bibitem[HPT24]{MR4750417}
Andr\'{e} Henriques, David Penneys, and James Tener.
\newblock Unitary {A}nchored {P}lanar {A}lgebras.
\newblock {\em Comm. Math. Phys.}, 405(6):Paper No. 137, 2024.
\newblock \mathscinet{MR4750417} \doi{10.1007/s00220-024-04985-w}
  \arxiv{2301.11114}.

\bibitem[HV19]{MR3971584}
Chris Heunen and Jamie Vicary.
\newblock {\em Categories for quantum theory}, volume~28 of {\em Oxford
  Graduate Texts in Mathematics}.
\newblock Oxford University Press, Oxford, 2019.
\newblock An introduction, \mathscinet{MR3971584}
  \doi{10.1093/oso/9780198739623.001.0001}.

\bibitem[JF22]{MR4444089}
Theo Johnson-Freyd.
\newblock On the classification of topological orders.
\newblock {\em Comm. Math. Phys.}, 393(2):989--1033, 2022.
\newblock \mathscinet{MR4444089} \doi{10.1007/s00220-022-04380-3}
  \arxiv{2003.06663}.

\bibitem[JNPW25]{MR4945955}
Corey Jones, Pieter Naaijkens, David Penneys, and Daniel Wallick.
\newblock Local topological order and boundary algebras.
\newblock {\em Forum Math. Sigma}, 13:Paper No. e135, 2025.
\newblock \mathscinet{MR4945955} \doi{10.1017/fms.2025.16} \arxiv{2307.12552}.

\bibitem[Jon24]{MR4814692}
Corey Jones.
\newblock D{HR} bimodules of quasi-local algebras and symmetric quantum
  cellular automata.
\newblock {\em Quantum Topol.}, 15(3):633--686, 2024.
\newblock \mathscinet{MR4814692} \doi{10.4171/qt/216} \arxiv{2304.00068}.

\bibitem[JP17]{MR3687214}
Corey Jones and David Penneys.
\newblock Operator algebras in rigid {$\rm C^*$}-tensor categories.
\newblock {\em Comm. Math. Phys.}, 355(3):1121--1188, 2017.
\newblock \mathscinet{MR3687214} \doi{10.1007/s00220-017-2964-0}
  \arxiv{1611.04620}.

\bibitem[JP19]{MR3948170}
Corey Jones and David Penneys.
\newblock Realizations of algebra objects and discrete subfactors.
\newblock {\em Adv. Math.}, 350:588--661, 2019.
\newblock \mathscinet{MR3948170} \doi{10.1016/j.aim.2019.04.039}
  \arxiv{1704.02035}.

\bibitem[JP20]{MR4079745}
Corey Jones and David Penneys.
\newblock Q-systems and compact {W}*-algebra objects.
\newblock In {\em Topological phases of matter and quantum computation}, volume
  747 of {\em Contemp. Math.}, pages 63--88. Amer. Math. Soc., Providence, RI,
  2020.
\newblock \mathscinet{MR4079745} \doi{10.1090/conm/747/15039}
  \arxiv{1707.02155}.

\bibitem[Kel05]{MR2177301}
G.~M. Kelly.
\newblock Basic concepts of enriched category theory.
\newblock {\em Repr. Theory Appl. Categ.}, (10):vi+137, 2005.
\newblock \mathscinet{MR2177301}, Reprint of the 1982 original [Cambridge Univ.
  Press, Cambridge; \mathscinet{MR0651714}].

\bibitem[Kit03]{MR1951039}
A.~Yu. Kitaev.
\newblock Fault-tolerant quantum computation by anyons.
\newblock {\em Ann. Physics}, 303(1):2--30, 2003.
\newblock \mathscinet{MR1951039} \doi{10.1016/S0003-4916(02)00018-0}
  \arxiv{quant-ph/9707021}.

\bibitem[Kit06]{cond-mat/0506438}
Alexei Kitaev.
\newblock Anyons in an exactly solved model and beyond.
\newblock {\em Annals of Physics}, 321(1):2--111, 2006.
\newblock January Special Issue. \doi{10.1016/j.aop.2005.10.005}
  \arxiv{cond-mat/0506438}.

\bibitem[KWZ17]{1702.00673}
Liang Kong, Xiao-Gang Wen, and Hao Zheng.
\newblock Boundary-bulk relation in topological orders.
\newblock {\em Nuclear Physics B}, 922:62--76, 2017.
\newblock \doi{10.1016/j.nuclphysb.2017.06.023} \arxiv{1702.00673}.

\bibitem[LW05]{PhysRevB.71.045110}
Michael~A. Levin and Xiao-Gang Wen.
\newblock String-net condensation: A physical mechanism for topological phases.
\newblock {\em Phys. Rev. B}, 71:045110, Jan 2005.
\newblock doi{10.1103/PhysRevB.71.045110} \arxiv{cond-mat/0404617}.

\bibitem[Naa12]{NaaijkensThesis}
Pieter Naaijkens.
\newblock {\em Anyons in infinite quantum systems : {QFT} in d=2+1 and the
  {T}oric {C}ode}.
\newblock PhD thesis, Radboud Universiteit Nijmegen, May 2012.
\newblock Available at \url{https://repository.ubn.ru.nl/handle/2066/92737}.

\bibitem[Oga22]{MR4362722}
Yoshiko Ogata.
\newblock A derivation of braided {$C^*$}-tensor categories from gapped ground
  states satisfying the approximate {H}aag duality.
\newblock {\em J. Math. Phys.}, 63(1):Paper No. 011902, 48, 2022.
\newblock \mathscinet{MR4362722}, \doi{10.1063/5.0061785}, \arxiv{2106.15741}.

\bibitem[Pen20]{MR4133163}
David Penneys.
\newblock Unitary dual functors for unitary multitensor categories.
\newblock {\em High. Struct.}, 4(2):22--56, 2020.
\newblock \mathscinet{MR4133163} \arxiv{1808.00323}.

\bibitem[Sch13]{MR3019263}
Gregor Schaumann.
\newblock Traces on module categories over fusion categories.
\newblock {\em J. Algebra}, 379:382--425, 2013.
\newblock \mathscinet{MR3019263} \doi{10.1016/j.jalgebra.2013.01.013}
  \arxiv{1206.5716}.

\bibitem[Wal21]{2104.02101}
Kevin Walker.
\newblock A universal state sum.
\newblock \arxiv{2104.02101}, 2021.

\bibitem[Wes]{UnitaryDiskLike}
Greyson Wesley.
\newblock Unitary disk-like categories, unitary tqfts, and higher hilbert
  spaces.
\newblock In preparation.

\bibitem[Yam07]{MR2325696}
Shigeru Yamagami.
\newblock Notes on operator categories.
\newblock {\em J. Math. Soc. Japan}, 59(2):541--555, 2007.
\newblock \mathscinet{MR2325696} \arxiv{math/0212136}.

\end{thebibliography}
}}
\end{document}